\documentclass[12pt, reqno]{amsart}

\usepackage{amsmath, amsthm, amssymb}
\usepackage{enumerate}
\usepackage[margin=1.0in]{geometry}
\usepackage{xcolor}
\definecolor{cite}{rgb}{0.30,0.60,1.00}
\definecolor{url}{rgb}{0.00,0.00,0.80}
\definecolor{link}{rgb}{0.40,0.10,0.20}

\usepackage{zref-clever}

\AddToHook{env/theorem/begin}{%
	\zcsetup{countertype={theorem=theorem}}}
\AddToHook{env/claim/begin}{%
	\zcsetup{countertype={theorem=theorem}}}  
\AddToHook{env/proposition/begin}{%
	\zcsetup{countertype={theorem=proposition}}}
\AddToHook{env/lemma/begin}{%
	\zcsetup{countertype={theorem=lemma}}}
\AddToHook{env/conjecture/begin}{%
	\zcsetup{countertype={theorem=theorem}}}  
\AddToHook{env/corollary/begin}{%
	\zcsetup{countertype={theorem=corollary}}}
\AddToHook{env/definition/begin}{%
	\zcsetup{countertype={theorem=definition}}}
\AddToHook{env/remark/begin}{%
	\zcsetup{countertype={theorem=remark}}}
\AddToHook{env/example/begin}{%
	\zcsetup{countertype={theorem=example}}}

\zcsetup{nameinlink=false}

\usepackage[pdfusetitle,colorlinks,linkcolor=link,urlcolor=url,citecolor=cite,breaklinks]{hyperref}
\usepackage{graphicx}
\usepackage{subcaption}
\usepackage{mathdots}
\usepackage{tikz-cd}
\usepackage{comment}
\usepackage{xypic}
\usepackage{mathtools}
\usepackage{physics}
\usepackage{float}

\newtheorem{theorem}{Theorem}[section]

\newtheorem{proposition}[theorem]{Proposition}
\newtheorem{lemma}[theorem]{Lemma}

\newtheorem{corollary}[theorem]{Corollary}

\theoremstyle{definition}

\theoremstyle{definition}
\newtheorem{remark}[theorem]{Remark}
\theoremstyle{definition}
\newtheorem{example}[theorem]{Example}

\newcommand{\cref}[1]{\zcref{#1}}
\newcommand{\Cref}[1]{\zcref[S]{#1}}

\newcommand{\cComplex}{\mathbb{C}}
\newcommand{\multiplicativegroup}[1]{#1^{\times}}

\newcommand{\Hom}{\mathrm{Hom}}

\newcommand{\Span}{\mathrm{span}}

\newcommand{\idmap}{\mathrm{id}}
\newcommand{\conjugate}[1]{\overline{#1}}
\newcommand{\indicatorFunction}[1]{\delta_{#1}}

\renewcommand{\abs}[1]{\left|#1\right|}
\newcommand{\sizeof}[1]{\left|#1\right|}

\renewcommand{\innerproduct}[2]{\left(#1,#2\right)}

\newcommand{\standardForm}[2]{\left\langle #1,#2\right\rangle}

\newcommand{\fieldCharacter}{\psi}
\newcommand{\centralCharacter}[1]{\omega_{#1}}
\newcommand{\Ind}[3]{\mathrm{Ind}_{#1}^{#2}\left(#3\right)}

\newcommand{\Whittaker}{\mathcal{W}}
\newcommand{\Contragradient}[1]{#1^{\vee}}

\newcommand{\besselFunction}{\mathcal{J}}

\newcommand{\SpehRepresentation}[2]{\Delta\left(#1, #2\right)}

\newcommand{\gbesselSpehFunction}[2]{\mathcal{BS}_{#1, #2}}
\newcommand{\besselSpehFunction}[2]{\mathcal{BS}_{\SpehRepresentation{#1}{#2}, \fieldCharacter}}
\newcommand{\specialBesselSpeh}[2][\fieldCharacter]{\mathcal{K}_{{#2}, {#1}}}
\newcommand{\fourierTransform}[2]{\mathcal{F}_{#1}#2}
\newcommand{\GKGammaFactor}[3]{\gamma^{\mathrm{GK}}\left(#1 \times #2, #3\right)}
\newcommand{\ShGammaFactor}[3]{\Gamma^{\mathrm{Sh}}\left(#1 \times #2, #3\right)}

\newcommand{\GKPreGammaFactor}[3]{\Gamma^{\mathrm{GK}}\left(#1 \times #2, #3\right)}
\newcommand{\gGJPreGammaFactor}[3]{\Gamma^{\mathrm{GJ}}\left(#1 \times #2, #3\right)}
\newcommand{\GJPreGammaFactor}[2]{\Gamma^{\mathrm{GJ}}\left(#1, #2\right)}
\newcommand{\genericRepGammaFactor}[3]{\Gamma^{\mathrm{gen}}\left(#1 \times #2, #3\right)}

\newcommand{\transpose}[1]{\, {}^{t}#1}
\newcommand{\inverseTranspose}[1]{#1^{\iota}}
\newcommand{\IdentityMatrix}[1]{I_{#1}}
\newcommand{\diag}{\mathrm{diag}}

\renewcommand{\trace}{\operatorname{tr}}
\newcommand{\GL}{\mathrm{GL}}
\newcommand{\UnipotentSubgroup}{U}
\newcommand{\UnipotentRadicalForWss}[2]{N_{\qty(#2^{#1})}}
\newcommand{\UnipotentRadicalForWssRecursion}[2]{\mathcal{Y}_{#2,#1}}
\newcommand{\UnipotentRadical}{N}
\newcommand{\UnipotentRadicalDeleted}[2]{\UnipotentRadicalForWss{#1}{#2}^\circ}
\newcommand{\ParabolicSubgroup}{P}

\makeatletter
\newcommand{\ostar}{\!{\mathpalette\make@circled *}\!}
\newcommand{\make@circled}[2]{%
	\ooalign{$\m@th#1\smallbigcirc{#1}$\cr\hidewidth$\m@th#1#2$\hidewidth\cr}%
} 
\newcommand{\smallbigcirc}[1]{%
	\vcenter{\hbox{\scalebox{0.77778}{$\m@th#1\bigcirc$}}}%
}
\makeatother

\newcommand{\Erdelyi}{Erd{\'e}lyi}
\newcommand{\Toth}{T{\'o}th}

\newcommand{\finiteField}{\mathbb{F}}

\newcommand{\algebraicClosure}[1]{\overline{#1}}

\newcommand{\squareMatrix}{\operatorname{Mat}}
\newcommand{\Mat}[2]{\operatorname{Mat}_{#1 \times #2}}

\newcommand{\Steinberg}{\operatorname{St}}
\newcommand{\ProjectionOperator}{\operatorname{pr}}
\newcommand{\SymmetricGroup}{\mathfrak{S}}
\newcommand{\whittakerVector}[1]{v_{#1, \fieldCharacter}}

\newcommand{\WhittakerProjection}{\ProjectionOperator_{\mathrm{Wh}}}
\newcommand{\ParabolicForSpeh}[2]{P_{\left({#1}^{#2}\right)}}
\newcommand{\UnipotentForSpeh}[2]{N_{\left({#1}^{#2}\right)}}
\newcommand{\PoincarePolynomial}[2]{P_{#2}}

\newcommand{\localField}{F}

\newcommand{\WhittakerFunctional}[1]{\ell_{#1, \fieldCharacter}}

\newcommand{\gSpehWhittakerFunctional}[3]{\ell_{\SpehRepresentation{#1}{#3}, \fieldCharacterkc{#2}{#3}}}
\newcommand{\gShortSpehWhittakerFunctional}[3]{\ell_{\SpehRepresentation{#1}{#3}}}
\newcommand{\GaussSum}[2]{\mathcal{G}\left(#1, #2\right)}
\newcommand{\GKGaussSum}[3]{\mathcal{G}\left(#1 \times #2, #3\right)}

\newcommand{\kcNotation}[3]{(#1,#2)_{#3}}
\newcommand{\fieldCharacterkc}[2]{\fieldCharacter_{\qty(#2^{#1})}}
\newcommand{\zelevinskyCharacter}[1]{\fieldCharacter^{\mathrm{Zel}}_{#1}}
\newcommand{\ExoticKloosterman}{\mathrm{Kl}}

\newcommand{\zetaOperator}{\mathrm{Z}}
\newcommand{\dualZetaOperator}{\mathrm{Z}^{\ast}}

\hypersetup{pdfauthor={Oded Carmon, Elad Zelingher},
	pdfsubject={Representation theory},
	pdfkeywords={Bessel functions, Kloosterman sums, Gauss sums}}

\title[Ginzburg--Kaplan gamma factors for $\GL_n$]{On Ginzburg--Kaplan gamma factors and Bessel--Speh functions for finite general linear groups}

\author{Oded Carmon}
\address{The Weizmann Institute of Science, Rehovot, Israel}
\email{oded.carmon@weizmann.ac.il}
\author{Elad Zelingher}
\address{Department of Mathematics, University of Michigan, 1844 East Hall, 530 Church Street, Ann Arbor, MI 48109-1043 USA}
\email{eladz@umich.edu}


\begin{document}

\begin{abstract}
	We give a new construction of tensor product gamma factors for a pair of irreducible representations of $\GL_c\left(\finiteField_q\right)$ and $\GL_k\left(\finiteField_q\right)$. This construction is a finite field analog of a construction of doubling type due to Kaplan in the local field case~\cite[Appendix A]{kaplan2018} and due to Ginzburg in the global case~\cite{ginzburg2019tensor}, and it only assumes that one of the representations in question is generic. We use this construction to establish a relation between special values of Bessel functions attached to Speh representations of generic principal series representations and twisted matrix Kloosterman sums. Using this relation, we establish the multiplicativity identity of twisted matrix Kloosterman sums.
\end{abstract}

\maketitle

\section{Introduction}\label{sec:introduction}

In the representation theory of $p$-adic groups, a useful way of studying irreducible representations is by attaching invariants to them. Among the invariants one may attach to an irreducible representation, the most prominent ones are \emph{local $L$-factors}. Local $L$-factors are important for the definition of global $L$-functions of automorphic representations. Poles of local $L$-factors (and of their counterpart global $L$-functions) often encode information and provide insights about the representations in question. A good theory of local $L$-factors often goes through a corresponding theory of other invariants called \emph{local gamma factors}. Local gamma factors often appear as proportionality factors for certain functional equations. In some cases~\cite{shahidi1984fourier, shahidi1990proof, CaiFriedbergKaplan2022}, one starts with a definition of local gamma factors, and uses them in order to define the corresponding local $L$-factors.

Given a theory of local gamma factors for irreducible representations of a reductive $p$-adic group $G$, one can try to study its finite field version, that is, the analogous theory of gamma factors for irreducible representations of the $\finiteField$-points of $G$, where $\finiteField$ is a finite field. This can be seen as a ``toy model'' of the gamma factor theory of the $p$-adic group case.

A theory of gamma factors for finite fields is usually related to its counterpart theory for $p$-adic groups via level zero representations~\cite{Ye18, YeZeligher18}. However, these factors are interesting by themselves. Let us give some examples in a special case. For irreducible generic representations $\pi$ and $\sigma$ of $\GL_n\left(\finiteField\right)$ and $\GL_m\left(\finiteField\right)$, respectively, let $\genericRepGammaFactor{\pi}{\sigma}{\fieldCharacter}$ be the tensor product gamma factor of $\GL_n\left(\finiteField\right) \times \GL_m\left(\finiteField\right)$ attached to $\pi \otimes \sigma$ \cite{Roditty10, Nien14, SoudryZelingher2023}, with respect to a fixed non-trivial character $\fieldCharacter \colon \finiteField \to \multiplicativegroup{\cComplex}$. Then
\begin{itemize}
	\item One can find the number of times that an irreducible cuspidal representation $\sigma$ appears in the cuspidal support of $\pi$ by computing the absolute value $\abs{\genericRepGammaFactor{\pi}{\sigma^{\vee}}{\fieldCharacter}}$ (see~\cite[Theorem 1.3]{SoudryZelingher2023}).
	\item The converse theorem~\cite{Nien14, SoudryZelingher2023} implies that $\pi$ is determined by its central character and the values of gamma factors $\left(\genericRepGammaFactor{\pi}{\sigma}{\fieldCharacter}\right)_{\sigma}$ where $\sigma$ ranges over all irreducible cuspidal representations of $\GL_m\left(\finiteField\right)$ for every $1 \le m \le \frac{n}{2}$.
	\item Values of the Bessel function of $\pi$ can be computed recursively~\cite{Zelingher2023} using gamma factors of the form $\genericRepGammaFactor{\pi}{\sigma}{\fieldCharacter}$. This allows one to give a realization of $\pi$ \cite{AlperinJames1995}.
\end{itemize}

In this work, we study a finite field analog of a new construction for the tensor product gamma factor by Kaplan~\cite[Appendix A]{kaplan2018} for $\GL_c\left(\finiteField\right) \times \GL_k\left(\finiteField\right)$. This construction can be seen as a generalization of two classical families of constructions of gamma factors for $\GL_n\left(\finiteField\right) \times \GL_1\left(\finiteField\right)$. Let us review these two families. For the next section, let $\pi$ be an irreducible representation of $\GL_n\left(\finiteField\right)$ and let $\chi \colon \multiplicativegroup{\finiteField} \to \multiplicativegroup{\cComplex}$ be a character. 

\subsection{Classical constructions of gamma factors}

\subsubsection{First family: construction via matrix coefficients}
The first family consists of a construction by Kondo~\cite{Kondo1963} and Godement--Jacquet~\cite{GodementJacquet1972}. This construction works for any irreducible representation $\pi$. It considers the operator $$\GKGaussSum{\pi}{\chi}{\fieldCharacter} = q^{-\frac{n^2}{2}} \sum_{g \in \GL_n\left(\finiteField\right)} \chi\left(\det g\right) \fieldCharacter\left(\trace g^{-1}\right) \pi\left(g\right).$$
Since the assignment $g \mapsto \chi\left(\det g\right) \fieldCharacter\left(\trace g^{-1}\right)$ is a class function, the operator $\GKGaussSum{\pi}{\chi}{\fieldCharacter}$ lies in $\Hom_{\GL_n\left(\finiteField\right)}\left(\pi, \pi\right)$, and therefore by Schur's lemma there exists a complex number $\gGJPreGammaFactor{\pi}{ \chi}{\fieldCharacter} \in \cComplex$ such that $$\GKGaussSum{\pi}{\chi}{\fieldCharacter} = \gGJPreGammaFactor{\pi}{\chi}{\fieldCharacter} \cdot \idmap_{\pi}.$$
This gamma factor fits into a functional equation~\cite{Macdonald80}, which we will explain now. If $\chi^{-1}$ does not appear in the cuspidal support of $\pi$, then for any $f \colon \squareMatrix_n\left(\finiteField\right) \to \cComplex$ we have $$\sum_{g \in \GL_n\left(\finiteField\right)} \fourierTransform{\fieldCharacter}{f}\left(g\right) \chi^{-1}\left(\det g\right) \pi\left(g^{-1}\right) = \GJPreGammaFactor{\pi \times \chi}{\fieldCharacter} \sum_{g \in \GL_n\left(\finiteField\right)} f\left(g\right) \chi\left(\det g\right) \pi\left(g\right),$$
where $\fourierTransform{\fieldCharacter}{f} \colon \squareMatrix_n\left(\finiteField\right) \to \cComplex$ is the Fourier transform of $f$ (see \Cref{sec:godement-jacquet-functional-equation}).

In~\cite{Kondo1963}, Kondo explicitly computed $\gGJPreGammaFactor{\pi}{\chi}{\fieldCharacter}$ in terms of Gauss sums of the characters that correspond to the cuspidal support of $\pi$. 

\subsubsection{Second family: constructions via Whittaker models}
The second family of constructions assume that the representation $\pi$ is \emph{generic}. We briefly recall this notion. Assume that $n \ge 2$ and define a character of the upper triangular unipotent subgroup $\UnipotentSubgroup_n$ of $\GL_n\left(\finiteField\right)$ by the formula
$$\fieldCharacter \begin{pmatrix}
	1 & x_1 & \ast & \ast  & \ast \\
	& 1 & x_2 & \ast & \ast \\
	& & \ddots & \ddots & \ast  \\
	& & & 1 &  x_{n-1} \\
	& & & & 1
\end{pmatrix} = \fieldCharacter\left( \sum_{j=1}^{n-1} x_j \right).$$
The representation $\pi$ is generic if it admits a non-zero \emph{$\fieldCharacter$-Whittaker vector}, that is, if there exists $0 \ne v \in \pi$ such that $\pi\left(u\right) v = \fieldCharacter\left(u\right) v$ for every $u \in \UnipotentSubgroup_n$. A well known result of Gelfand--Graev is that in this case, such a vector $v$ is unique up to scalar multiplication. Let $\Whittaker\left(\pi, \fieldCharacter\right)$ be the unique subspace of $\Ind{\UnipotentSubgroup_n}{\GL_n\left(\finiteField\right)}{\fieldCharacter}$ that is isomorphic to $\pi$, this is the \emph{$\fieldCharacter$-Whittaker model of $\pi$}. Choose an inner product $\innerproduct{\cdot}{\cdot}$ on $\pi$ invariant to the action of $\GL_n\left(\finiteField\right)$ and consider the matrix coefficient $$\besselFunction_{\pi, \fieldCharacter}\left(g\right) = \frac{\innerproduct{\pi\left(g\right) v}{v}}{\innerproduct{v}{v}},$$
where $0 \ne v \in \pi$ is a $\fieldCharacter$-Whittaker vector. The matrix coefficient $\besselFunction_{\pi, \fieldCharacter}$ is independent of the choice of $v$ and of the inner product, and we call it the \emph{normalized Bessel function of $\pi$ with respect to $\fieldCharacter$}.

Using the Bessel function we may define a gamma factor for the pair $\pi \times \chi$ by the formula $$\genericRepGammaFactor{\pi}{\chi}{\fieldCharacter} = q^{\frac{n-2}{2}} \sum_{x \in \multiplicativegroup{\finiteField}} \besselFunction_{\pi, \fieldCharacter} \begin{pmatrix}
	& \IdentityMatrix{n-1}\\
	x
\end{pmatrix} \chi\left(x\right).$$
This gamma factor fits into a functional equation of Jacquet--Piatetski-Shapiro--Shalika~\cite{Jacquet1983rankin}: Let $W \in \Whittaker\left(\pi, \fieldCharacter\right)$ be an element of the $\fieldCharacter$-Whittaker model of $\pi$. Set
$$\zetaOperator\left(W,\chi\right) = \sum_{x \in \multiplicativegroup{\finiteField}} W \begin{pmatrix}
	x\\
	& \IdentityMatrix{n-1}
\end{pmatrix}\chi\left(x\right)$$ and \begin{equation*}
	\dualZetaOperator\left(W, \chi\right) = q^{-\frac{n-2}{2}} \sum_{\substack{x \in \multiplicativegroup{\finiteField}\\
			Y \in \Mat{1}{\left(n-2\right)}\left(\finiteField\right)}} W\begin{pmatrix}
		& 1\\
		& & \IdentityMatrix{n-2}\\		
		x & & Y
	\end{pmatrix} \chi\left(x\right).
\end{equation*} If $\chi^{-1}$ does not appear in the cuspidal support of $\pi$, then the following functional equation holds: $$\dualZetaOperator\left(W, \chi\right) = \genericRepGammaFactor{\pi}{\chi}{\fieldCharacter} \zetaOperator\left(W, \chi\right).$$

As explained in~\cite{SoudryZelingher2023}, the gamma factor $\genericRepGammaFactor{\pi}{\chi}{\fieldCharacter}$ can also be realized using Shahidi's construction~\cite{shahidi1984fourier} which considers the intertwining operator between the parabolically induced representations $\pi \circ \chi^{-1} \to \chi^{-1} \circ \pi$. We will not use this approach in this paper.

In~\cite{Nien17}, Nien computed the value of $\genericRepGammaFactor{\pi}{\chi}{\fieldCharacter}$ for an irreducible cuspidal representation $\pi$. In~\cite[Section 4.3]{SoudryZelingher2023}, Soudry and the second author used the multiplicativity property of this gamma factor and Nien's computation to establish an equality of this gamma factor and the one from the previous section.

We would like to mention that Nien's computation involves computing the special values $\besselFunction_{\pi, \fieldCharacter}\left(\begin{smallmatrix}
	& \IdentityMatrix{n-1}\\
	x
\end{smallmatrix}\right)$ appearing in the definition of $\genericRepGammaFactor{\pi}{\chi}{\fieldCharacter}$ for cuspidal representations $\pi$. A computation of this special value for general irreducible generic representations $\pi$ was done before by Curtis--Shinoda in~\cite{curtis2004zeta}. These special values of the Bessel function turn out to be exotic Kloosterman sums, which were defined and studied by Katz~\cite[Sections 8.8.4--8.8.7]{katz2016gauss} (see also~\cite[Page 152]{katz1993estimates}). In~\cite{SoudryZelingher2023} it was shown that the result of Curtis--Shinoda can be deduced by Nien's computation combined with the multiplicativity property of the gamma factor $\genericRepGammaFactor{\pi}{\chi}{\fieldCharacter}$.

Although this second family of constructions might seem more involved than the first family we discussed, it generalizes to any pair of generic representations of $\GL_n\left(\finiteField\right)$ and $\GL_m\left(\finiteField\right)$ for any $n$ and $m$. This was established for local fields in the pioneering works of Jacquet--Piatetski-Shapiro--Shalika~\cite{Jacquet1983rankin} and Shahidi~\cite{shahidi1984fourier} (which was later vastly generalized in \cite{shahidi1990proof}). The analogous theory of the construction of Jacquet--Piatetski-Shapiro--Shalika over finite fields was developed by Edva Roditty-Gershon in her master's thesis~\cite{Roditty10} and was studied later by Nien~\cite{Nien14, Nien17}, who used it to state and prove the converse theorem. The analogous theory of Shahidi's construction over finite fields was studied by Soudry and the second author~\cite{soudry1979, SoudryZelingher2023}.

\subsubsection{Constructions of doubling type}

In the last few years, constructions of doubling type of tensor product gamma factors over local fields appeared in the literature~\cite{CaiFriedbergGinzburgKaplan2019, CaiFriedbergKaplan2022, kaplan2018, GourevitchKaplan2023, GinzburgSoudry2020}. These constructions are integral representations for the $L$-function corresponding to the tensor product representation, attached to a pair of representations $\pi$ and $\tau$, of $G$ and $\GL_k$, respectively, where $G$ is a classical group. A remarkable feature of these constructions is that they only assume that $\tau$ is generic, and they do not assume anything about $\pi$ (in particular, they do not assume that $\pi$ is generic). This is crucial because irreducible cuspidal representations of $G$ need not be generic. This allowed Cai--Friedberg--Kaplan to give a new proof for the existence of a standard functorial lift for non-generic cuspidal automorphic representations of split classical groups~\cite{kaplan2018}. Another interesting feature of these constructions is that they do not use the representation $\tau$ directly, but instead go through a generalized Whittaker model of a Speh representation attached to it. We move to explain this point in the context of finite fields. 

Suppose that $k \ge 1$ and $c \ge 1$. Let $\UnipotentRadicalForWss{k}{c}$ be the unipotent radical of $\GL_{kc}\left(\finiteField\right)$ corresponding to the composition $\left(c^k\right)$. We define a character $\fieldCharacterkc{k}{c} \colon \UnipotentRadicalForWss{k}{c} \to \multiplicativegroup{\cComplex}$ by the formula $$\fieldCharacterkc{k}{c} \begin{pmatrix}
	\IdentityMatrix{c} & X_1 & \ast & \ast  & \ast \\
	& \IdentityMatrix{c} & X_2 & \ast & \ast \\
	& & \ddots & \ddots & \ast  \\
	& & & \IdentityMatrix{c} &  X_{k-1} \\
	& & & & \IdentityMatrix{c}
\end{pmatrix} = \fieldCharacter\left( \sum_{j=1}^{k-1} \trace X_j \right).$$
Let $\tau$ be an irreducible generic representation of $\GL_k\left(\finiteField\right)$. In his master's thesis~\cite{Carmon2023}, the first author defined the notion of the \emph{Speh representation} $\SpehRepresentation{\tau}{c}$. It is a representation of $\GL_{kc}\left(\finiteField\right)$. When $\tau$ is cuspidal, it is also known in the literature as the \emph{generalized trivial representation}. In the case that $\tau$ is cuspidal, $\SpehRepresentation{\tau}{c}$ is the irreducible subrepresentation of minimal dimension of the parabolic induction $\tau^{\circ c} = \tau \circ \dots \circ \tau$. The first author proved that for any irreducible generic representation $\tau$, the Speh representation $\SpehRepresentation{\tau}{c}$ admits a multiplicity one property with respect to $\fieldCharacterkc{k}{c}$: he showed that there exists a unique (up to scalar multiplication) $\kcNotation{k}{c}{\fieldCharacter}$ vector $0 \ne v \in \SpehRepresentation{\tau}{c}$, that is, a vector $v$ such that $\SpehRepresentation{\tau}{c}\left(u\right) v = \fieldCharacterkc{k}{c}\left(u\right)v$ for any $u \in \UnipotentRadicalForWss{k}{c}$.
Furthermore, for any $h\in\GL_c(\finiteField)$, this $v$ satisfies
\(
\SpehRepresentation{\tau}{c}(\diag^k(h))v
=
\centralCharacter{\tau}(\det h)v,
\)
where $\diag^k(h)=\diag(h,\dots,h)$ and $\centralCharacter{\tau}$ is the central character of $\tau$.
These results are analogous to the results of~\cite{CaiFriedbergGourevitchKaplan2023}.

Let $\centralCharacter{\tau}^{\diag^k\left(\GL_c\right)}$ be the character of $\diag^k\left(\GL_c(\finiteField)\right)$ given by ${\centralCharacter{\tau}^{\diag^k\left(\GL_c\right)}(\diag^k\left(h\right)) = \centralCharacter{\tau}\left(\det h\right)}.$
Similarly to before, denote by $\Whittaker\left(\SpehRepresentation{\tau}{c}, \fieldCharacterkc{k}{c}\right)$ the unique subspace of $$\Ind{\diag^k\left(\GL_c(\finiteField)\right)\ltimes\UnipotentRadicalForWss{k}{c}}{\GL_{kc}\left(\finiteField\right)}{\centralCharacter{\tau}^{\diag^k\left(\GL_c\right)}\otimes\fieldCharacterkc{k}{c}}$$ that is isomorphic to $\SpehRepresentation{\tau}{c}$. This is the \emph{$\kcNotation{k}{c}{\fieldCharacter}$~model of $\SpehRepresentation{\tau}{c}$}. Using the uniqueness of the $\kcNotation{k}{c}{\fieldCharacter}$ vector $v$, we may define an analogous Bessel function. Choose an inner product $\innerproduct{\cdot}{\cdot}$ on $\SpehRepresentation{\tau}{c}$ invariant to the $\GL_{kc}\left(\finiteField\right)$ action and define
$$\besselSpehFunction{\tau}{c}\left(g\right) = \frac{\innerproduct{\SpehRepresentation{\tau}{c}\left(g\right)v}{v}}{\innerproduct{v}{v}}.$$
We call $\besselSpehFunction{\tau}{c}$ the \emph{normalized Bessel--Speh function of $\SpehRepresentation{\tau}{c}$ with respect to $\fieldCharacter$}.

It is natural to ask at this point whether the Bessel--Speh function can be used in order to define gamma factors. It is also natural to ask what is the formula for the special values $\besselSpehFunction{\tau}{c}\left(\begin{smallmatrix}
	& \IdentityMatrix{\left(k-1\right)c}\\
	h
\end{smallmatrix}\right)$ for $h \in \GL_c\left(\finiteField\right)$. Are these Kloosterman sums of higher dimension?

\subsection{Main results}

\subsubsection{Ginzburg--Kaplan gamma factors}\label{subsec: GK gamma factors}

For any $h \in \GL_c\left(\finiteField\right)$, the function $\besselSpehFunction{\tau}{c}\left(g\right)$ is invariant under conjugation by $\diag^k\left(h\right) \in \GL_{kc}\left(\finiteField\right)$. It follows that the assignment $$h \mapsto \specialBesselSpeh{\tau}\left(h\right) \coloneqq  \begin{dcases}
	\besselSpehFunction{\tau}{c}\begin{pmatrix}
		& \IdentityMatrix{\left(k-1\right)c}\\
		h
	\end{pmatrix} & k \ge 2,\\
	\tau\left(\det h\right) \fieldCharacter\left(\trace h^{-1}\right) & k=1
\end{dcases}$$ is a class function of $\GL_c\left(\finiteField\right)$.  Given an irreducible representation $\pi$ of $\GL_c\left(\finiteField\right)$, we may consider the operator
$$\GKGaussSum{\pi}{\tau}{\fieldCharacter} = q^{\frac{\left(k-2\right)c^2}{2}} \sum_{h \in \GL_c\left(\finiteField\right)} \specialBesselSpeh{\tau}\left(h\right) \pi\left(h\right).$$
This operator is closely related to a new construction of the tensor product gamma factor due to Kaplan in the local case~\cite[Appendix A]{kaplan2018} and due to Ginzburg in the global case~\cite{ginzburg2019tensor}. By the discussion above, $\GKGaussSum{\pi}{\tau}{\fieldCharacter}$ lies in $\Hom_{\GL_c\left(\finiteField\right)}\left(\pi, \pi\right)$, and therefore by Schur's lemma is a scalar operator. Denote $$\GKGaussSum{\pi}{\tau}{\fieldCharacter} = \GKPreGammaFactor{\pi}{\tau}{\fieldCharacter} \cdot \idmap_{\pi},$$
where $\GKPreGammaFactor{\pi}{\tau}{\fieldCharacter} \in \cComplex$. We call  $\GKPreGammaFactor{\pi}{\tau}{\fieldCharacter}$ the \emph{Ginzburg--Kaplan gamma factor}.

We prove that this gamma factor fits into a functional equation that can be seen as a generalization of the previous functional equations discussed above (see \Cref{rem:gk-zeta-integral-with-function} and \Cref{thm: functional equation 1}. See also \Cref{thm: functional equation 2}).

\begin{theorem}
	Suppose that $k \ge 2$. For $W \in \Whittaker\left(\SpehRepresentation{\tau}{c}, \fieldCharacterkc{k}{c}\right)$ and $f \colon \squareMatrix_c\left(\finiteField\right) \to \cComplex$ let $$\zetaOperator\left(W, f, \pi \times \tau\right) = \sum_{g \in \GL_c\left(\finiteField\right)}
	W\begin{pmatrix}
		g \\
		& \IdentityMatrix{\left(k-1\right)c}
	\end{pmatrix} f\left(g\right) \pi\left(g\right)$$ and $$\dualZetaOperator\left(W, f, \pi \times \tau\right) = 	q^{-\frac{\left(k-3\right)c^2}{2}}
	\sum_{\substack{{X \in \Mat{c}{\left(k-2\right)c}\left(\finiteField\right)}\\
			Y \in \squareMatrix_c\left(\finiteField\right)\\
			{g \in \GL_c\left(\finiteField\right)}}}
	W \begin{pmatrix}
		& \IdentityMatrix{c}& \\
		& & \IdentityMatrix{\left(k-2\right)c} \\
		g & Y & X
	\end{pmatrix} \fourierTransform{\fieldCharacter^{-1}}{f}\left(g^{-1} Y\right)  \pi\left(g\right).$$
	Suppose that $\pi$ and $\Contragradient{\tau}$ do not have common elements in their cuspidal supports. Then
	$$\dualZetaOperator\left(W, f, \pi \times \tau\right) = \GKPreGammaFactor{\pi}{\tau}{\fieldCharacter} \cdot \zetaOperator\left(W, f, \pi \times \tau\right).$$
\end{theorem}
Notice that when $c=1$, $k=n$ and $f = 1$, we get the zeta integral of Jacquet--Piatetski-Shapiro--Shalika. Also notice that even though we only allow $k \ge 2$, the definition of $\zetaOperator$ also makes sense when $k=1$, and if $c=n$ we get that $\zetaOperator$ coincides with the zeta integral defined by Godement--Jacquet (however, in this case the dual integral $\dualZetaOperator$ does not make sense and requires a modification).

We prove that this gamma factor is multiplicative in both arguments (\Cref{thm:multiplicativitiy-in-first-variable,thm:multiplicativitiy-in-second-variable}). 
\begin{theorem}
	\begin{enumerate}
		\item Suppose that $k = k_1 + k_2$ and that for $j=1,2$, $\tau_j$ is an irreducible generic representation of $\GL_{k_j}\left(\finiteField\right)$, such that $\tau$ is an irreducible generic subrepresentation of the parabolic induction $\tau_1 \circ \tau_2$. Then 		$$\GKPreGammaFactor{\pi}{\tau}{\fieldCharacter} = \centralCharacter{\pi}\left(-1\right) \cdot \GKPreGammaFactor{\pi}{\tau_1}{\fieldCharacter} \cdot \GKPreGammaFactor{\pi}{\tau_2}{\fieldCharacter}.$$		
		\item Suppose that $c = c_1 + c_2$ and that for $j=1,2$, $\pi_j$ is an irreducible representation of $\GL_{c_j}\left(\finiteField\right)$, such that $\pi$ is an irreducible subrepresentation of the parabolic induction $\pi_1 \circ \pi_2$. Then
$$\GKPreGammaFactor{\pi}{\tau}{\fieldCharacter} = \GKPreGammaFactor{\pi_1}{\tau}{\fieldCharacter} \cdot \GKPreGammaFactor{\pi_2}{\tau}{\fieldCharacter}.$$		
	\end{enumerate}
\end{theorem}

\subsubsection{Relation to matrix Kloosterman sums}
The proof of these multiplicativity properties goes through proving certain identities for the Bessel--Speh function, which are described in the following theorem. These identities show a relation between the Bessel--Speh function and twisted matrix Kloosterman sums (\Cref{lem:convolution-of-bessel-speh-functions} and \Cref{thm:unipotent-fourier-transofrm-bessel-speh}).
\begin{theorem}\label{thm:bessel-speh-function-identities}
	Using the notations and the assumptions above we have
	\begin{enumerate}
		\item For any $h \in \GL_c\left(\finiteField\right)$, we have
		\begin{equation*}
			 \specialBesselSpeh{\tau}\left(h\right) = q^{- c^2} \sum_{\substack{x,y \in \GL_c\left(\finiteField\right)\\
					xy = -h}} \specialBesselSpeh{\tau_1}\left(x\right) \specialBesselSpeh{\tau_2}\left(y\right).
		\end{equation*}
		\item For any $h = \diag\left(h_1, h_2\right)$ where $h_1 \in \GL_{c_1}\left(\finiteField\right)$ and $h_2 \in \GL_{c_2}\left(\finiteField\right)$, we have,
		\begin{equation*}
			\begin{split}
				& q^{-c_1 c_2} \sum_{n \in \UnipotentRadical_{\left(c_1, c_2\right)}} \specialBesselSpeh{\tau}\left(nh\right)  = q^{- c_1 c_2 \left(k-1\right)} \specialBesselSpeh{\tau}\left(h_1\right) \specialBesselSpeh{\tau} \left(h_2\right).
			\end{split}
		\end{equation*}
	\end{enumerate}
\end{theorem}
Using the first identity repeatedly implies that if $\tau$ is a subrepresentation of the parabolic induction $\alpha_1 \circ \dots \circ \alpha_k$, where $\alpha_1, \dots, \alpha_k \colon \multiplicativegroup{\finiteField} \to \multiplicativegroup{\cComplex}$ are characters (i.e., $\tau$ is a generic principal series representation), then for any $h \in \GL_c\left(\finiteField\right)$,
$$ \specialBesselSpeh{\tau}\left(h\right) = q^{-\left(k-1\right)c^2} \sum_{\substack{x_1,\dots,x_k \in \GL_c\left(\finiteField\right)\\
		x_1 \cdot \dots \cdot x_k = \left(-1\right)^{k-1} h^{-1} }} \left(\prod_{j=1}^k \alpha_j^{-1}\left(\det x_j\right)\right) \fieldCharacter\left(\sum_{j=1}^k \trace x_j\right).$$
The sum appearing on the right hand side is the twisted matrix Kloosterman sum $\ExoticKloosterman(\alpha^{-1}, \fieldCharacter, \left(-1\right)^{k-1} h^{-1})$ studied in ~\cite{Zelingher2023}, where $\alpha = \alpha_1 \times \dots \times \alpha_k$.

As an immediate corollary of the second identity of \Cref{thm:bessel-speh-function-identities}, we get the following multiplicativity property of matrix Kloosterman sums (see \Cref{thm:multiplicativity-of-exotic-kloosterman-sums}).
\begin{theorem}\label{thm:multiplicativity-of-matrix-kloosterman-sums}
	If $h_1 \in \GL_{c_1}\left(\finiteField\right)$ and $h_2 \in \GL_{c_2}\left(\finiteField\right)$ have no common eigenvalues over the algebraic closure $\algebraicClosure{\finiteField}$, then for $h = \diag\left(h_1, h_2\right)$ we have
	$$\ExoticKloosterman\left(\alpha, \fieldCharacter, h\right) = q^{\left(k-1\right) c_1 c_2} \cdot \ExoticKloosterman\left(\alpha, \fieldCharacter, h_1\right) \cdot \ExoticKloosterman\left(\alpha, \fieldCharacter, h_2\right).$$
\end{theorem}
This was proved by \Erdelyi{}--\Toth{} for the special case $k =2$ and $\alpha_1 = \alpha_2 = 1$ in~\cite[Theorem 1.1]{erdelyi2021matrix}. However, our proof is representation theoretical and based on $\kcNotation{k}{c}{\fieldCharacter}$ vectors for Speh representations, and our result is more general, as it holds for any $k$ and any $\alpha$.

It is now natural to consider generic representations $\tau$ that are not principal series representations. Is the special value $\specialBesselSpeh{\tau}\left(h\right)$ also some sort of matrix Kloosterman sum? In a sequel paper \cite{Zelingher2025}, the second author answers this question affirmatively. He defines the notion of exotic matrix Kloosterman sums and explains how to attach to each irreducible generic representation $\tau$ of $\GL_k\left(\finiteField\right)$ an exotic matrix Kloosterman sum such that a similar identity holds for $\specialBesselSpeh{\tau}\left(h\right)$.

\subsubsection{A new converse theorem}
Using the relation between the Ginzburg--Kaplan gamma factor and the tensor product $\varepsilon_0$-factor, we are able to deduce a new converse theorem, based on the special values $\specialBesselSpeh{\tau}$ (\Cref{thm:converse-theorem-based-on-bessel-speh-representation}).
\begin{theorem}
	Let $\tau_1$ and $\tau_2$ be irreducible generic representations of $\GL_k\left(\finiteField\right)$ with the same central character. Suppose that for every $1 \le c \le \frac{k}{2}$ and every $h \in \GL_c\left(\finiteField\right)$, $$\specialBesselSpeh{\tau_1}\left(h\right) = \specialBesselSpeh{\tau_2}\left(h\right).$$
	Then $\tau_1$ and $\tau_2$ are isomorphic.
\end{theorem}

\subsection*{Organization of the paper}
The paper is organized as follows. In \Cref{sec:speh-representations}, we discuss the notion of representations of $(k,c)$ type and the notion of Speh representations, and discuss the $\kcNotation{k}{c}{\fieldCharacter}$ models associated with them.
 
In \Cref{sec:gamma-factors}, we recall the definition of the Godement--Jacquet gamma factor and define the Ginzburg--Kaplan gamma factor. We show that the Ginzburg--Kaplan gamma factor is multiplicative in both arguments and that it fits into several functional equations. 
 
In \Cref{sec:special-values-of-bessel-speh-function}, we use the multiplicativity results from \Cref{sec:gamma-factors} to deduce the multiplicativity property of twisted matrix Kloosterman sums. In addition, we give a new converse theorem. 
 
In Appendix \ref{appendix:appendix-special-value-in-for-c-2}, we show another approach for computing the special values of the Bessel--Speh function for the special case $c=2$.

\subsection*{Acknowledgments}
We would like to thank Eyal Kaplan for his thorough reading of an earlier draft of this manuscript and for his helpful comments and suggestions, which improved the mathematical exposition.

This work and \cite{Zelingher2025} originally appeared as a single preprint. Acting on a referee's recommendation, the work was split into two parts.

Finally, we would like to thank the referee for their careful reading and valuable feedback which improved the paper.

\tableofcontents

\section{Representations of \texorpdfstring{$\left(k,c\right)$}{(k,c)} type and Speh representations}\label{sec:speh-representations}

In this section, we discuss the notions of representations of $(k,c)$ type and $\kcNotation{k}{c}{\fieldCharacter}$ vectors and functionals. We then define a special family of representations attached to irreducible generic representations, called \emph{Speh representations}. These are representations of $(k,c)$ type. We end this section with explaining a recursive construction of $\kcNotation{k}{c}{\fieldCharacter}$ functionals for these Speh representations.

Let $\finiteField$ be a field with $q$ elements. Let $\fieldCharacter \colon \finiteField \rightarrow \multiplicativegroup{\cComplex}$ be a non-trivial character.

\subsection{Parameterization of irreducible representations of \texorpdfstring{$\GL_n\left(\finiteField\right)$}{general linear groups}}
In this section, we give a quick overview of the parameterization of irreducible representations of $\GL_n\left(\finiteField\right)$. See also~\cite[Section 6.2]{SayagVerma2020} or~\cite[Section 1]{Macdonald80}.

\subsubsection{Parabolic induction}\label{sec: Parabolic induction}
If $n$ is a non-negative integer, a \emph{composition} of $n$ is a tuple $\left(n_1,\dots,n_r\right)$ of positive integers, such that $n_1 + \dots + n_r = n$. The number $r$ is called the \emph{length} of the composition. A \emph{partition} of $n$ is a composition $\lambda = \left(n_1,\dots,n_r\right)$ of $n$ where the sequence is weakly decreasing, i.e., $n_1 \ge n_2 \ge \dots n_r > 0$ with $n_1 + \dots + n_r = n$. We write $\lambda \vdash n$ to specify that $\lambda$ is a partition of $n$ and denote $\abs{\lambda} = n$. Notice that if $n = 0$ then $()$ is the only composition (and partition) of $n$.

Given a composition $\left(n_1,\dots,n_r\right)$ of $n$, we denote $$\UnipotentRadical_{\left(n_1,\dots,n_r\right)} = \left\{ \begin{pmatrix}
	\IdentityMatrix{n_1} & \ast & \ast & \ast \\
	& \IdentityMatrix{n_2} & \ast & \ast \\
	& & \ddots & \ast \\
	& & & \IdentityMatrix{n_r}
\end{pmatrix}\right\} \,\,\,\text{and}\,\,\,\, D_{\left(n_1,\dots,n_r\right)} = \left\{\diag\left(g_1,\dots,g_r\right) \mid g_j \in \GL_{n_j}\left(\finiteField\right)\right\}.$$
We set $\ParabolicSubgroup_{\left(n_1,\dots,n_r\right)} = D_{\left(n_1,\dots,n_r\right)} \ltimes \UnipotentRadical_{\left(n_1,\dots,n_r\right)}$, this is the \emph{standard parabolic subgroup of $\GL_n\left(\finiteField\right)$ corresponding to the composition $\left(n_1,\dots,n_r\right)$}. We call $\UnipotentRadical_{\left(n_1,\dots,n_r\right)}$ and $D_{\left(n_1,\dots,n_r\right)}$ the unipotent radical and the Levi part of $\ParabolicSubgroup_{\left(n_1,\dots,n_r\right)}$, respectively.

Given representations $\sigma_1, \dots, \sigma_r$ of $\GL_{n_1}\left(\finiteField\right)$, $\dots$, $\GL_{n_r}\left(\finiteField\right)$, we define a representation $\sigma_1 \overline{\otimes} \dots \overline{\otimes} \sigma_r$ of $\ParabolicSubgroup_{\left(n_1,\dots,n_r\right)}$ by inflation as follows. The representation $\sigma_1 \overline{\otimes} \dots \overline{\otimes} \sigma_r$ acts on the tensor product space $\sigma_1 \otimes \dots \otimes \sigma_r$ with action $$\left(\sigma_1 \overline{\otimes} \dots \overline{\otimes} \sigma_r\right) \left(du\right) = \sigma_1\left(g_1\right) \otimes \dots \otimes \sigma_r\left(g_r\right),$$
where $d = \diag\left(g_1,\dots,g_r\right) \in D_{\left(n_1,\dots,n_r\right)}$ and $u \in \UnipotentRadical_{\left(n_1,\dots,n_r\right)}$.

For $\sigma_1,\dots,\sigma_r$ as above, we define $\sigma_1 \circ \dots \circ \sigma_r$ to be the following induced representation of $\GL_n\left(\finiteField\right)$:
$$\sigma_1 \circ \dots \circ \sigma_r = \Ind{\ParabolicSubgroup_{\left(n_1,\dots,n_r\right)}}{\GL_n\left(\finiteField\right)}{\sigma_1 \overline{\otimes} \dots \overline{\otimes} \sigma_r}.$$
The operation $\circ$ is called \emph{parabolic induction}\footnote{Over local fields a more common notation for parabolic induction is $\times$. However, we stick to the notation $\circ$, commonly used in the literature of finite fields~\cite{Green55, Gelfand70, Macdonald80}.}. It is a commutative and associative operation.

\subsubsection{Cuspidal representations and cuspidal support}
A representation $\pi$ of $\GL_n\left(\finiteField\right)$ is called \emph{cuspidal} if for any composition $\left(n_1,\dots,n_r\right) \ne \left(n\right)$ of $n$, $\pi$ does not admit non-zero $\UnipotentRadical_{\left(n_1,\dots,n_r\right)}$-fixed vectors, that is, if $v \in \pi$ is such that $\pi\left(u\right)v = v$ for every $u \in \UnipotentRadical_{\left(n_1,\dots,n_r\right)}$ for some composition $\left(n_1,\dots,n_r\right) \ne \left(n\right)$ of $n$, then $v = 0$. If $\pi$ is irreducible, $\pi$ is cuspidal if and only if $\pi$ is not a subrepresentation of a parabolically induced representation $\sigma_1 \circ \dots \circ \sigma_r$ as above, where $r \ge 2$.

By~\cite[Theorem 2.4]{Gelfand70}, for any irreducible representation $\pi$ of $\GL_n\left(\finiteField\right)$, there exists a composition $\left(n_1,\dots,n_r\right)$ of $n$ and irreducible cuspidal representations $\sigma_1$,$\dots$,$\sigma_r$ of $\GL_{n_1}\left(\finiteField\right)$, $\dots$, $\GL_{n_r}\left(\finiteField\right)$, respectively, such that $\pi$ is a subrepresentation of the parabolic induction $\sigma_1 \circ \dots \sigma_r$. The composition $\left(n_1,\dots,n_r\right)$ and the equivalence classes of $\sigma_1$, $\dots$, $\sigma_r$ are unique, up to permutation. We define the \emph{cuspidal support of $\pi$} to be the multiset (of equivalence classes of representations) $\left\{\sigma_1,\dots,\sigma_r\right\}$.

\subsubsection{Parameters of irreducible representations}

There exists a refinement of the cuspidal support of $\pi$ described in the previous section. A \emph{parameter} $\varphi$ is a partition valued assignment $\left(d, \sigma\right) \mapsto \varphi\left(d, \sigma\right)$, where $d \ge 1$ is an integer, $\sigma$ is an (equivalence class of an) irreducible cuspidal representation of $\GL_d\left(\finiteField\right)$ and $\varphi\left(d, \sigma\right)$ is a partition, such that for almost every $d$, and $\sigma$, $\varphi\left(d, \sigma\right) = ()$. Given a parameter $\varphi$, we define $$\abs{\varphi} = \sum_{\left(d, \sigma\right)} d \cdot \abs{\varphi\left(d, \sigma\right)}.$$
Then the irreducible representations of $\GL_n\left(\finiteField\right)$ are in bijection with parameters $\varphi$ such that $\abs{\varphi} = n$. We give some details about this bijection.

If $\sigma$ is an irreducible representation of $\GL_d\left(\finiteField\right)$, denote $\sigma^{\left(1\right)} = \sigma$. For $m \ge 2$, the parabolic induction $\sigma^{\circ m} = \sigma \circ \dots \circ \sigma$ is not irreducible. The equivalence classes of its irreducible subrepresentations\footnote{The representation $\sigma^{\circ m}$ is not multiplicity free for $m \ge 3$.} are parameterized by partitions of $m$ (see~\cite[Appendices B and C]{GurevichHowe2021}). Given a partition $\lambda \vdash m$, let $\sigma^{\lambda}$ be the irreducible subrepresentation of $\sigma^{\circ m}$ corresponding to the partition $\lambda$.

Given a parameter $\varphi$ with $\abs{\varphi} = n$, consider the parabolically induced representation $$\pi_{\varphi} = \sigma_1^{\varphi\left(d_1, \sigma_1\right)} \circ \dots \circ \sigma_r^{\varphi\left(d_r, \sigma_r\right)},$$
where $\left(d_1, \sigma_1\right), \dots, \left(d_r, \sigma_r\right)$ are all the pairs such that $\varphi\left(d, \sigma\right) \ne ()$. Then $\pi_{\varphi}$ is an irreducible representation of $\GL_n\left(\finiteField\right)$, and the map $\varphi \mapsto \pi_{\varphi}$ is a bijection. Notice that if $\varphi$ is such that $\varphi\left(n, \sigma\right) = (1)$ and $\varphi\left(d, \sigma'\right) = ()$ for any other pair $(d, \sigma')$, then $\pi_{\varphi} = \sigma$ is a cuspidal representation. More generally, the cuspidal support of $\pi_{\varphi}$ is the sum of the following multisets: $\abs{\varphi\left(d_1, \sigma_1\right)}$ copies of $\sigma_1$, $\abs{\varphi\left(d_2, \sigma_2\right)}$ copies of $\sigma_2$, $\dots$, $\abs{\varphi\left(d_r, \sigma_r\right)}$ copies of $\sigma_r$.

\subsection{Representations of \texorpdfstring{$\left(k,c\right)$}{(k,c)} type}\label{sec:kc-representations}
In this section, we define the notion of representations of $(k,c)$ type. We also define the notion of $\kcNotation{k}{c}{\fieldCharacter}$ vectors and functionals. We explain that representations of $(k,c)$ type admit $\kcNotation{k}{c}{\fieldCharacter}$ vectors and functionals.

\subsubsection{Zelevinsky characters}
We recall the notion of Zelevinsky characters. It will be used in order to define representations of $(k,c)$ type in the next section.

We recall the partial dominance order of partitions. If $\lambda = \left(n_1,\dots,n_r\right)$ and $\lambda' = \left(n'_1,\dots,n'_{r'}\right)$ are partitions of $n$, we write $\lambda \ge \lambda'$ if for every $1 \le j \le \min\left(r, r'\right),$ we have $n_1 + \dots + n_j \ge n'_1 + \dots + n'_j$.

If $\left(n_1,\dots,n_r\right)$ and $\left(n'_1,\dots,n'_{r'}\right)$ are compositions, we write $\left(n_1,\dots,n_r\right) \ge \left(n'_1,\dots,n'_{r'}\right)$ if the partitions $\lambda$ and $\mu$ obtained by rearranging $\left(n_1,\dots,n_r\right)$ and $\left(n'_1,\dots,n'_{r'}\right)$ in non-increasing order, respectively, satisfy $\lambda \ge \mu$.

Let $\UnipotentSubgroup_n$ be the upper triangular unipotent subgroup of $\GL_n\left(\finiteField\right)$. For every composition $\lambda = \left(n_1, \dots, n_r\right)$ of $n$ we define a character $\zelevinskyCharacter{\lambda} \colon \UnipotentSubgroup_n \to \multiplicativegroup{\cComplex}$ as follows. Let $I_{\lambda} = \left\{n_1, n_1 + n_2, \dots, n_1 + \dots + n_r \right\}$. Then we define $$\zelevinskyCharacter{\lambda} \begin{pmatrix}
	1 & x_1 & \ast & \ast  & \ast \\
	& 1 & x_2 & \ast & \ast \\
	& & \ddots & \ddots & \ast  \\
	& & & 1 &  x_{n-1} \\
	& & & & 1
\end{pmatrix} = \fieldCharacter\left( \sum_{\substack{1 \le j \le n-1\\
j \notin I_{\lambda}}} x_j \right).$$

\subsubsection{Representations of $\left(k,c\right)$ type and $\kcNotation{k}{c}{\fieldCharacter}$ vectors}\label{sec:representations-of-k-c-type-and-whittaker-vectors}

Let $\Sigma$ be a finite dimensional representation of $\GL_{n}\left(\finiteField\right)$ and let $\lambda$ be a composition of $n$. A vector $0 \ne v \in \Sigma$ is called a \emph{$\fieldCharacter$-Zelevinsky vector of type $\lambda$} if for every $u \in \UnipotentSubgroup_n$, we have $\Sigma\left(u\right)v = \zelevinskyCharacter{\lambda}\left(u\right)v$. 

Suppose that $n = kc$. We say that $\Sigma$ is \emph{a representation of $\left(k,c\right)$ type} if $\Sigma$ admits a unique (up to scalar multiplication) $\fieldCharacter$-Zelevinsky vector of type $(k^c)$\footnote{Note that this definition is slightly different than the one in~\cite[Definition 1]{CaiFriedbergGourevitchKaplan2023}, as our definition goes through Zelevinsky vectors.}, and if for any composition $\lambda$ of $kc$ that is greater or not comparable with $(k^c)$, the representation $\Sigma$ does not admit a $\fieldCharacter$-Zelevinsky vector of type $\lambda$. This definition does not depend on the choice of the additive character $\fieldCharacter$.

Let $\UnipotentRadicalForWss{k}{c}$ be the unipotent radical of $\GL_{kc}\left(\finiteField\right)$ corresponding to the composition $(c^k)$. We define a character $\fieldCharacterkc{k}{c} \colon \UnipotentRadicalForWss{k}{c} \to \multiplicativegroup{\cComplex}$ by the formula
$$\fieldCharacterkc{k}{c} \begin{pmatrix}
	\IdentityMatrix{c} & X_1 & \ast & \ast  & \ast \\
	& \IdentityMatrix{c} & X_2 & \ast & \ast \\
	& & \ddots & \ddots & \ast  \\
	& & & \IdentityMatrix{c} &  X_{k-1} \\
	& & & & \IdentityMatrix{c}
\end{pmatrix} = \fieldCharacter\left( \sum_{j=1}^{k-1} \trace X_j \right).$$
A vector $0 \ne v \in \Sigma$ is called a \emph{$\kcNotation{k}{c}{\fieldCharacter}$ vector} if $\Sigma\left(u\right) v = \fieldCharacterkc{k}{c}\left(u\right)v$ for every $u \in \UnipotentRadicalForWss{k}{c}$. A non-zero element of $\Hom_{\UnipotentRadicalForWss{k}{c}}\left(\Sigma, \fieldCharacter\right)$ is called a \emph{$\kcNotation{k}{c}{\fieldCharacter}$ functional}. The representation $\Sigma$ admits a $\kcNotation{k}{c}{\fieldCharacter}$ vector if and only if it admits a $\kcNotation{k}{c}{\fieldCharacter}$ functional. When $c = 1$, we use the terms $\fieldCharacter$-Whittaker vector and $\fieldCharacter$-Whittaker functional, and write $\fieldCharacter$ for $\fieldCharacterkc{k}{1}$. Notice that in this case we have that $\fieldCharacter = \zelevinskyCharacter{(k)}$, and therefore the notion of a $\fieldCharacter$-Zelevinsky vector of type $(k)$ coincides with the notion of a $\fieldCharacter$-Whittaker vector.

We say that $\Sigma$ is \emph{generic} if $\Sigma$ admits a non-zero $\fieldCharacter$-Whittaker vector. We say that $\Sigma$ is of Whittaker type if $\Sigma$ is generic and if its $\fieldCharacter$-Whittaker vector is unique (up to scalar multiplication). As before, these definitions do not depend on the choice of $\fieldCharacter$. It is well known that if $\Sigma$ is an irreducible cuspidal representation of $\GL_k\left(\finiteField\right)$, then $\Sigma$ is generic. By another well known theorem of Gelfand--Graev, irreducible generic representations of $\GL_k\left(\finiteField\right)$ are of Whittaker type. Since $(k)$ dominates every composition of $k$, we conclude that the notion of representations of $\GL_k\left(\finiteField\right)$ of Whittaker type coincides with the notion of representations of type $(k, 1)$.

The first author proved in~\cite[Theorem 6.18]{Carmon2023} the following multiplicity one result (the proof is similar to~\cite[Theorem 4]{CaiFriedbergGourevitchKaplan2023} and goes through Bernstein--Zelevinsky derivatives).
\begin{theorem}
	 Suppose that $\Sigma$ is a representation of $\left(k,c\right)$ type. Then $\Sigma$ admits a $\kcNotation{k}{c}{\fieldCharacter}$ vector. This vector is unique up to scalar multiplication.
\end{theorem}

He also proved~\cite[Theorem 6.3]{Carmon2023} the following heredity-type result. If $k = k_1 + k_2$ and $\Sigma_1$ and $\Sigma_2$ are representations of $\left(k_1, c\right)$ and $\left(k_2, c\right)$ type, respectively, then the parabolic induction $\Sigma_1 \circ \Sigma_2$ is a representation of $\left(k,c\right)$ type. This result is analogous to~\cite[Proposition 2]{CaiFriedbergGourevitchKaplan2023}.

When $c=1$, a well known corollary of this heredity-type result is that for any irreducible generic representation $\Sigma$ of $\GL_k\left(\finiteField\right)$, there exist irreducible cuspidal representations $\Sigma_1, \dots, \Sigma_s$ of $\GL_{k_1}\left(\finiteField\right), \dots, \GL_{k_s}\left(\finiteField\right)$, respectively, with $k_1 + \dots + k_s = k$, such that $\Sigma$ is the unique irreducible generic subrepresentation of the parabolically induced representation $\Sigma_1 \circ \Sigma_2 \circ \dots \circ \Sigma_s$.

We may describe the parameter of an irreducible generic representation $\Sigma$ of $\GL_k\left(\finiteField\right)$ in terms of its cuspidal support. Suppose that $\sigma_1$, $\dots$, $\sigma_r$ are all of the different elements appearing in the cuspidal support of an irreducible generic representation $\Sigma$, such that for every $j$, $\sigma_j$ is an irreducible cuspidal representation of $\GL_{d_j}\left(\finiteField\right)$ that appears $m_j$ times in the cuspidal support of $\Sigma$. Then the parameter $\varphi_{\Sigma}$ of $\Sigma$ is the parameter satisfying $\varphi_{\Sigma}\left(d_j, \sigma_j\right) = (m_j)$ for every $j$ and $\varphi\left(d', \sigma'\right) = ()$ for any other pair.

\subsection{\texorpdfstring{$\kcNotation{k}{c}{\fieldCharacter}$}{(k,c)} models}\label{sec:whittaker-models}
Let $\Sigma$ be an irreducible representation of $\left(k,c\right)$ type. Then by Frobenius reciprocity, there exists a unique subspace of $\Ind{\UnipotentRadicalForWss{k}{c}}{\GL_{kc}\left(\finiteField\right)}{\fieldCharacterkc{k}{c}}$ which is isomorphic to $\Sigma$. 
We denote this space by $\Whittaker\left(\Sigma, \fieldCharacterkc{k}{c}\right)$ and call it the \emph{$\kcNotation{k}{c}{\fieldCharacter}$ model} of $\Sigma$. When $c=1$, we call this space the \emph{$\fieldCharacter$-Whittaker model} of $\Sigma$ and denote it by $\Whittaker\left(\Sigma, \fieldCharacter\right)$.

Let $$w_{\left(c^k\right)} = \begin{pmatrix}
	& & & & \IdentityMatrix{c}\\
	& & & \IdentityMatrix{c}\\
	& \iddots\\
	\IdentityMatrix{c} 
\end{pmatrix}.$$ We have a map $\Whittaker\left(\Sigma, \fieldCharacterkc{k}{c}\right) \to \Whittaker\left(\Sigma^{\vee}, \fieldCharacterkc{k}{c}^{-1}\right)$ denoted $W \mapsto \tilde{W}$, given by $$\tilde{W}\left(g\right) = W\left(w_{\left(c^k\right)} \inverseTranspose{g}\right),$$
where $g \in \GL_{kc}\left(\finiteField\right)$ and $\inverseTranspose{g} = \transpose{g}^{-1}$.

Denote for $h \in \GL_c\left(\finiteField\right)$, $$\diag^k\left(h\right) = \diag\left(h,\dots,h\right) \in \GL_{kc}\left(\finiteField\right).$$
For a character $\alpha \colon \multiplicativegroup{\finiteField} \to \multiplicativegroup{\cComplex}$ denote by $\alpha^{\diag^k\left(\GL_c\right)} \colon \diag^k\left(\GL_c\left(\finiteField\right)\right) \to \multiplicativegroup{\cComplex}$ the character $$\alpha^{\diag^k\left(\GL_c\right)}\left(\diag^k\left(h\right)\right) = \alpha\left(\det h\right).$$

Let $0 \ne v \in \Sigma$ be a $\kcNotation{k}{c}{\fieldCharacter}$ vector. We notice that for any $h \in \GL_c\left(\finiteField\right)$, the element $\diag^k\left(h\right)$ lies in the stabilizer of $\fieldCharacterkc{k}{c}$, and therefore $\Sigma\left(\diag^k\left(h\right)\right) v$ is also a $\kcNotation{k}{c}{\fieldCharacter}$ vector for $\Sigma$. 
From uniqueness of the $\kcNotation{k}{c}{\fieldCharacter}$ vector, it follows that there exists a character $\chi_{\Sigma} \colon \multiplicativegroup{\finiteField} \rightarrow \multiplicativegroup{\cComplex}$ such that $\Sigma\left(\diag^k\left(h\right)\right) v = \chi_{\Sigma}\left(\det h\right) v.$ When $c=1$, we have that $\chi_{\Sigma} = \centralCharacter{\Sigma}$ is the central character of $\Sigma$. It follows that the $\kcNotation{k}{c}{\fieldCharacter}$ model $\Whittaker\left(\Sigma, \fieldCharacterkc{k}{c}\right)$ is actually a subspace of $\Ind{\diag^k\left(\GL_c\left(\finiteField\right)\right) \ltimes \UnipotentRadicalForWss{k}{c}}{\GL_{kc}\left(\finiteField\right)}{\chi_{\Sigma}^{\diag^k\left(\GL_c\right)} \otimes \fieldCharacterkc{k}{c}}$.

\subsubsection{Bessel--Speh function}\label{sec:Bessel--Speh-function}
By realizing $\Sigma$ via its $\kcNotation{k}{c}{\fieldCharacter}$ model, any $\kcNotation{k}{c}{\fieldCharacter}$ vector corresponds to a function $0 \ne W \in \Whittaker\left(\Sigma, \fieldCharacterkc{k}{c} \right)$, such that $W\left(gu\right) = \fieldCharacterkc{k}{c}\left(u\right) W\left(g\right)$ for any $u \in \UnipotentRadicalForWss{k}{c}$ and any $g \in \GL_{kc}\left(\finiteField\right)$. 
As we will see below, for such $W$ we have $W\left(\IdentityMatrix{kc}\right) \ne 0$, and we denote by $\gbesselSpehFunction{\Sigma}{\fieldCharacter}$ the unique $W$ as such with $W\left(\IdentityMatrix{kc}\right) = 1$. 
We call $\gbesselSpehFunction{\Sigma}{\fieldCharacter}$ the $\left(k,c\right)$ \emph{normalized Bessel--Speh function of $\Sigma$ with respect to $\fieldCharacter$}. 
To summarize, $\gbesselSpehFunction{\Sigma}{\fieldCharacter} \in \Whittaker\left(\Sigma, \fieldCharacterkc{k}{c}\right)$ is the unique element such that:
\begin{enumerate}
	\item $\gbesselSpehFunction{\Sigma}{\fieldCharacter}\left(gu\right) = \fieldCharacterkc{k}{c}\left(u\right)\gbesselSpehFunction{\Sigma}{\fieldCharacter}\left(g\right)$ for any $g \in \GL_{kc}\left(\finiteField\right)$ and $u \in \UnipotentRadicalForWss{k}{c}$.
	\item $\gbesselSpehFunction{\Sigma}{\fieldCharacter}\left(\IdentityMatrix{kc}\right) = 1$.
\end{enumerate}

Explicitly, by choosing an invariant inner product $\innerproduct{\cdot}{\cdot}_{\Sigma}$ on $\Sigma$ and a $\kcNotation{k}{c}{\fieldCharacter}$ vector $0 \ne v_{\Sigma, \fieldCharacterkc{k}{c}} \in \Sigma$, the $\kcNotation{k}{c}{\fieldCharacter}$ model is given by $$\Whittaker\left(\Sigma, \fieldCharacterkc{k}{c}\right) =  \left\{g \mapsto\innerproduct{\Sigma\left(g\right)v}{v_{\Sigma, \fieldCharacterkc{k}{c}}}_{\Sigma} \mid v \in \Sigma \right\}.$$ The Bessel--Speh function is the matrix coefficient of $\Sigma$ given by the formula $$\gbesselSpehFunction{\Sigma}{\fieldCharacter}\left(g\right) = \frac{\innerproduct{\Sigma\left(g\right)v_{\Sigma, \fieldCharacterkc{k}{c}}}{v_{\Sigma, \fieldCharacterkc{k}{c}}}_{\Sigma}}{\innerproduct{v_{\Sigma, \fieldCharacterkc{k}{c}}}{v_{\Sigma, \fieldCharacterkc{k}{c}}}_{\Sigma}}.$$

It follows from this formula that for any $h \in \GL_c\left(\finiteField\right)$ and any $g \in \GL_{kc}\left(\finiteField\right)$, $$\gbesselSpehFunction{\Sigma}{\fieldCharacter}\left(\diag^k\left(h\right)g\right) = \gbesselSpehFunction{\Sigma}{\fieldCharacter}\left(g\,\diag^k\left(h\right)\right) = \chi_{\Sigma}\left(\det h\right) \gbesselSpehFunction{\Sigma}{\fieldCharacter}\left(g\right).$$
It also follows from this formula that for any $g \in \GL_{kc}\left(\finiteField\right),$
$$ \gbesselSpehFunction{\Sigma}{\fieldCharacter}\left(g^{-1}\right) = \gbesselSpehFunction{\Contragradient{\Sigma}}{\fieldCharacter^{-1}}\left(g\right) = \conjugate{ \gbesselSpehFunction{\Sigma}{\fieldCharacter}\left(g\right)}.$$

\begin{proposition}\label{prop:bessel-speh-on-diag-block-matrices}
	If $g \in \GL_c\left(\finiteField\right)$ is such that $\gbesselSpehFunction{\Sigma}{\fieldCharacter}\left(\diag\left(g, \IdentityMatrix{\left(k-1\right)c}\right)\right) \ne 0$ then $g = \IdentityMatrix{c}$.
\end{proposition}
\begin{proof}
	For any $X \in \squareMatrix_c\left(\finiteField\right)$ let $$u_X = \diag\left( \begin{pmatrix}
		\IdentityMatrix{c} & X\\
		& \IdentityMatrix{c}
	\end{pmatrix}, \IdentityMatrix{\left(k-2\right)c} \right).$$
	Then $$\diag\left(g, \IdentityMatrix{\left(k-1\right)c}\right) u_X = u_{g X} \cdot \diag\left(g, \IdentityMatrix{\left(k-1\right)c}\right).$$
	By applying $\gbesselSpehFunction{\Sigma}{\fieldCharacter}$ to both sides of the last equality and using the $\UnipotentRadicalForWss{k}{c}$ equivariance properties, we get that $$\fieldCharacter\left(\trace X\right) \gbesselSpehFunction{\Sigma}{\fieldCharacter}\left(\diag\left(g, \IdentityMatrix{\left(k-1\right)c}\right)\right) = \fieldCharacter\left(\trace \left(gX\right)\right) \gbesselSpehFunction{\Sigma}{\fieldCharacter}\left(\diag\left(g, \IdentityMatrix{\left(k-1\right)c}\right)\right).$$
	Therefore if $\gbesselSpehFunction{\Sigma}{\fieldCharacter}\left(\diag\left(g, \IdentityMatrix{\left(k-1\right)c}\right)\right) \ne 0$, we must have $\fieldCharacter\left(\trace\left(\left(g - \IdentityMatrix{c}\right)X\right)\right) = 1$ for every $X \in \squareMatrix_c\left(\finiteField\right)$, which implies $g - \IdentityMatrix{c} = 0$.
\end{proof}

When $c = 1$, we call $\gbesselSpehFunction{\Sigma}{\fieldCharacter}$ the \emph{normalized Bessel function of $\Sigma$ with respect to $\fieldCharacter$} and denote it by $\besselFunction_{\Sigma, \fieldCharacter}$.

\subsection{Speh representations}
In this section, we define an important class of irreducible representations of $\GL_{kc}\left(\finiteField\right)$ called Speh representations. These are also known in the literature as ``generalized trivial representations''. These can be thought of as finite field analogs of Speh representations introduced by Jacquet in~\cite[Section 2.2]{Jacquet1984}.

Let $\tau$ be an irreducible cuspidal representation of $\GL_k\left(\finiteField\right)$. Denote $\SpehRepresentation{\tau}{1} = \tau$. When $c > 1$, the parabolic induction $\tau^{\circ c} = \tau \circ \dots \circ \tau$, a representation of $\GL_{kc}\left(\finiteField\right)$, is not irreducible. The irreducible subrepresentations of $\tau^{\circ c}$ are in bijection with the irreducible representations of the symmetric group $\SymmetricGroup_c$, which in turn are in bijection with partitions of $c$. There are two irreducible subrepresentations of $\tau^{\circ c}$ that appear with multiplicity one. These are the Steinberg representation $\Steinberg\left(\tau, c\right)$, which corresponds to the partition $\left(c\right)$ and the Speh representation $\SpehRepresentation{\tau}{c}$, which corresponds to the partition $\left(1^c\right)$. The Steinberg representation $\Steinberg\left(\tau, c\right)$ is the irreducible subrepresentation of $\tau^{\circ c}$ with maximal dimension, while the Speh representation $\SpehRepresentation{\tau}{c}$ is the irreducible subrepresentation of $\tau^{\circ c}$ with minimal dimension (to show this, use the fact that the dimension of the irreducible subrepresentation of $\tau^{\circ}$ corresponding to the partition $\lambda \vdash c$ is a scalar multiple of the unipotent representation of $\GL_{c}\left(\finiteField_k\right)$ corresponding to $\lambda$ (see~\cite[Corollary 5.3 (ii)]{Howe1985} or~\cite[Lemma 7.4]{Green55}), and then use Corollary 3.4 or Proposition 3.5 of~\cite{larsen2013largest}).

In the special case when $k = 1$, we have that $\tau = \chi \colon \multiplicativegroup{\finiteField} \to \multiplicativegroup{\cComplex}$ is a character and $\SpehRepresentation{\tau}{c} = \chi_{\GL_c}$, where $\chi_{\GL_c} \colon \GL_c\left(\finiteField\right) \to \multiplicativegroup{\cComplex}$ is the character $$\chi_{\GL_c}\left(g\right) = \chi\left(\det g\right).$$

In~\cite[Theorem 6.4]{Carmon2023}, the first author proved that for an irreducible cuspidal representation $\tau$ of $\GL_k\left(\finiteField\right)$, the Speh representation $\SpehRepresentation{\tau}{c}$ is of $\left(k,c\right)$ type. We will use this fact to define the notion of a Speh representation for an irreducible generic representation of $\GL_{k}\left(\finiteField\right)$.

Let $\tau$ be an irreducible generic representation of $\GL_k\left(\finiteField\right)$. In this case, there exist irreducible cuspidal representations $\tau_1, \dots, \tau_s$ of $\GL_{k_1}\left(\finiteField\right), \dots, \GL_{k_s}\left(\finiteField\right)$, respectively, with $k_1 + \dots + k_s =k$, such that $\tau$ is the unique irreducible generic subrepresentation of the parabolic induction $\tau_1 \circ \dots \circ \tau_s$. Then the parabolic induction $\SpehRepresentation{\tau_1}{c} \circ \dots \circ \SpehRepresentation{\tau_s}{c}$ is a representation of $\left(k,c\right)$ type, and therefore there exists a unique irreducible subrepresentation $\Sigma$ of  $\SpehRepresentation{\tau_1}{c} \circ \dots \circ \SpehRepresentation{\tau_s}{c}$, such that $\Sigma$ admits a $\kcNotation{k}{c}{\fieldCharacter}$ vector. We denote $\SpehRepresentation{\tau}{c} = \Sigma$.
In~\cite[Proposition 6.8]{Carmon2023}, the first author showed that all irreducible representations of $\GL_{kc}(\finiteField)$ of $(k,c)$ type arise in this way.

We may also describe the parameter of the representation $\SpehRepresentation{\tau}{c}$ in terms of the cuspidal support of $\tau$. Since $\tau$ is an irreducible generic representation, its parameter $\varphi_{\tau}$ has the following form. If $\sigma_1, \dots, \sigma_r$ are all the different elements in the cuspidal support of $\tau$ and for every $j$, $\sigma_j$ is an irreducible cuspidal representation of $\GL_{d_j}\left(\finiteField\right)$ that appears $m_j$ times in the cuspidal support of $\tau$, then $\varphi_{\tau}$ is the parameter satisfying $\varphi_{\tau}\left(d_j, \sigma_j\right) = (m_j)$ for every $j$, and $\varphi_{\tau}\left(d', \sigma'\right) = ()$ for any other pair $(d', \sigma')$. Then the parameter $\varphi_{\SpehRepresentation{\tau}{c}}$ of $\SpehRepresentation{\tau}{c}$ is the parameter satisfying $$\varphi_{\SpehRepresentation{\tau}{c}}\left(d_j, \sigma_j\right) = (m_j^c) = (m_j,\dots,m_j),$$
for every $j$ and $\varphi_{\SpehRepresentation{\tau}{c}}\left(d', \sigma'\right) = ()$ for any other pair.

It turns out that if $\tau$ is a generic representation and $c = c_1 + c_2$, then $\SpehRepresentation{\tau}{c}$ is a subrepresentation of the parabolic induction $\SpehRepresentation{\tau}{c_1} \circ \SpehRepresentation{\tau}{c_2}$. It also turns out that $\Contragradient{\SpehRepresentation{\tau}{c}}$ is isomorphic to $\SpehRepresentation{\Contragradient{\tau}}{c}$.

In~\cite[Proposition 6.21]{Carmon2023}, the first author showed that for an irreducible generic representation $\tau$, the character $\chi_{\SpehRepresentation{\tau}{c}}$ discussed in \Cref{sec:Bessel--Speh-function} is given by $\chi_{\SpehRepresentation{\tau}{c}} = \centralCharacter{\tau}$, where $\centralCharacter{\tau}$ is the central character of $\tau$ (this result is analogous to~\cite[Lemma 12]{CaiFriedbergGourevitchKaplan2023}).
Thus, we may consider the $\kcNotation{k}{c}{\fieldCharacter}$ model $\Whittaker\left(\SpehRepresentation{\tau}{c}, \fieldCharacterkc{k}{c}\right)$ as a subrepresentation of $\Ind{\diag^k\left(\GL_c(\finiteField)\right)\ltimes\UnipotentRadicalForWss{k}{c}}{\GL_{kc}\left(\finiteField\right)}{\centralCharacter{\tau}^{\diag^k\left(\GL_c\right)}\otimes\fieldCharacterkc{k}{c}}$.

\subsection{Inductive construction of \texorpdfstring{$\kcNotation{k}{c}{\fieldCharacter}$}{(k,c)} functionals of Speh representations}\label{sec:wss-models}
Let $\tau$ be an irreducible generic representation of $\GL_k\left(\finiteField\right)$. In this section we explain how to construct a $\kcNotation{k}{c}{\fieldCharacter}$ functional for $\SpehRepresentation{\tau}{c}$, given that we have constructed such functionals for smaller $c$.

Suppose that $c = c_1 + c_2$, where $c_1, c_2 > 0$. For $j=1,2$, let $\gShortSpehWhittakerFunctional{\tau}{k}{c_j} = \gSpehWhittakerFunctional{\tau}{k}{c_j}$ be a (non-zero) $\kcNotation{k}{c_j}{\fieldCharacter}$ functional of $\SpehRepresentation{\tau}{c_j}$.

Let $\kappa \in \SymmetricGroup_{kc}$ be the permutation $$\kappa\left(1 + a + b c\right) = \begin{cases}
	1 + a + b c_1 & 0 \le a \le c_1 - 1,\\
	1+ a + c_1 k + b c_2 & c_1 \le a \le c_1 + c_2 - 1,
\end{cases}$$
where $0 \le a \le c-1$ and $0 \le b \le k-1$.
Also denote by $\kappa$ the column permutation matrix corresponding to $\kappa$, that is
$$ \kappa = \begin{pmatrix}
	\transpose{e_{\kappa\left(1\right)}} & \dots & \transpose{e_{\kappa\left(kc\right)}}
\end{pmatrix} = \begin{pmatrix}
	\IdentityMatrix{c_1} & 0\\
	0 & 0 & \IdentityMatrix{c_1} & 0\\
	\vdots & \vdots & &\ddots & \ddots\\
	0 & 0 & \cdots & 0 & 0 & \IdentityMatrix{c_1} & 0\\
	0 & \IdentityMatrix{c_2}\\
	0 & 0 & 0 & \IdentityMatrix{c_2}\\
	\vdots & \vdots & & \ddots & \ddots\\
	0 & 0 & \cdots & 0 & 0 & 0 & \IdentityMatrix{c_2}
\end{pmatrix},$$
where for every $1 \le j \le kc$, $\transpose{e_j}$ is the $j$th standard column vector. We will often use the following identity for various choices of $Z_{ij} = \left(\begin{smallmatrix}
	A_{ij} & B_{ij}\\
	C_{ij} & D_{ij}
\end{smallmatrix}\right)$, where for every $1 \le i, j \le k$, $A_{ij} \in \squareMatrix_{c_1}\left(\finiteField\right)$, $B_{ij} \in \Mat{c_1}{c_2}\left(\finiteField\right)$, $C_{ij} \in \Mat{c_2}{c_1}\left(\finiteField\right)$, $D_{ij} \in \squareMatrix_{c_2}\left(\finiteField\right)$:
$$\begin{pmatrix}
	A_{11} & \dots  & A_{1k} & B_{11} & \dots & B_{1k}  \\
	\vdots & \ddots & \vdots & \vdots & \ddots & \vdots \\
	A_{k1} & \dots  & A_{kk} & B_{k1} & \dots & B_{kk} \\
	C_{11} & \dots  & C_{1k} & D_{11} & \dots & D_{1k}  \\
\vdots & \ddots & \vdots & \vdots & \ddots & \vdots \\
C_{k1} & \dots  & C_{kk} & D_{k1} & \dots & D_{kk} \end{pmatrix} \kappa = \kappa \begin{pmatrix}
	Z_{11} & \dots  & Z_{1k} \\
	\vdots & \ddots & \vdots \\
	Z_{k1} & \dots  & Z_{kk} \end{pmatrix}.$$

Let \begin{equation}\label{eq:Y-R-subgroup}
	\mathcal{Y} = \left\{ \begin{pmatrix}
		\IdentityMatrix{k c_1}\\
		Y & \IdentityMatrix{k c_2} \end{pmatrix} \mid Y = \begin{pmatrix}
		0& y_{12} & y_{13} & \dots & y_{1k}\\
		&0 & y_{23} & \dots & y_{2k}\\
		& & \ddots & \ddots & \vdots\\
		& & & 0 & y_{k-1,k}\\
		& & & & 0
	\end{pmatrix} \text{ where } y_{ij} \in \Mat{c_2}{c_1}\left(\finiteField\right) \right\}.
\end{equation}

By realizing $\SpehRepresentation{\tau}{c}$ as a subrepresentation of the parabolic induction $\SpehRepresentation{\tau}{c_1} \circ \SpehRepresentation{\tau}{c_2}$, we form a recursive process to obtain a $\kcNotation{k}{c}{\fieldCharacter}$ functional $\gSpehWhittakerFunctional{\tau}{k}{c} = \gShortSpehWhittakerFunctional{\tau}{k}{c} \colon \SpehRepresentation{\tau}{c_1} \circ \SpehRepresentation{\tau}{c_2} \to \cComplex$, given by the recursive formula
$$ \standardForm{\Phi}{\gShortSpehWhittakerFunctional{\tau}{k}{c}} = \frac{1}{\sizeof{\mathcal{Y}}} \sum_{y \in \mathcal{Y}} \standardForm{\Phi\left(  y \kappa \right)}{\gShortSpehWhittakerFunctional{\tau}{k}{c_1} \otimes \gShortSpehWhittakerFunctional{\tau}{k}{c_2}},$$
where $\Phi \in \SpehRepresentation{\tau}{c_1} \circ \SpehRepresentation{\tau}{c_2}$.

As in~\cite[Lemma 9]{CaiFriedbergGourevitchKaplan2023}, it can be shown using the method of root elimination that $\gShortSpehWhittakerFunctional{\tau}{k}{c}$ is non-zero on any $\GL_{kc}\left(\finiteField\right)$-invariant subspace of $\SpehRepresentation{\tau}{c_1} \circ \SpehRepresentation{\tau}{c_2}$, and it follows that it is a non-zero $\kcNotation{k}{c}{\fieldCharacter}$ functional for $\SpehRepresentation{\tau}{c}$.

\section{Gamma factors}\label{sec:gamma-factors}

In this section, we use the Bessel--Speh function from \Cref{sec:Bessel--Speh-function} to define a gamma factor, which we call the \emph{Ginzburg--Kaplan gamma factor}. We prove that this gamma factor is multiplicative in both arguments, and prove that it satisfies a functional equation. We begin with recalling the definition and properties of the Godement--Jacquet gamma factor, since it will be serve as the definition of the Ginzburg--Kaplan gamma factor for the $k = 1$ case.

\subsection{Godement--Jacquet gamma factors}\label{sec:godement-jacquet-gamma-factors}

In this section, we recall the definition and properties of the gamma factor due to Kondo and Godement--Jacquet.

\subsubsection{Non-abelian Gauss sums}
Let $\pi$ be a representation of $\GL_c\left(\finiteField\right)$. 
Consider the operator $\GaussSum{\pi}{\fieldCharacter} \colon \pi \to \pi$ defined by
$$ \GaussSum{\pi}{\fieldCharacter} = q^{-\tfrac{c^2}{2}} \sum_{g \in \GL_c\left(\finiteField\right)} \fieldCharacter\left(\trace g^{-1}\right) \pi\left(g\right).$$
Then $\GaussSum{\pi}{\fieldCharacter} \in \Hom_{\GL_c\left(\finiteField\right)}\left(\pi, \pi\right)$. 
If $\pi$ is irreducible, then by Schur's lemma there exists a scalar $\GJPreGammaFactor{\pi}{\fieldCharacter} \in \cComplex$, such that $\GaussSum{\pi}{\fieldCharacter} = \GJPreGammaFactor{\pi}{\fieldCharacter} \cdot \idmap_\pi$. 

We have the following multiplicativity property.
\begin{theorem}\label{thm:multiplicativity-of-godement-jacquet}
	Let $c = c_1 + c_2$. Suppose that $\pi_1$ and $\pi_2$ are irreducible representations of $\GL_{c_1}\left(\finiteField\right)$ and $\GL_{c_2}\left(\finiteField\right)$, respectively, and suppose that $\pi$ is an irreducible representation of $\GL_{c}\left(\finiteField\right)$, such that $\pi$ is a subrepresentation of the parabolic induction $\pi_1 \circ \pi_2$. 
 Then $$\GJPreGammaFactor{\pi}{\fieldCharacter} = \GJPreGammaFactor{\pi_1}{\fieldCharacter} \GJPreGammaFactor{\pi_2}{\fieldCharacter}.$$ 
\end{theorem}
\begin{proof}
	Denote $\pi' = \pi_1 \circ \pi_2$.
	Let $w_{\pi_1} \in \pi_1$ and $w_{\pi_2} \in \pi_2$ and let $\phi_{w_{\pi_1} \otimes w_{\pi_2}} \in \pi'$ be the function supported on $\ParabolicSubgroup_{\left(c_1, c_2\right)}$, such that $\phi_{w_{\pi_1} \otimes w_{\pi_2}} \left(\IdentityMatrix{c}\right) = w_{\pi_1} \otimes w_{\pi_2}$. 

We first show that the projection of $\phi_{w_{\pi_1} \otimes w_{\pi_2}}$ to any non-zero subrepresentation $V$ of $\pi_1 \circ \pi_2$ is non-zero.
Let $0\neq f \in V$ and let $g_0\in\GL_{c}(\finiteField)$ be such that $f(g_0)\neq 0$.
By replacing $f$ with $\pi'(g_0)f\in V$, we may assume $g_0=\IdentityMatrix{c}$.
Since $\pi_1\otimes \pi_2$ is an irreducible representation of $\GL_{c_1}(\finiteField)\times \GL_{c_2}(\finiteField)$, we have that
\begin{equation*}
	\pi_1\otimes \pi_2 = \Span_{\cComplex}\left\{ \pi_1(g_1)\otimes \pi_2(g_2) f(\IdentityMatrix{c}) \mid g_i \in \GL_{c_i}(\finiteField)  \right\}.
\end{equation*}
It therefore follows that the set
\begin{equation*}
	\Span_{\cComplex}\left\{  \pi'\begin{pmatrix}
		g_1 &\\ & g_2
	\end{pmatrix} f \mid g_i \in \GL_{c_i}(\finiteField) \right\}\subset V
\end{equation*}
contains an element $f'$ such that $f'(\IdentityMatrix{c})=w_{\pi_1}\otimes w_{\pi_2}$.
Fix a non-zero $\ParabolicSubgroup_{(c_1,c_2)}$-invariant inner product $\innerproduct{\cdot}{\cdot}_{\pi_1\otimes \pi_2}$ on $\pi_1 \overline{\otimes} \pi_2$ and let $\innerproduct{\cdot}{\cdot}_{\pi'}$ be the non-zero $\GL_{c}\left(\finiteField\right)$-invariant inner product on $\pi'$ given by
\begin{equation*}
	\innerproduct{f_1}{f_2}_{\pi'}=\sum_{g \in \ParabolicSubgroup_{(c_1,c_2)}\backslash \GL_{c}\left(\finiteField\right)} \innerproduct{f_1\left(g\right)}{f_2\left(g\right)}_{\pi_1 \otimes \pi_2}.
\end{equation*}
In particular, we have that $\innerproduct{\phi_{w_{\pi_1} \otimes w_{\pi_2}}}{f'}_{\pi'}$ is equal to
\begin{equation*}
	\innerproduct{\phi_{w_{\pi_1} \otimes w_{\pi_2}}(\IdentityMatrix{c})}{f'(\IdentityMatrix{c})}_{\pi_1 \otimes \pi_2} = 
	\innerproduct{w_{\pi_1}\otimes w_{\pi_2}}{w_{\pi_1}\otimes w_{\pi_2}}_{\pi_1 \otimes \pi_2}
\end{equation*}
and thus is non-zero.

We now proceed to prove the multiplicativity property in the statement of the theorem.
Since $\phi_{w_{\pi_1} \otimes w_{\pi_2}}$ is right invariant under $\UnipotentRadical_{(c_1, c_2)}$, we can write
\begin{equation} \label{eq:gj-zeta-operator-applied-to-full-section}
	\begin{split}
		&\sum_{h \in \GL_c\left(\finiteField\right)} \fieldCharacter\left(\trace h^{-1}\right) \pi'\left(h\right) \phi_{w_{\pi_1} \otimes w_{\pi_2}} = \frac{1}{\sizeof{\UnipotentRadical_{\left(c_1, c_2\right)}}}\sum_{h \in \GL_c\left(\finiteField\right)} \sum_{n \in \UnipotentRadical_{\left(c_1, c_2\right)}}  \fieldCharacter\left(\trace \left(nh^{-1}\right) \right) \pi'\left(h\right) \phi_{w_{\pi_1} \otimes w_{\pi_2}}.
	\end{split}
\end{equation}
Writing $n = \left(\begin{smallmatrix}
	\IdentityMatrix{c_1} & X\\
	& \IdentityMatrix{c_2}
\end{smallmatrix}\right)$ and $h^{-1} = \left(\begin{smallmatrix}
h'_{c_1 \times c_1} & h'_{c_1 \times c_2}\\
h'_{c_2 \times c_1} & h'_{c_2 \times c_2}
\end{smallmatrix}\right)$, we have that $$\fieldCharacter\left(\trace\left(nh^{-1} \right)\right) = \fieldCharacter\left(\trace\left(h'_{c_1 \times c_1}\right) + \trace\left(h'_{c_2 \times c_2}\right)\right) \fieldCharacter\left(\trace\left(X h'_{c_2 \times c_1} \right)\right),$$ and therefore the inner sum over $n \in \UnipotentRadical_{(c_1, c_2)}$ in \eqref{eq:gj-zeta-operator-applied-to-full-section} vanishes unless $h'_{c_2 \times c_1} = 0$, which is equivalent to $h \in \ParabolicSubgroup_{(c_1, c_2)}$. This implies that $$\GaussSum{\pi'}{\fieldCharacter} \phi_{w_{\pi_1} \otimes w_{\pi_2}} = q^{-\frac{\left(c_1 + c_2\right)^2}{2}} \sum_{p \in \ParabolicSubgroup_{(c_1, c_2)}} \pi'\left(p\right) \fieldCharacter\left(\trace p^{-1}\right) \phi_{w_{\pi_1} \otimes w_{\pi_2}}.$$ It follows that $\GaussSum{\pi'}{\fieldCharacter} \phi_{w_{\pi_1} \otimes w_{\pi_2}}$ is also supported on $\ParabolicSubgroup_{(c_1, c_2)}$. To determine $\GaussSum{\pi'}{\fieldCharacter} \phi_{w_{\pi_1} \otimes w_{\pi_2}}$, it suffices to compute its value at $\IdentityMatrix{c}$. A direct computation yields
$$\left(\GaussSum{\pi'}{\fieldCharacter} \phi_{w_{\pi_1} \otimes w_{\pi_2}}\right)\left(\IdentityMatrix{c}\right) = \GaussSum{\pi_1}{\fieldCharacter} w_{\pi_1} \otimes \GaussSum{\pi_2}{\fieldCharacter} w_{\pi_2}.$$
Hence we have that \begin{equation}\label{eq:gj-final-multiplicativity-equation}
	\GaussSum{\pi'}{\fieldCharacter} \phi_{w_{\pi_1} \otimes w_{\pi_2}} = \GJPreGammaFactor{\pi_1}{\fieldCharacter} \GJPreGammaFactor{\pi_2}{\fieldCharacter} \phi_{w_{\pi_1} \otimes w_{\pi_2}}.
\end{equation}
To complete the proof, decompose $\pi_1 \circ \pi_2 = \bigoplus_j \sigma_j$, where $\sigma_j$ is an irreducible invariant subspace of $\pi_1 \circ \pi_2$ and the $\sigma_j$'s need not be non-isomorphic. Write $$\phi_{w_{\pi_1} \otimes w_{\pi_2}} = \sum_{j} \phi_{w_{\pi_1} \otimes w_{\pi_2}}^j,$$
where $\phi_{w_{\pi_1} \otimes w_{\pi_2}}^j \in \sigma_j$ (recall that $\phi_{w_{\pi_1} \otimes w_{\pi_2}}^j \ne 0$ since the projection of $\phi_{w_{\pi_1} \otimes w_{\pi_2}}$ to any irreducible subspace of $\pi'$ is non-zero).  Then \eqref{eq:gj-final-multiplicativity-equation} implies that \begin{equation*}
 \GJPreGammaFactor{\sigma_j}{\fieldCharacter} \phi_{w_{\pi_1} \otimes w_{\pi_2}}^j = \GJPreGammaFactor{\pi_1}{\fieldCharacter} \GJPreGammaFactor{\pi_2}{\fieldCharacter} \phi^j_{w_{\pi_1} \otimes w_{\pi_2}},
\end{equation*}
which implies that for every $j$ $$\GJPreGammaFactor{\sigma_j}{\fieldCharacter} = \GJPreGammaFactor{\pi_1}{\fieldCharacter}\GJPreGammaFactor{\pi_2}{\fieldCharacter},$$
as required.
\end{proof}

\begin{remark}
	The constant $\GJPreGammaFactor{\pi}{\fieldCharacter}$ was computed explicitly by Kondo in~\cite{Kondo1963}. \Cref{thm:multiplicativity-of-godement-jacquet} follows immediately from his computation. However, we included this proof of \Cref{thm:multiplicativity-of-godement-jacquet} since it is similar to the final step of the proof of \Cref{thm:multiplicativitiy-in-first-variable} that we will give later. One can give an alternative proof of Kondo's theorem by computing $\GJPreGammaFactor{\pi}{\fieldCharacter}$ for an irreducible cuspidal representation $\pi$ (which is considerably easier than for general $\pi$) and by using \Cref{thm:multiplicativity-of-godement-jacquet}.
\end{remark}

\subsubsection{Godement--Jacquet functional equation}\label{sec:godement-jacquet-functional-equation}

Let $\mathcal{S}\left( \squareMatrix_c \left(\finiteField\right)\right) = \left\{f \colon \squareMatrix_c\left(\finiteField\right) \to \cComplex \right\}.$ 
For a function $f \in \mathcal{S}\left(\squareMatrix_c\left(\finiteField\right)\right)$ we define $\transpose{f}\left(X\right) = f\left(\transpose{X}\right)$. We define the Fourier transform of $f \in \mathcal{S}\left(\squareMatrix_c\left(\finiteField\right)\right)$ by the formula
$$ \fourierTransform{\fieldCharacter}{f}\left(X\right) = q^{-\tfrac{c^2}{2}} \sum_{Y \in \squareMatrix_c\left(\finiteField\right)} f\left(Y\right) \fieldCharacter\left(\trace\left(XY\right)\right).$$
The Fourier inversion formula has the following form:
\begin{align*}
	\fourierTransform{\fieldCharacter^{-1}}{\fourierTransform{\fieldCharacter}{f}}\left(X\right) &= f\left(X\right),\\
	\fourierTransform{\fieldCharacter}{\fourierTransform{\fieldCharacter}{f}}\left(X\right) &= f\left(-X\right).
\end{align*}

Let $\pi$ be an irreducible representation of $\GL_c\left(\finiteField\right)$. 
For a function $f \in \mathcal{S}\left(\squareMatrix_c\left(\finiteField\right)\right)$ we define the following \emph{Godement--Jacquet zeta operator}. It is a finite field analog of the Godement--Jacquet zeta integral~\cite{GodementJacquet1972}. 
$$\zetaOperator\left(f,\pi\right) = \sum_{g \in \GL_c \left(\finiteField\right)} f\left(g\right) \pi\left(g\right).$$
Macdonald proved the following finite field version of the functional equation of the Godement--Jacquet zeta operator:
\begin{theorem}[{\cite[(2.7)]{Macdonald80}}]
	Suppose that the trivial character $1$ of $\GL_1\left(\finiteField\right) = \multiplicativegroup{\finiteField}$ does not appear in the cuspidal support of $\pi$. 
 Then for every $f \in \mathcal{S}\left(\squareMatrix_c\left(\finiteField\right)\right)$ the following functional equation holds: $$\zetaOperator\left(\transpose{\fourierTransform{\fieldCharacter}{f}}, \Contragradient{\pi}\right) = \GJPreGammaFactor{\pi}{\fieldCharacter} \zetaOperator\left(f, \pi\right).$$
	Here, we realize the action of $\Contragradient{\pi}$ on the space of $\pi$ by $\Contragradient{\pi}\left(g\right) = \pi\left(\inverseTranspose{g}\right),$ where $\inverseTranspose{g} = \transpose{g}^{-1}$. \end{theorem}

\subsubsection{Twisting by a character}\label{subsubsection-twisting-by-a-character}
Let $\chi \colon \multiplicativegroup{\finiteField} \to \multiplicativegroup{\cComplex}$ be a character. 
We define $$\GKGaussSum{\pi}{\chi}{\fieldCharacter} = \GaussSum{\pi \otimes \chi_{\GL_c}}{\fieldCharacter} = \sum_{g \in \GL_c\left(\finiteField\right)} \chi\left(\det g\right) \fieldCharacter\left(\trace g^{-1}\right) \pi\left(g\right).$$ 
We also define $\gGJPreGammaFactor{\pi}{\chi}{\fieldCharacter} = \GJPreGammaFactor{\pi \otimes \chi_{\GL_c}}{\fieldCharacter}$. 
For a function $f \in \mathcal{S}\left(\squareMatrix_c\left(\finiteField\right)\right)$, consider the \emph{twisted Godement--Jacquet zeta operator} $$\zetaOperator\left(f, \pi \times \chi\right) = \sum_{g \in \GL_c\left(\finiteField\right)} f\left(g\right) \chi\left(\det g\right) \pi\left(g\right).$$

Immediately from the properties of $\GJPreGammaFactor{\pi}{\fieldCharacter}$ we obtain analogous properties for $\gGJPreGammaFactor{\pi}{\chi}{\fieldCharacter}$.
\begin{theorem}
	Suppose that $\pi_1$ and $\pi_2$ are irreducible representations of $\GL_{c_1}\left(\finiteField\right)$ and $\GL_{c_2}\left(\finiteField\right)$, respectively, and suppose that $\pi$ is an irreducible representation of $\GL_{c_1 + c_2}\left(\finiteField\right)$, such that $\pi$ is a subrepresentation of the parabolic induction $\pi_1 \circ \pi_2$. 
 Then $$\gGJPreGammaFactor{\pi}{\chi}{\fieldCharacter} = \gGJPreGammaFactor{\pi_1}{\chi}{\fieldCharacter} \gGJPreGammaFactor{\pi_2}{\chi}{\fieldCharacter}.$$ 
\end{theorem}
\begin{theorem}
	Suppose that $\chi^{-1}$ does not appear in the cuspidal support of $\pi$. 
 Then for any $f \in \mathcal{S}\left(\squareMatrix_c\left(\finiteField\right)\right)$ we have
	$$\zetaOperator\left(\transpose{\fourierTransform{\fieldCharacter}{f}}, \pi^{\vee} \times \chi^{-1} \right) = \gGJPreGammaFactor{\pi}{\chi}{\fieldCharacter} Z \left(f, \pi \times \chi\right).$$
\end{theorem}

\subsection{Ginzburg--Kaplan gamma factors}

In this section, we define the Ginzburg--Kaplan gamma factor. It is a finite field analog of a gamma factor defined by Kaplan in the local field case~\cite[Appendix A]{kaplan2018}, which in turn has a global counterpart defined by Ginzburg~\cite{ginzburg2019tensor}. We prove the multiplicativity properties this gamma factor satisfies.

\subsubsection{Non-abelian Gauss sums}\label{subsec:gk-non-abelian-gauss-sums}
Let $\tau$ be an irreducible generic representation of $\GL_k\left(\finiteField\right)$. Let $c \ge 1$ and consider the Speh representation $\SpehRepresentation{\tau}{c}$. 
As explained in \Cref{sec:wss-models}, the Bessel--Speh function $\besselSpehFunction{\tau}{c}\left(g\right)$ satisfies for any $h \in \GL_c\left(\finiteField\right)$ and any $g \in \GL_{kc}\left(\finiteField\right)$, $$\besselSpehFunction{\tau}{c}\left(g\,\diag^k\left(h\right)\right) = \besselSpehFunction{\tau}{c}\left(\diag^k\left(h\right) g \right) = \centralCharacter{\tau}\left(\det h\right) \besselSpehFunction{\tau}{\fieldCharacter}\left(g\right).$$
It follows that the assignment $\GL_c\left(\finiteField\right) \to \cComplex$ given by \begin{equation}\label{eq:bessel-speh-voronoi-assignment}
	h \mapsto \specialBesselSpeh{\tau}\left(h\right) \coloneqq \begin{dcases}
		\besselSpehFunction{\tau}{c}\begin{pmatrix}
			& \IdentityMatrix{\left(k-1\right)c}\\
			h
		\end{pmatrix} & k \ge 2,\\
		\tau\left(\det h\right) \fieldCharacter\left(\trace h^{-1}\right) & k = 1
	\end{dcases}
\end{equation} is invariant under $\GL_c\left(\finiteField\right)$ conjugation. We remark that the notation $\specialBesselSpeh{\tau}\left(h\right)$ does not include the value $c$, nor the value $k$. These values should be inferred from $h$ and $\tau$, respectively.

Let $\pi$ be a representation of $\GL_c\left(\finiteField\right)$. 
We define a linear operator $\pi \to \pi$ by$$\GKGaussSum{\pi}{\tau}{\fieldCharacter} = q^{\frac{\left(k-2\right)c^2}{2}} \sum_{h \in \GL_c\left(\finiteField\right)} \specialBesselSpeh{\tau}\left(h\right) \pi\left(h\right).$$
Since the assignment in \eqref{eq:bessel-speh-voronoi-assignment} is invariant under conjugation, we have that $\GKGaussSum{\pi}{\tau}{\fieldCharacter} \in \Hom_{\GL_c\left(\finiteField\right)}\left(\pi, \pi\right)$. 
If $\pi$ is irreducible, then by Schur's lemma there exists a scalar $\GKPreGammaFactor{\pi}{\tau}{\fieldCharacter} \in \cComplex$ such that $\GKGaussSum{\pi}{\tau}{\fieldCharacter} = \GKPreGammaFactor{\pi}{\tau}{\fieldCharacter} \cdot \idmap_\pi$. 
We define for any $k \ge 1,$ $$\GKGammaFactor{\pi}{\tau}{\fieldCharacter} = \centralCharacter{\pi}\left(-1\right)^{k-1} \GKPreGammaFactor{\pi}{\tau}{\fieldCharacter}.$$

Notice that if $k=1$ and $\tau$ is an irreducible representation of $\GL_1\left(\finiteField\right)$, then $\tau \colon \multiplicativegroup{\finiteField} \to \multiplicativegroup{\cComplex}$ is a character. In this case, we have $$\GKGammaFactor{\pi}{\tau}{\fieldCharacter} = \GKPreGammaFactor{\pi}{\tau}{\fieldCharacter} = \gGJPreGammaFactor{\pi}{\tau}{\fieldCharacter},$$ where $\gGJPreGammaFactor{\pi}{\tau}{\fieldCharacter}$ is as in \Cref{subsubsection-twisting-by-a-character}.

\subsubsection{Contragredient property}\label{subsec:ginzburg-kaplan-contragradient}

Recall that $\SpehRepresentation{\Contragradient{\tau}}{c}$ is isomorphic to $\Contragradient{\SpehRepresentation{\tau}{c}}$. We have that
$$\GKPreGammaFactor{\Contragradient{\pi}}{\Contragradient{\tau}}{\fieldCharacter^{-1}} = \frac{q^{\frac{\left(k-2\right)c^2}{2}}}{\dim \Contragradient{\pi}} \sum_{h \in \GL_c\left(\finiteField\right)} \specialBesselSpeh[\fieldCharacter^{-1}]{\Contragradient{\tau}}\left(h\right) \trace \Contragradient{\pi}\left(h\right).$$
Using that $\trace \Contragradient{\pi}\left(h\right) = \conjugate{\trace \pi\left(h\right)}$ and that $\gbesselSpehFunction{\Contragradient{\SpehRepresentation{\tau}{c}}}{\fieldCharacter^{-1}}\left(h\right) = \conjugate{\besselSpehFunction{\tau}{c}\left(h\right)},$ we get $$\GKPreGammaFactor{\Contragradient{\pi}}{\Contragradient{\tau}}{\fieldCharacter^{-1}} = \frac{q^{\frac{\left(k-2\right)c^2}{2}}}{\dim \pi} \sum_{h \in \GL_c\left(\finiteField\right)} \conjugate{\specialBesselSpeh{\tau}\left(h\right)} \cdot \conjugate{\trace \pi\left(c\right)} = \conjugate{\GKPreGammaFactor{\pi}{\tau}{\fieldCharacter}}.$$

\subsubsection{Multiplicativity property in the second argument}

The aim of this section is to prove the following multiplicativity property:
\begin{theorem}\label{thm:multiplicativitiy-in-second-variable}
	Let $\tau_1$ and $\tau_2$ be irreducible generic representations of $\GL_{k_1}\left(\finiteField\right)$ and $\GL_{k_2}\left(\finiteField\right)$, respectively. 
 Suppose that $k = k_1 + k_2$ and that $\tau$ is the unique irreducible generic subrepresentation of the parabolically induced representation $\tau_1 \circ \tau_2$. 
 Then $$\GKGammaFactor{\pi}{\tau}{\fieldCharacter} = \GKGammaFactor{\pi}{\tau_1}{\fieldCharacter} \cdot \GKGammaFactor{\pi}{\tau_2}{\fieldCharacter}.$$
\end{theorem}
The proof will occupy this section and will go through finding a recursive expression for a special value of the Bessel--Speh function for a parabolic induction of representations of types $\left(k_1, c\right)$ and $\left(k_2, c\right)$.

We start with the following lemma.

\begin{lemma}\label{lem:convolution-of-bessel-speh-functions}
	Let $\Sigma_1$ and $\Sigma_2$ be irreducible representations of types $\left(k_1, c\right)$ and $\left(k_2, c\right)$, respectively. 
 Let $k = k_1 + k_2$ and suppose that $\Sigma \subset \Sigma_1 \circ \Sigma_2$ is an irreducible representation admitting a $\kcNotation{k}{c}{\fieldCharacter}$ vector. 
 Then for any $h \in \GL_c\left(\finiteField\right)$,
	\begin{equation*}
		\specialBesselSpeh{\Sigma}\left(h\right) = q^{- c^2} \sum_{\substack{x,y \in \GL_c\left(\finiteField\right)\\
				xy = -h}} \specialBesselSpeh{\Sigma_1}\left(x\right) \specialBesselSpeh{\Sigma_2}\left(y\right).
	\end{equation*}
\end{lemma}

\begin{proof}
	Let $v_1$ and $v_2$ be a $\kcNotation{k_1}{c}{\fieldCharacter}$ vector and a $\kcNotation{k_2}{c}{\fieldCharacter}$ vector for $\Sigma_1$ and $\Sigma_2$, respectively.  
 A $\kcNotation{k}{c}{\fieldCharacter}$ vector for $\rho = \Sigma_1 \circ \Sigma_2$ is given by the formula
	$$f\left(g\right) = \begin{cases}
	\fieldCharacterkc{k}{c}\left(u\right) \left(\Sigma_1 \overline{\otimes} \Sigma_2\right)\left(p\right)  v_1 \otimes v_2 & g = p \begin{pmatrix}
			& \IdentityMatrix{k_1 c}\\
			\IdentityMatrix{k_2 c}
		\end{pmatrix} u, p \in \ParabolicSubgroup_{\left(k_1 c, k_2 c\right)}, u \in \UnipotentRadicalForWss{k}{c}, \\
		0 & \text{otherwise},
	\end{cases}$$
(see \Cref{sec: Parabolic induction} for the definition of $\overline{\otimes}$).
Since $\Sigma$ is the unique irreducible subrepresentation of  $\Sigma_1 \circ \Sigma_2$ admitting a $\kcNotation{k}{c}{\fieldCharacter}$ vector, we must have that $f \in \Sigma$.

Fix non-zero $\GL_{k_1 c}\left(\finiteField\right)$ and $\GL_{k_2 c}\left(\finiteField\right)$ invariant inner products on $\Sigma_1$ and on $\Sigma_2$, respectively, and let $\innerproduct{\cdot}{\cdot}_{\Sigma_1 \otimes \Sigma_2}$ be the invariant product they induce on the tensor product $\Sigma_1 \otimes \Sigma_2$. 
We may assume without loss of generality that $v_1$ and $v_2$ have norm one with respect to our choices of the inner products. 
A non-zero $\GL_{kc}\left(\finiteField\right)$ invariant inner product on $\rho = \Sigma_1 \circ \Sigma_2$ is given by $$\innerproduct{f_1}{f_2}_{\rho} = \sum_{g \in \ParabolicSubgroup_{\left(k_1 c, k_2 c\right)} \backslash \GL_{kc}\left(\finiteField\right)} \innerproduct{f_1\left(g\right)}{f_2\left(g\right)}_{\Sigma_1 \otimes \Sigma_2}.$$
We have that $\innerproduct{f}{f}_{\rho} = \sizeof{\UnipotentRadical_{\left(k_2 c, k_1 c\right)}}$. 
Let $h \in \GL_c\left(\finiteField\right)$.
Then $\specialBesselSpeh{\tau}\left(h\right)$ is given by
\begin{equation*}
\frac{\innerproduct{\rho \begin{pmatrix}
			& \IdentityMatrix{\left(k-1\right)c}\\
			h
		\end{pmatrix} f}{f}_{\rho}}{\innerproduct{f}{f}_\rho},
\end{equation*}
which is equal to
 \begin{equation}\label{eq:sum-for-bessel-speh-of-parabolic-induction-before-expanding-right}
	 \frac{1}{\sizeof{\UnipotentRadical_{\left(k_2 c, k_1 c\right)}}} \sum_{u \in \UnipotentRadical_{\left(k_2c, k_1c\right)}} \innerproduct{f\left( \begin{pmatrix}
			& \IdentityMatrix{k_1 c}\\
			\IdentityMatrix{k_2 c}
		\end{pmatrix} u \begin{pmatrix}
			& \IdentityMatrix{\left(k-1\right)c}\\
			h
		\end{pmatrix} \right)}{{f\left( \begin{pmatrix}
				& \IdentityMatrix{k_1 c}\\
				\IdentityMatrix{k_2 c}
			\end{pmatrix} u \right)}}_{\Sigma_1 \otimes \Sigma_2}.
\end{equation}
Evaluating the second entry of the inner product in \eqref{eq:sum-for-bessel-speh-of-parabolic-induction-before-expanding-right}, we get
\begin{equation}\label{eq:sum-for-bessel-speh-of-parabolic-induction}
	\frac{1}{\sizeof{\UnipotentRadical_{\left(k_2 c, k_1 c\right)}}} \sum_{u \in \UnipotentRadical_{\left(k_2c, k_1c\right)}} \innerproduct{f\left( \begin{pmatrix}
			& \IdentityMatrix{k_1 c}\\
			\IdentityMatrix{k_2 c}
		\end{pmatrix} u \begin{pmatrix}
			& \IdentityMatrix{\left(k-1\right)c}\\
			h
		\end{pmatrix} \right)}{v_1 \otimes v_2}_{\Sigma_1 \otimes \Sigma_2} \fieldCharacterkc{k}{c}^{-1}\left(u\right).
\end{equation}
In order for the summand of \eqref{eq:sum-for-bessel-speh-of-parabolic-induction} to be in the support of $f$, we must have that $$\begin{pmatrix}
	& \IdentityMatrix{k_1 c}\\
	\IdentityMatrix{k_2 c}
\end{pmatrix} u \begin{pmatrix}
	& \IdentityMatrix{\left(k-1\right)c}\\
	h
\end{pmatrix} = p \begin{pmatrix}
& \IdentityMatrix{k_1 c}\\
\IdentityMatrix{k_2 c}
\end{pmatrix} u' $$ for some $p \in \ParabolicSubgroup_{\left(k_1c, k_2c\right)}$ and some $u' \in \UnipotentRadical_{\left(k_2c, k_1c\right)}$.

Let $u = \begin{pmatrix}
	\IdentityMatrix{k_2 c} & X\\
	& \IdentityMatrix{k_1 c}
\end{pmatrix}$ be in the support of the summand of \eqref{eq:sum-for-bessel-speh-of-parabolic-induction}, where $X \in \Mat{k_2 c}{k_1 c}\left(\finiteField\right)$.  Write $p = \begin{pmatrix}
	p_1 & Y\\
	& p_2
\end{pmatrix}$, where $p_1 \in \GL_{k_1 c}\left(\finiteField\right)$, $p_2 \in \GL_{k_2 c}\left(\finiteField\right)$ and $Y \in \Mat{k_1c}{k_2 c}\left(\finiteField\right)$. 
Also write $u' = \begin{pmatrix}
\IdentityMatrix{k_2 c} & X'\\
& \IdentityMatrix{k_1 c}
\end{pmatrix}$, where $X' \in \Mat{k_2 c}{k_1 c}\left(\finiteField\right)$. 
Then we have \begin{equation}\label{eq:decomposition-for-voronoi-element-parabolic-induction}
\begin{pmatrix}
	\IdentityMatrix{k_2 c} & X\\
	& \IdentityMatrix{k_1 c} 
\end{pmatrix} \begin{pmatrix}
	& \IdentityMatrix{\left(k-1\right)c}\\
	h
\end{pmatrix} = \begin{pmatrix}
	p_2 & \\
	Y & p_1
\end{pmatrix} \begin{pmatrix}
	\IdentityMatrix{k_2 c} & X'\\
	& \IdentityMatrix{k_1 c}
\end{pmatrix} = \begin{pmatrix}
	p_2 & p_2X'\\
	Y & YX' + p_1
\end{pmatrix}.
\end{equation}
Decomposing \begin{equation}\label{eq:formula-for-X}
	X = \begin{pmatrix}
		X_{\left(k_2 - 1\right)c \times \left(k_1-1\right)c} & X_{\left(k_2 - 1\right)c \times c} \\
		X_{c \times \left(k_1-1\right)c} & X_{c \times c}
	\end{pmatrix}
\end{equation} we get
$$  \begin{pmatrix}
	\IdentityMatrix{k_2 c} & X\\
	& \IdentityMatrix{k_1 c} 
\end{pmatrix} \begin{pmatrix}
	& \IdentityMatrix{\left(k-1\right)c}\\
	h
\end{pmatrix} = \begin{pmatrix}
X_{\left(k_2 - 1\right)c\times c}h & \IdentityMatrix{\left(k_2 - 1\right)c} &  & X_{\left(k_2 - 1\right)c \times \left(k_1-1\right)c}\\
X_{c \times c}h &  & \IdentityMatrix{c} & X_{c \times \left(k_1 - 1\right)c}\\
& & & \IdentityMatrix{\left(k_1-1\right)c}\\
h
\end{pmatrix}.$$
Hence, we get that if $u$ is such that it is in the support of the summand of \eqref{eq:sum-for-bessel-speh-of-parabolic-induction}, then \begin{equation}\label{eq:formula-for-p2}
p_2 = \begin{pmatrix}
	X_{\left(k_2 - 1\right)c\times c}h & \IdentityMatrix{\left(k_2 - 1\right)c}\\
	X_{c \times c}h &
\end{pmatrix}
\end{equation} is invertible, which is equivalent to requiring $X_{c \times c}$ to be invertible, in which case the inverse of $p_2$ is $$p_2^{-1} = \begin{pmatrix}
& h^{-1} X_{c \times c}^{-1}\\
\IdentityMatrix{\left(k_2 - 1\right)c} & -X_{\left(k_2 - 1\right)c \times c} X_{c \times c}^{-1}
\end{pmatrix}.$$ 
In this case, we also have that \begin{equation}\label{eq:formula-for-Y}
Y = \begin{pmatrix}
	0_{\left(k_1-1\right)c \times c} & 0_{\left(k_1-1\right)c \times \left(k_2-1\right)c} \\
	h & 0_{c\times \left(k_2-1\right)c}
\end{pmatrix}
\end{equation} and \begin{equation}\label{eq:formula-for-X-prime}
X' = p_2^{-1} \begin{pmatrix}
	& X_{\left(k_2 - 1\right)c \times \left(k_1 - 1\right)c}\\
	\IdentityMatrix{c} & X_{c \times \left(k_1-1\right)c}
\end{pmatrix} = \begin{pmatrix}
	h^{-1} X_{c \times c}^{-1} & h^{-1}X_{c\times c}^{-1}X_{c\times\left(k_1-1\right)c}\\
	-X_{\left(k_2 - 1\right)c \times c} X_{c \times c}^{-1} & \ast
\end{pmatrix},
\end{equation} which implies
\begin{equation}\label{eq:formula-for-p1}
	p_1 = \begin{pmatrix}
		& \IdentityMatrix{\left(k_1 - 1\right)c}\\
		0_{c \times c}
	\end{pmatrix} - YX' = \begin{pmatrix}
		& \IdentityMatrix{\left(k_1-1\right)c}\\
		-X_{c \times c}^{-1} & -X_{c \times c}^{-1} X_{c \times \left(k_1 - 1\right)c}
	\end{pmatrix}.
\end{equation}

We decompose $$p_1 = \begin{pmatrix} & \IdentityMatrix{\left(k_1-1\right)c}\\
	-X_{c\times c}^{-1}
\end{pmatrix}\begin{pmatrix}\IdentityMatrix{c} & X_{c\times\left(k_1-1\right)c}\\
	& \IdentityMatrix{\left(k_1-1\right)c}
\end{pmatrix}$$ and $$p_2 = \begin{pmatrix}\IdentityMatrix{\left(k_2-1\right)c} & X_{\left(k_2-1\right)c\times c}X_{c\times c}^{-1}\\
	& \IdentityMatrix{c}
\end{pmatrix}\begin{pmatrix} & \IdentityMatrix{\left(k_2-1\right)c}\\
	X_{c\times c}h
\end{pmatrix}.$$

It is easy to check that given $X \in \Mat{k_2 c}{k_1 c}\left(\finiteField\right)$ as in \eqref{eq:formula-for-X} such that $X_{c \times c} \in \GL_c\left(\finiteField\right)$, then $p_1, p_2, Y, X'$ that are given by formulas \eqref{eq:formula-for-p2}-\eqref{eq:formula-for-p1} satisfy \eqref{eq:decomposition-for-voronoi-element-parabolic-induction}. 
Therefore, we have that $\sizeof{\UnipotentRadical_{\left(k_2 c, k_1 c\right)}} \gbesselSpehFunction{\Sigma}{\fieldCharacter}\begin{pmatrix}
	& \IdentityMatrix{\left(k-1\right)c}\\
	h
\end{pmatrix}$ is given by \begin{align*}
\MoveEqLeft[3]\sum_{X} \gbesselSpehFunction{\Sigma_1}{\fieldCharacter}\left(\begin{pmatrix} & \IdentityMatrix{\left(k_1-1\right)c}\\
	-X_{c\times c}^{-1}
\end{pmatrix}\begin{pmatrix}\IdentityMatrix{c} & X_{c\times\left(k_1-1\right)c}\\
	& \IdentityMatrix{\left(k_1-1\right)c}
\end{pmatrix}\right)\\
\times &  \gbesselSpehFunction{\Sigma_2}{\fieldCharacter}\left(\begin{pmatrix}\IdentityMatrix{\left(k_2-1\right)c} & X_{\left(k_2-1\right)c\times c}X_{c\times c}^{-1}\\
	& \IdentityMatrix{c}
\end{pmatrix}\begin{pmatrix} & \IdentityMatrix{\left(k_2-1\right)c}\\
	X_{c\times c}h
\end{pmatrix}\right)  \\
\times &
\fieldCharacterkc{k}{c} \begin{pmatrix}\IdentityMatrix{\left(k_2-1\right)c} &  & \ast & \ast \\
	& \IdentityMatrix{c} & -X_{c\times\left(k_1-1\right)c} & -X_{c\times c}\\
	&  & \IdentityMatrix{\left(k_1-1\right)c}\\
	&  &  & \IdentityMatrix{c}
\end{pmatrix}\\
\times &  \fieldCharacterkc{k}{c} \begin{pmatrix}
\IdentityMatrix{c} & & h^{-1} X_{c \times c}^{-1} & \ast\\
& \IdentityMatrix{\left(k_2 - 1\right)c} & -X_{\left(k_2 - 1\right)c \times c} X_{c \times c}^{-1} & \ast \\
& & \IdentityMatrix{c}\\
& & & \IdentityMatrix{\left(k_1 - 1\right)c}
\end{pmatrix},
\end{align*}
where the summation is over all $X = \left(\begin{smallmatrix}
	X_{\left(k_2 - 1\right)c \times \left(k_1-1\right)c} & X_{\left(k_2 - 1\right)c \times c} \\
	X_{c \times \left(k_1-1\right)c} & X_{c \times c}
\end{smallmatrix}\right) \in \Mat{k_2 c}{k_1 c}\left(\finiteField\right)$ having the property that $X_{c \times c} \in \GL_c\left(\finiteField\right)$.

The theorem now follows by using the $(\UnipotentRadicalForWss{k_j}{c}, \fieldCharacterkc{k}{c_j})$-equivariance properties of the Bessel--Speh functions, by noticing that $$ \fieldCharacterkc{k_1}{c} \begin{pmatrix}\IdentityMatrix{c} & X_{c\times\left(k_1-1\right)c}\\
	& \IdentityMatrix{\left(k_1-1\right)c}
\end{pmatrix} \cdot \fieldCharacterkc{k}{c} \begin{pmatrix}\IdentityMatrix{\left(k_2-1\right)c} &  & \ast & \ast \\
& \IdentityMatrix{c} & -X_{c\times\left(k_1-1\right)c} & -X_{c\times c}\\
&  & \IdentityMatrix{\left(k_1-1\right)c}\\
&  &  & \IdentityMatrix{c}
\end{pmatrix} = \fieldCharacter\left(-\trace X_{c \times c}\right)^{\delta_{k_1, 1}}$$
and \begin{align*}
	\MoveEqLeft[3] \fieldCharacterkc{k_2}{c} \begin{pmatrix}\IdentityMatrix{\left(k_2-1\right)c} & X_{\left(k_2-1\right)c\times c}X_{c\times c}^{-1}\\
		& \IdentityMatrix{c}
	\end{pmatrix} \cdot \fieldCharacterkc{k}{c}\begin{pmatrix}
		\IdentityMatrix{c} & & h^{-1} X_{c \times c}^{-1} & \ast\\
		& \IdentityMatrix{\left(k_2 - 1\right)c} & -X_{\left(k_2 - 1\right)c \times c} X_{c \times c}^{-1} & \ast \\
		& & \IdentityMatrix{c}\\
		& & & \IdentityMatrix{\left(k_1 - 1\right)c}
	\end{pmatrix} \\
= {} &\fieldCharacter\left(\trace \left(h^{-1} X_{c \times c}^{-1}\right)\right)^{\indicatorFunction{k_2, 1}},
\end{align*}
(where for $j=1,2$, $\indicatorFunction{k_j,1}$ is Kronecker's delta function) and, finally, by changing variables $x = -X_{c \times c}^{-1}$, $y = X_{c \times c} h$.
\end{proof}

We are ready to prove \Cref{thm:multiplicativitiy-in-second-variable}.
\begin{proof}
	We have that $\SpehRepresentation{\tau}{c}$ is the unique irreducible subrepresentation of $\SpehRepresentation{\tau_1}{c} \circ \SpehRepresentation{\tau_2}{c}$ that admits a $\kcNotation{k}{c}{\fieldCharacter}$ vector. Therefore by \Cref{lem:convolution-of-bessel-speh-functions},
	\begin{equation*}
		\GKGaussSum{\pi}{\tau}{\fieldCharacter} = q^{\frac{\left(k - 2\right)c^2}{2}} q^{-c^2} \sum_{x,y \in \GL_c\left(\finiteField\right)} \specialBesselSpeh{\tau_1}\left(x\right) \specialBesselSpeh{\tau_2}\left(y\right) \pi\left(-xy\right).
	\end{equation*}
	Hence $$\GKGaussSum{\pi}{\tau}{\fieldCharacter} = \centralCharacter{\pi}\left(-1\right) \GKGaussSum{\pi}{\tau_1}{\fieldCharacter} \GKGaussSum{\pi}{\tau_2}{\fieldCharacter}$$
	and the result follows.
\end{proof}

\subsubsection{Multiplicativity property in the first argument}

In this section, we prove the multiplicativity property in the first argument. It was proved in the local field case in~\cite[Theorem A.2]{kaplan2018}.

\begin{theorem}\label{thm:multiplicativitiy-in-first-variable}
	Let $\pi_1$ and $\pi_2$ be irreducible representations of $\GL_{c_1}\left(\finiteField\right)$ and of $\GL_{c_2}\left(\finiteField\right)$, respectively. 
 Denote $c = c_1 + c_2$ and let $\pi$ be an irreducible  representation of $\GL_c\left(\finiteField\right)$, such that $\pi$ is a subrepresentation of the parabolic induction $\pi_1 \circ \pi_2$. 
 Then for any $k$ and any irreducible generic representation $\tau$ of $\GL_k\left(\finiteField\right)$
	$$\GKGammaFactor{\pi}{\tau}{\fieldCharacter} = \GKGammaFactor{\pi_1}{\tau}{\fieldCharacter} \GKGammaFactor{\pi_2}{\tau}{\fieldCharacter}.$$
\end{theorem}
The proof of this theorem requires some preparation.

Throughout this section, let $\rho_1 = \SpehRepresentation{\tau}{c_1}$ and $\rho_2 = \SpehRepresentation{\tau}{c_2}$. Let us realize $\SpehRepresentation{\tau}{c}$ as a subrepresentation of $\rho = \SpehRepresentation{\tau}{c_1} \circ \SpehRepresentation{\tau}{c_2}$. 
For $j=1,2$, let $v_j \in \SpehRepresentation{\tau}{c_j}$ be a $\kcNotation{k}{c_j}{\fieldCharacter}$ vector.
Let $\mathcal{Y}$ be as in \Cref{sec:wss-models}.
We first prove the following proposition regarding a $\kcNotation{k}{c}{\fieldCharacter}$ vector of $\SpehRepresentation{\tau}{c_1} \circ \SpehRepresentation{\tau}{c_2}$.

Fix invariant inner products $\innerproduct{\cdot}{\cdot}_1$ and $\innerproduct{\cdot}{\cdot}_2$ on $\SpehRepresentation{\tau}{c_1}$ and $\SpehRepresentation{\tau}{c_2}$, respectively, such that $v_1$ and $v_2$ have norm one with respect to them. 
Let $\innerproduct{\cdot}{\cdot}_{1 \otimes 2}$ be the inner product induced on the tensor product $\SpehRepresentation{\tau}{c_1} \otimes \SpehRepresentation{\tau}{c_2}$. 
An inner product on $\SpehRepresentation{\tau}{c_1} \circ \SpehRepresentation{\tau}{c_2}$ is given by
$$\innerproduct{f_1}{f_2} = \sum_{g \in \ParabolicSubgroup_{\left(k c_1, k c_2\right)} \backslash \GL_{kc}\left(\finiteField\right)} \innerproduct{f_1\left(g\right)}{f_2\left(g\right)}_{1 \otimes 2}.$$

\begin{proposition}\label{prop:f-kappa-is-not-zero}
	Let $f \in \SpehRepresentation{\tau}{c} \subset \SpehRepresentation{\tau}{c_1} \circ \SpehRepresentation{\tau}{c_2}$ be a non-zero $\kcNotation{k}{c}{\fieldCharacter}$ vector. Then $f\left(\kappa\right)$ is a non-zero multiple of $v_1 \otimes v_2$ (see \Cref{sec:wss-models} for the definition of $\kappa$).
\end{proposition}

\begin{proof}	
	If $u_1 \in \UnipotentRadicalForWss{k}{c_1}$ and $u_2 \in \UnipotentRadicalForWss{k}{c_2}$ are given by $$u_1 = \begin{pmatrix}
		\IdentityMatrix{c_1} & X_{12} & X_{13} & \cdots  & X_{1k} \\
		& \IdentityMatrix{c_1} & X_{23} & \cdots & X_{2k} \\
		& & \ddots & \ddots & \vdots  \\
		& & & \IdentityMatrix{c_1} &  X_{k-1,k} \\
		& & & & \IdentityMatrix{c_1}
	\end{pmatrix} \text{ and } u_2 = \begin{pmatrix}
	\IdentityMatrix{c_2} & Y_{12} & Y_{13} & \cdots  & Y_{1k} \\
	& \IdentityMatrix{c_2} & Y_{23} & \cdots & X_{2k} \\
	& & \ddots & \ddots & \vdots  \\
	& & & \IdentityMatrix{c_2} &  Y_{k-1,k} \\
	& & & & \IdentityMatrix{c_2}
\end{pmatrix},$$
then $$\diag\left(u_1, u_2\right) \kappa = \kappa \begin{pmatrix}
	\IdentityMatrix{c} & Z_{12} & Z_{13} & \cdots  & Z_{1k} \\
	& \IdentityMatrix{c} & Z_{23} & \cdots & Z_{2k} \\
	& & \ddots & \ddots & \vdots  \\
	& & & \IdentityMatrix{c} &  Z_{k-1,k} \\
	& & & & \IdentityMatrix{c}
\end{pmatrix},$$
where $Z_{ij} = \diag\left(X_{ij}, Y_{ij}\right)$. Hence, we have that $$\SpehRepresentation{\tau}{c_1}\left(u_1\right) \otimes \SpehRepresentation{\tau}{c_2}\left(u_2\right) f\left(\kappa\right) = f\left(\diag\left(u_1,u_2\right) \kappa\right) = \fieldCharacterkc{k}{c_1}\left(u_1\right) \fieldCharacterkc{k}{c_2}\left(u_2\right) f\left(\kappa\right).$$ By the uniqueness of the $\kcNotation{k}{c_j}{\fieldCharacter}$ vector of $\SpehRepresentation{\tau}{c_j}$ for $j = 1,2$, we must have that $f\left(\kappa\right)$ is a scalar multiple of $v_1 \otimes v_2$.

We move to show that this scalar is non-zero. A non-zero $\kcNotation{k}{c}{\fieldCharacter}$ vector of $\SpehRepresentation{\tau}{c_1} \circ \SpehRepresentation{\tau}{c_2}$ is given by the function $f_{\fieldCharacterkc{k}{c}}$ supported on the double coset $\ParabolicSubgroup_{\left(kc_1, kc_2\right)} \cdot \kappa \cdot \left(\kappa^{-1} \mathcal{Y} \kappa\right)$, such that for any $u \in \mathcal{Y}$, $$f_{\fieldCharacterkc{k}{c}}\left( u \kappa \right) = v_1 \otimes v_2.$$ 
By the results of \Cref{sec:wss-models}, the functional $f' \mapsto \left(f', f_{\fieldCharacterkc{k}{c}}\right)$ is a non-zero $\left(k,c\right)$ functional of $\SpehRepresentation{\tau}{c}$. Since $f$ is the unique (up to scalar multiplication) $\kcNotation{k}{c}{\fieldCharacter}$ vector of $\SpehRepresentation{\tau}{c}$, it follows that there exists a constant $0 \ne c \in \cComplex$, such that $\innerproduct{f'}{f_{\fieldCharacterkc{k}{c}}} = c \innerproduct{f'}{f}$. In particular, $\innerproduct{f}{f_{\fieldCharacterkc{k}{c}}} \ne 0$, which implies that $f\left(\kappa\right) \ne 0$.
	\end{proof}

The following theorem serves as a key ingredient for the proof of \Cref{thm:multiplicativitiy-in-first-variable}. It is proved using the method of root elimination.

\begin{theorem}\label{thm:unipotent-fourier-transofrm-bessel-speh}
	Let $f \in \SpehRepresentation{\tau}{c_1} \circ \SpehRepresentation{\tau}{c_2}$ be a $\kcNotation{k}{c}{\fieldCharacter}$ vector with $f\left(\kappa\right) = \sizeof{\mathcal{Y}}^{-1} v_1 \otimes v_2$. 		Then
	\begin{enumerate}
		\item For any $h \in \GL_c\left(\finiteField\right)$, we have \begin{equation}\label{eq:sum-for-multiplicativity-in-the-first-variable}
			\frac{1}{\sizeof{\UnipotentRadical_{\left(c_1, c_2\right)}}} \sum_{n \in \UnipotentRadical_{\left(c_1, c_2\right)}} \specialBesselSpeh{\tau}\left(nh\right) = q^{c_1 c_2 \binom{k-1}{2}} \innerproduct{f\left( \kappa \begin{pmatrix}
					& \IdentityMatrix{\left(k-1\right)c}\\
					h
				\end{pmatrix} \right)}{v_1 \otimes v_2}.
		\end{equation}
		\item Moreover, if $h = \diag\left(h_1, h_2\right)$ where $h_1 \in \GL_{c_1}\left(\finiteField\right)$ and $h_2 \in \GL_{c_2}\left(\finiteField\right)$, we have 
		\begin{equation*}
			\begin{split}
				\frac{1}{\sizeof{\UnipotentRadical_{\left(c_1, c_2\right)}}} \sum_{n \in \UnipotentRadical_{\left(c_1, c_2\right)}} \specialBesselSpeh{\tau}\left(nh\right) = q^{- c_1 c_2 \left(k-1\right)} \specialBesselSpeh{\tau}\left(h_1\right) \specialBesselSpeh{\tau}\left(h_2\right).
			\end{split}
		\end{equation*}
	\end{enumerate}
\end{theorem}
\begin{proof}
		Notice that by our normalization, we have that $\besselSpehFunction{\tau}{c}\left(g\right) = \innerproduct{\SpehRepresentation{\tau}{c} \left(g\right)f}{ f_{\fieldCharacterkc{k}{c}}}$. Let us write
	\begin{equation}\label{eq:bessel-speh-voronoi-element-average}
		\specialBesselSpeh{\tau}\left(h\right) = \sum_{u \in \mathcal{Y}} \innerproduct{f\left(u \kappa \begin{pmatrix}
				& \IdentityMatrix{\left(k-1\right)c}\\
				h
			\end{pmatrix} \right)}{v_1 \otimes v_2}.
	\end{equation}
	
	We begin with noticing that
	$$\kappa \,\diag\left( \IdentityMatrix{\left(k-1\right)c},
	\begin{pmatrix}
				\IdentityMatrix{c_1} & X\\
		& \IdentityMatrix{c_2}
	\end{pmatrix}\right) = \diag\left(\IdentityMatrix{\left(k-1\right)c_1}, \begin{pmatrix} \IdentityMatrix{c_1} & & X\\
	& \IdentityMatrix{\left(k-1\right)c_2}\\
	& &  \IdentityMatrix{c_2} 
\end{pmatrix}\right) \kappa,$$
	and that if $u = \begin{pmatrix}
		\IdentityMatrix{kc_1} \\
		Y & I_{k c_2}
	\end{pmatrix}$ where $Y$ is as in \eqref{eq:Y-R-subgroup}, then we have $$u\,\diag\left(
		\IdentityMatrix{\left(k-1\right)c_1}, \begin{pmatrix}
		\IdentityMatrix{c_1} & & X\\
		& \IdentityMatrix{\left(k-1\right)c_2}\\
		& &  \IdentityMatrix{c_2} 
	\end{pmatrix} \right) = \diag\left(
		\IdentityMatrix{\left(k - 1\right) c_1},
		\begin{pmatrix}\IdentityMatrix{c_1} & & X\\
		& \IdentityMatrix{\left(k - 1\right) c_2} & Y_k X\\
		& & \IdentityMatrix{c_2}
	\end{pmatrix} \right) u,$$
	where $Y_k = \begin{pmatrix}
		y_{1k}\\
		\vdots\\
		y_{k-1,k}
	\end{pmatrix}.$
	Thus by \eqref{eq:bessel-speh-voronoi-element-average}, we have that the left hand side of \eqref{eq:sum-for-multiplicativity-in-the-first-variable} is given by \begin{equation*}
		q^{-c_1 c_2} \sum_{X \in \Mat{c_1}{c_2}\left(\finiteField\right)} \sum_{u \in \mathcal{Y}} \fieldCharacter\left(\trace\left( y_{k-1,k} X\right)\right) \innerproduct{f \left( u  \kappa  \begin{pmatrix}
			& \IdentityMatrix{\left(k-1\right)c}\\
			h
		\end{pmatrix} \right)}{v_1 \otimes v_2}.
	\end{equation*}
	In order for the sum over $X$ not to vanish, we must have $y_{k-1,k} = 0$, in which case the sum evaluates to $\sizeof{\Mat{c_2}{c_1}\left(\finiteField\right)}$. Therefore, we have that the left hand side of \eqref{eq:sum-for-multiplicativity-in-the-first-variable} is given by \begin{equation}\label{eq:sum-for-multiplicativity-in-the-first-variable-first-root-eliminated}
		\sum_{\substack{u \in \mathcal{Y} \\
				y_{k-1,k} = 0}} \innerproduct{f \left( u  \kappa \begin{pmatrix}
				& \IdentityMatrix{\left(k-1\right)c}\\
				h
				\end{pmatrix} \right)}{v_1 \otimes v_2}.
	\end{equation}

	For $u$ as in \eqref{eq:sum-for-multiplicativity-in-the-first-variable-first-root-eliminated} we have that $$u \kappa = \kappa \begin{pmatrix}
	\IdentityMatrix{c} & Z_{12} & Z_{13} & \cdots  & Z_{1k} \\
	& \IdentityMatrix{c} & Z_{23} & \cdots & Z_{2k} \\
	& & \ddots & \ddots & \vdots  \\
	& & & \IdentityMatrix{c} &  Z_{k-1,k} \\
	& & & & \IdentityMatrix{c}
\end{pmatrix},$$
where $Z_{ij} = \left(\begin{smallmatrix}
	0_{c_1} & 0_{c_1 \times c_2}\\
	y_{ij} & 0_{c_2}
\end{smallmatrix}\right).$ Let us write $$u \kappa = u_{Y_k} \kappa\,\diag\left(u_Z, \IdentityMatrix{c}\right),$$ where $$u_{Y_k}=\diag\left(\IdentityMatrix{\left(k-1\right)c_1}, \begin{pmatrix}
\IdentityMatrix{c_1}\\
Y_k & \IdentityMatrix{\left(k-1\right)c_2}
\end{pmatrix}, \IdentityMatrix{c_2}\right) \,\,\,\,\,\text{and}\,\,\,\,\,\, u_Z = \begin{pmatrix}
	\IdentityMatrix{c} & Z_{12} & Z_{13} & \cdots  & Z_{1,k-1} \\
	& \IdentityMatrix{c} & Z_{23} & \cdots & Z_{2,k-1} \\
	& & \ddots & \ddots & \vdots  \\
	& & & \IdentityMatrix{c} &  Z_{k-2,k-1} \\
	& & & & \IdentityMatrix{c}
\end{pmatrix}.$$ Therefore,
$$ u \kappa \begin{pmatrix}
	& \IdentityMatrix{\left(k-1\right)c}\\
	h
\end{pmatrix} = u_{Y_k} \cdot \kappa \begin{pmatrix}
& \IdentityMatrix{\left(k-1\right)c}\\
h
\end{pmatrix} \diag\left(\IdentityMatrix{c}, u_Z\right).$$
Since $\fieldCharacterkc{k}{c}\left(\diag\left(\IdentityMatrix{c}, u_Z\right)\right) = 1$, we get from \eqref{eq:sum-for-multiplicativity-in-the-first-variable-first-root-eliminated} that the left hand side of \eqref{eq:sum-for-multiplicativity-in-the-first-variable} is given by
\begin{equation}\label{eq:bessel-speh-fourier-transform-2}
		q^{c_1 c_2 \binom{k-1}{2}} \sum_{\substack{y_{1k}, \dots, y_{k-2,k} \in \Mat{c_2}{c_1}\left(\finiteField\right) \\
		y_{k - 1,k} = 0}} \innerproduct{f\left( u_{Y_k} \cdot \kappa \begin{pmatrix}
		& \IdentityMatrix{\left(k-1\right)c}\\
		h
		\end{pmatrix} \right)}{v_1 \otimes v_2}.
\end{equation}

Next, we perform root elimination to eliminate $y_{1k},\dots,y_{k-2,k}$. Denote for $1 \le j \le k-1$ $$S_j = \sum_{\substack{y_{1k}, \dots, y_{j-1,k} \in \Mat{c_2}{c_1}\left(\finiteField\right)\\
y_{j,k} = \dots = y_{k-1,k} = 0}} \innerproduct{f \left( u_{Y_k} \cdot  \kappa \begin{pmatrix}
& \IdentityMatrix{\left(k-1\right)c}\\
h
\end{pmatrix} \right)}{v_1 \otimes v_2}.$$
We will prove by induction that $S_j = S_{k-1}$ for any $j$. For $j = k-1$ there is nothing to prove. Suppose that we know that $S_{j+1} = S_{k-1}$. Let $$e_j = \left(0_c, \dots , 0_c, \IdentityMatrix{c}, 0_c, \dots ,0_c\right)
 \in \Mat{c}{\left(k-2\right)c}\left(\finiteField\right),$$ where $\IdentityMatrix{c}$ is located at the $j$-th component. We have that for any $g \in \GL_{kc}\left(\finiteField\right)$ and any $z \in \squareMatrix_c\left(\finiteField\right)$, $$f\left(g \begin{pmatrix}
 		\IdentityMatrix{c} & & z e_j\\
 		& \IdentityMatrix{c} &\\
 		& & \IdentityMatrix{\left(k-2\right)c}
 	\end{pmatrix}\right) = {f\left(g\right)}.$$
Notice that $$\begin{pmatrix}
	& \IdentityMatrix{\left(k-1\right)c}\\
	h
\end{pmatrix} \begin{pmatrix}
\IdentityMatrix{c} & & h^{-1} z e_j\\
& \IdentityMatrix{c} &\\
& & \IdentityMatrix{\left(k-2\right)c}
\end{pmatrix} = \begin{pmatrix}
\IdentityMatrix{c}\\
& \IdentityMatrix{\left(k-2\right)c}\\
& z e_j & \IdentityMatrix{c}
\end{pmatrix} \begin{pmatrix}
& \IdentityMatrix{\left(k-1\right)c}\\
h
\end{pmatrix}.$$
Let us choose $z = \begin{pmatrix}
	0_{c_1} & x\\
	0_{c_2 \times c_1} & 0_{c_2}
\end{pmatrix},$ where $x \in \Mat{c_1}{c_2}\left(\finiteField\right)$. Then we have the relation $$ \kappa \begin{pmatrix}
\IdentityMatrix{c}\\
& \IdentityMatrix{\left(k-2\right)c}\\
& z e_j & \IdentityMatrix{c}
\end{pmatrix} \kappa^{-1} = \diag\left(\IdentityMatrix{\left(k-1\right)c_1},\begin{pmatrix}
\IdentityMatrix{c_1} & & x e'_j \\
& \IdentityMatrix{c_2}\\
& & \IdentityMatrix{\left(k-2\right)c_2}\\
\end{pmatrix},\IdentityMatrix{c_2}\right)  \eqcolon r^{+}_{j,k}\left(x\right),$$
where $e'_j = \left(0_{c_2},\dots,0_{c_2}, \IdentityMatrix{c_2},0_{c_2}, \dots, 0_{c_2}\right) \in \Mat{c_2}{\left(k-2\right)c_2}\left(\finiteField\right).$
By averaging over all $x$, we get that $S_{j+1}$ equals $$q^{-c_1 c_2} \sum_{x \in \Mat{c_1}{c_2}\left(\finiteField\right)} \sum_{\substack{y_{1k}, \dots, y_{j,k} \in \Mat{c_2}{c_1}\left(\finiteField\right)\\
y_{j+1,k} = \dots = y_{k-1,k} = 0}} \innerproduct{f \left( u_{Y_k} \cdot r^+_{j,k}\left(x\right) \kappa \begin{pmatrix}
& \IdentityMatrix{\left(k-1\right)c}\\
h
\end{pmatrix} \right)}{v_1 \otimes v_2}.$$
A computation shows that given that $y_{j+1,k} = \dots = y_{k-1,k} = 0_{c_2 \times c_1}$, we have $$u_{Y_k} \cdot r_{j,k}^{+}\left(x\right) = \diag\left(\IdentityMatrix{kc_1}, \begin{pmatrix}
	\IdentityMatrix{j c_2} & Y^j_k x\\
	& \IdentityMatrix{c_2}
\end{pmatrix}, \IdentityMatrix{\left(k-j-1\right)c_2}\right) r_{j,k}^{+}\left(x\right) u_{Y_k} \,\,\,\,\, \text{where} \,\,\,\,\, Y^j_k = \begin{pmatrix}
y_{1,k}\\
\vdots\\
y_{j,k}
\end{pmatrix}.$$
Therefore $S_{j+1}$ is given by
$$q^{-c_1 c_2} \sum_{x \in \Mat{c_1}{c_2}\left(\finiteField\right)} \sum_{\substack{y_{1k}, \dots, y_{j,k} \in \Mat{c_2}{c_1}\left(\finiteField\right)\\
y_{j+1,k} = \dots y_{k-1,k} = 0}} \fieldCharacter\left(y_{j,k} x\right) \innerproduct{f \left( u_{Y_k} \cdot \kappa \begin{pmatrix}
& \IdentityMatrix{\left(k-1\right)c}\\
h
\end{pmatrix} \right)}{v_1 \otimes v_2}.$$
Once again, the summation over $x$ vanishes unless $y_{j,k} = 0$, in which case the sum evaluates to $\sizeof{\Mat{c_2}{c_1}\left(\finiteField\right)}$. This shows that $S_j = S_{j+1}$.

	We therefore showed $S_{k-1} = S_1$ which implies from \eqref{eq:sum-for-multiplicativity-in-the-first-variable-first-root-eliminated} that the left hand side of \eqref{eq:sum-for-multiplicativity-in-the-first-variable} is given by
	$$		q^{c_1 c_2 \binom{k-1}{2}} \innerproduct{f\left( \kappa \begin{pmatrix}
			& \IdentityMatrix{\left(k-1\right)c}\\
			h
		\end{pmatrix} \right)}{v_1 \otimes v_2}. $$
	This completes the proof of the first part.
	
	Regarding the second part, since $$\kappa \begin{pmatrix}
	&  \IdentityMatrix{\left(k-1\right)c}\\
	\IdentityMatrix{c}
\end{pmatrix} = \diag\left(\begin{pmatrix}
	& \IdentityMatrix{\left(k-1\right)c_1}\\
	\IdentityMatrix{c_1}
\end{pmatrix}, 		\begin{pmatrix}
	& \IdentityMatrix{\left(k-1\right)c_2}\\
	\IdentityMatrix{c_2}
\end{pmatrix}\right)\kappa,$$
we have that \begin{equation*}
	\begin{split}
		&f\left(\kappa \begin{pmatrix}
			& \IdentityMatrix{\left(k-1\right)c}\\
			h
		\end{pmatrix}\right) = \rho_1\begin{pmatrix}
			& \IdentityMatrix{\left(k-1\right)c_1}\\
			\IdentityMatrix{c_1}
		\end{pmatrix} \otimes \rho_2\begin{pmatrix}
			& \IdentityMatrix{\left(k-1\right)c_2}\\
			\IdentityMatrix{c_2}
		\end{pmatrix} f\left(\kappa\,\diag\left(h, \IdentityMatrix{\left(k-1\right)c}\right)\right).
	\end{split}
\end{equation*}
Using the identity $$\kappa\,\diag\left(h_1,h_2, \IdentityMatrix{\left(k-1\right)c}\right) = \diag\left(h_1, \IdentityMatrix{\left(k-1\right)c_2}, h_2, \IdentityMatrix{\left(k-1\right)c_2}\right) \kappa,$$ and the fact that $f\left(\kappa\right) = q^{-\binom{k}{2} c_1 c_2} v_1 \otimes v_2$, we get 
\begin{align*}
	&\frac{1}{\sizeof{\UnipotentRadical_{\left(c_1, c_2\right)}}} \sum_{n \in \UnipotentRadical_{\left(c_1, c_2\right)}} \specialBesselSpeh{\tau}\left(nh\right) = q^{c_1 c_2 \left(\binom{k-1}{2} - \binom{k}{2}\right)} \specialBesselSpeh{\tau}\left(h_1\right) \specialBesselSpeh{\tau}\left(h_2\right),
\end{align*}
as required.

\end{proof}

\begin{proposition}\label{prop:support-of-whittaker-vector-in-induced-speh}
	Suppose that $h \in \GL_c\left(\finiteField\right)$ is such that $$\sum_{n \in \UnipotentRadical_{(c_1, c_2)}} \specialBesselSpeh{\tau}\left(hn\right) \ne 0.$$
	Then $h \in \ParabolicSubgroup_{\left(c_1, c_2\right)}.$
\end{proposition}
\begin{proof}

	By \Cref{thm:unipotent-fourier-transofrm-bessel-speh}, $$\sum_{n \in \UnipotentRadical_{(c_1, c_2)}} \specialBesselSpeh{\tau}\left(hn\right) = q^{c_1 c_2 \left(\binom{k-1}{2} + 1\right)} \innerproduct{v_f\left( h \right)}{v_1 \otimes v_2},$$ where
	$$v_f\left(h\right) = f\left(\kappa \begin{pmatrix}
		& \IdentityMatrix{\left(k-1\right)c}\\
		h
	\end{pmatrix} \right).$$
	
	Let $$u_Y = \begin{pmatrix}
		\IdentityMatrix{c_2} & \\
		& \IdentityMatrix{\left(k-2\right)c_2} &\\
		Y & & \IdentityMatrix{c_2}
	\end{pmatrix},$$ where $Y \in \squareMatrix_{c_2}\left(\finiteField\right)$. For any $g \in \GL_{kc}\left(\finiteField\right)$ we have that $$f\left( \diag\left(\IdentityMatrix{k c_1}, u_Y\right) g \right) =  \idmap_{\rho_1} \otimes \rho_2\left(u_Y\right) f\left(g\right).$$
	Since $$\diag\left(\IdentityMatrix{k c_1}, u_Y\right) \kappa \begin{pmatrix}
		& \IdentityMatrix{\left(k-1\right)c}\\
		h
	\end{pmatrix} = \kappa \begin{pmatrix}
		& \IdentityMatrix{\left(k-1\right)c}\\
		h
	\end{pmatrix} \diag\left(\begin{pmatrix}
		\IdentityMatrix{c} & h^{-1} \diag\left(0_{c_1}, Y\right) \\
		& \IdentityMatrix{c}
	\end{pmatrix}, \IdentityMatrix{\left(k-2\right)c}\right),$$ and since $f$ is a $\kcNotation{k}{c}{\fieldCharacter}$ vector, we have that
	\begin{align*}
		\idmap_{\rho_1} \otimes \rho_2\left(u_Y\right) v_f\left(h\right) &= \fieldCharacter\left(\trace\left(h^{-1} \diag\left(0_{c_1}, Y\right)\right)\right) v_f\left(h\right).
	\end{align*}
	Write $$h^{-1} = \begin{pmatrix}
		h'_{c_1 \times c_1} & h'_{c_1 \times c_2}\\
		h'_{c_2 \times c_1} & h'_{c_2 \times c_2}
	\end{pmatrix}.$$
	Then $$\idmap_{\rho_1} \otimes \rho_2\left(u_Y\right) v_f\left(h\right) = \fieldCharacter\left(\trace \left(h'_{c_2 \times c_2} Y\right)\right) v_f\left(h\right).$$
	
	Notice that $v_f\left(h\right)$ is left invariant under $\UnipotentRadical_{\left(c_1,c_2\right)}$, because for $A \in \Mat{c_1}{c_2}\left(\finiteField\right)$, $$ \begin{pmatrix}
		\IdentityMatrix{\left(k-1\right)c_1}\\
		& \IdentityMatrix{c_1} & & A\\
		& & \IdentityMatrix{\left(k-1\right)c_2}\\
		& & & \IdentityMatrix{c_2} 
	\end{pmatrix} \kappa = \kappa\,\diag\left(\IdentityMatrix{\left(k-1\right)c}, \begin{pmatrix}
	\IdentityMatrix{c_1} & A\\
	& \IdentityMatrix{c_2}
	\end{pmatrix}\right).$$
	Hence, $$v_f\left(h\right) = \frac{1}{\sizeof{\Mat{c_1}{c_2}\left(\finiteField\right)}}\sum_{A \in \Mat{c_1}{c_2}\left(\finiteField\right)}v_f\left({\begin{pmatrix}
			\IdentityMatrix{c_1} & A\\
			& \IdentityMatrix{c_2}
		\end{pmatrix}h}\right).$$
	Applying $\idmap_{\rho_1} \otimes \rho_2\left(u_Y\right)$ to both sides yields $$\fieldCharacter\left(\trace\left(h'_{c_2 \times c_2} Y\right)\right) v_f\left(h\right) = \frac{1}{\sizeof{\Mat{c_1}{c_2}\left(\finiteField\right)}} \sum_{A \in \Mat{c_1}{c_2}\left(\finiteField\right)}\fieldCharacter\left(\trace\left(h'_{c_2 \times c_2} - h'_{c_2 \times c_1} A\right)Y\right)v_f\left(h\right),$$
	and therefore
	$$v_f\left(h\right) = \frac{1}{\sizeof{\Mat{c_1}{c_2}\left(\finiteField\right)}} \sum_{A \in \Mat{c_1}{c_2}\left(\finiteField\right)}\fieldCharacter\left(\trace\left(-h'_{c_2 \times c_1} A Y\right)\right)v_f\left(h\right).$$
	Choosing $Y = \IdentityMatrix{c_2}$, we have that $A \mapsto \fieldCharacter\left(-\trace\left(h'_{c_2 \times c_1} A\right)\right)$ is a non-trivial character of $\Mat{c_1}{c_2}\left(\finiteField\right)$ unless $h'_{c_2 \times c_1} = 0$, and therefore $v_f\left(h\right) = 0$ unless $h'_{c_2 \times c_1} = 0$. This proves the result because $h^{-1} \in \ParabolicSubgroup_{\left(c_1, c_2\right)}$ if and only if $h \in \ParabolicSubgroup_{\left(c_1, c_2\right)}$.
\end{proof}

We are now ready to prove \Cref{thm:multiplicativitiy-in-first-variable}.
\begin{proof}
	Let $\pi'$, $w_1$, $w_2$ and $\phi_{w_{\pi_1} \otimes w_{\pi_2}}$ be as in the proof of \Cref{thm:multiplicativity-of-godement-jacquet}. Since $\phi_{w_{\pi_1} \otimes w_{\pi_2}}$ is right invariant under $\UnipotentRadical_{(c_1, c_2)}$, we can write
	\begin{equation} \label{eq:zeta-operator-applied-to-full-section}
		\begin{split}
			&\sum_{h \in \GL_c\left(\finiteField\right)} \fieldCharacter\left(\trace h^{-1}\right) \pi'\left(h\right) \phi_{w_{\pi_1} \otimes w_{\pi_2}} = \frac{1}{\sizeof{\UnipotentRadical_{\left(c_1, c_2\right)}}}\sum_{h \in \GL_c\left(\finiteField\right)} \sum_{n \in \UnipotentRadical_{\left(c_1, c_2\right)}}  \specialBesselSpeh{\tau}\left(hn\right) \pi'\left(h\right) \phi_{w_{\pi_1} \otimes w_{\pi_2}}.
		\end{split}
	\end{equation}
	By \Cref{prop:support-of-whittaker-vector-in-induced-speh}, we have that \eqref{eq:zeta-operator-applied-to-full-section} is given by
	\begin{equation*}
	\frac{1}{\sizeof{\UnipotentRadical_{(c_1, c_2)}}} 
\sum_{n, n' \in \UnipotentRadical_{\left(c_1, c_2\right)}} \sum_{\substack{h = \diag\left(h_1, h_2\right)\\
			h_1 \in \GL_{c_1}\left(\finiteField\right)\\
			h_2 \in \GL_{c_2}\left(\finiteField\right) }} \specialBesselSpeh{\tau}\left(h n' n\right) \pi'\left(h n'\right) \phi_{w_{\pi_1} \otimes w_{\pi_2}}.
\end{equation*}
After changing variables and $n \mapsto \left(n'\right)^{-1} n$ and using again the fact that $\phi_{w_{\pi_1} \otimes w_{\pi_2}}$ is right invariant under $\UnipotentRadical_{(c_1, c_2)}$, this becomes
\begin{equation}\label{eq:zeta-operator-applied-to-full-section-before-replacing-h-and-n}		\sum_{n \in \UnipotentRadical_{\left(c_1, c_2\right)}} \sum_{\substack{h = \diag\left(h_1, h_2\right)\\
				h_1 \in \GL_{c_1}\left(\finiteField\right)\\
				h_2 \in \GL_{c_2}\left(\finiteField\right) }} \specialBesselSpeh{\tau}\left(hn\right) \pi'\left(h\right) \phi_{w_{\pi_1} \otimes w_{\pi_2}}.
\end{equation}
Since $h$ normalizes $\UnipotentRadical_{(c_1, c_2)}$, we can replace $hn$ with $nh$ in \eqref{eq:zeta-operator-applied-to-full-section-before-replacing-h-and-n}. It then follows from \Cref{thm:unipotent-fourier-transofrm-bessel-speh} that \eqref{eq:zeta-operator-applied-to-full-section} is the same as
$$q^{\left(2-k\right)c_1 c_2 } \sum_{\substack{h = \diag\left(h_1, h_2\right)\\
		h_1 \in \GL_{c_1}\left(\finiteField\right)\\
		h_2 \in \GL_{c_2}\left(\finiteField\right) }} \specialBesselSpeh{\tau}\left(h_1\right) \specialBesselSpeh{\tau}\left(h_2\right) \pi'\left(h\right) \phi_{w_{\pi_1} \otimes w_{\pi_2}}.$$
Thus,
\begin{equation} \label{eq:final-multiplicativity-equation}
	\begin{split}
		\MoveEqLeft[3] q^{\left(k - 2\right)c_1 c_2 } \sum_{h \in \GL_c\left(\finiteField\right)} \specialBesselSpeh{\tau}\left(h\right) \pi'\left(h\right) \phi_{w_{\pi_1} \otimes w_{\pi_2}} \\
		&=
		\sum_{\substack{h = \diag\left(h_1, h_2\right)\\
				h_1 \in \GL_{c_1}\left(\finiteField\right)\\
				h_2 \in \GL_{c_2}\left(\finiteField\right) }} \specialBesselSpeh{\tau}\left(h_1\right) \specialBesselSpeh{\tau}\left(h_2\right) \pi_1\left(h_1\right) \otimes \pi_2\left(h_2\right) \phi_{w_{\pi_1} \otimes w_{\pi_2}}.
	\end{split}
\end{equation}
The right hand side of \eqref{eq:final-multiplicativity-equation} is $$q^{-\left(k-2\right)\frac{\left(c_1^2 + c_2^2\right)}{2}} \GKPreGammaFactor{\pi_1}{\tau}{\fieldCharacter} \GKPreGammaFactor{\pi_2}{\tau}{\fieldCharacter} \phi_{w_{\pi_1} \otimes w_{\pi_2}}.$$
Hence we showed $$\GKGaussSum{\pi'}{\tau}{\fieldCharacter} \phi_{w_{\pi_1} \otimes w_{\pi_2}} = \GKPreGammaFactor{\pi_1}{\tau}{\fieldCharacter} \GKPreGammaFactor{\pi_2}{\tau}{\fieldCharacter} \phi_{w_{\pi_1} \otimes w_{\pi_2}}.$$
The proof now proceeds as the proof of \Cref{thm:multiplicativity-of-godement-jacquet} after \eqref{eq:gj-final-multiplicativity-equation}.
\end{proof}

\subsection{Kaplan's functional equation}\label{sec:kaplan-functional-equation}
The purpose of this section is to establish a functional equation that the Ginzburg--Kaplan gamma factor fits into.

Suppose $k \ge 2$. 
Let $\pi$ be an irreducible representation of $\GL_c\left(\finiteField\right)$ and let $\tau$ be an irreducible generic representation of $\GL_k\left(\finiteField\right)$. 
For any $0 \le j \le k - 2$ and any $W \in \Whittaker\left(\SpehRepresentation{\tau}{c}, \fieldCharacterkc{k}{c}\right)$, we consider the following \emph{zeta operators}: 

\begin{equation*}
	\zetaOperator_j\left(W, \pi \times \tau\right) = q^{-\frac{(k-j-2)c^2}{2}}
	\sum_{\substack{
		{X \in \Mat{\left(k-j-2\right)c}{c}\left(\finiteField\right)}\\
		{g \in \GL_c\left(\finiteField\right)}
		}}
		  W\begin{pmatrix}
		g \\
		X & \IdentityMatrix{{\left(k-2-j\right)c}}\\
		& & \IdentityMatrix{\left(j+1\right)c}
	\end{pmatrix} \pi\left(g\right)
\end{equation*}  
and
$$\dualZetaOperator_j\left(W, \pi \times \tau\right) = 
q^{-\frac{jc^2}{2}}
\sum_{\substack{{X \in \Mat{c}{jc}\left(\finiteField\right)}\\
{g \in \GL_c\left(\finiteField\right)}}}
	 W \begin{pmatrix}
	& \IdentityMatrix{\left(k-1-j\right)c}& \\
	& & \IdentityMatrix{jc} \\
	g & & X
\end{pmatrix} \pi\left(g\right).$$
A local field analog of these zeta operators for $j = k-2$ was studied in~\cite[Appendix A]{kaplan2018} and~\cite{Kaplan2023}.

\begin{remark}\label{rem:z-star-expressed-in-terms-of-Z}
A simple computation shows that the operator $\dualZetaOperator_j\left(W, \pi \times \tau\right)$ is also given by the formula
$$\dualZetaOperator_j\left(W, \pi \times \tau \right) = \zetaOperator_{k-2-j}\left( \SpehRepresentation{\Contragradient{\tau}}{c}\begin{pmatrix}
	\IdentityMatrix{c} &\\
	& w_{\left(c^{k-1}\right)}
\end{pmatrix} \tilde{W},
\Contragradient{\pi} \times \Contragradient{\tau}
\right).$$
See \Cref{sec:whittaker-models} for the definition of $\tilde{W}$ and of $w_{\left(c^{k-1}\right)}$.
\end{remark}
\begin{remark}
When $c=1$, these zeta operators correspond to the Jacquet--Piatetski-Shapiro--Shalika integrals studied by Roditty-Gershon~\cite{Roditty10} and Nien~\cite[Theorem 2.10]{Nien14}, for the special case $r=1$ (in the notation of~\cite{Nien14}). 
\end{remark}

We will establish a functional equation between $\zetaOperator_j$ and $\dualZetaOperator_j$ for every $j$. Our main result is the following functional equation for $j=k-2$.
\begin{theorem}\label{thm: functional equation 1}
	Suppose that the cuspidal supports of $\pi$ and of $\Contragradient{\tau}$ do not intersect. 
 Then for any $W \in \Whittaker\left(\SpehRepresentation{\tau}{c}, \fieldCharacterkc{k}{c}\right)$,
	$$\dualZetaOperator_{k-2}\left(W, \pi \times \tau\right) = \GKPreGammaFactor{\pi}{\tau}{\fieldCharacter} \zetaOperator_{k-2}\left(W, \pi \times \tau\right).$$
\end{theorem}
The proof of this theorem will occupy the next subsections of this section.

Once we know that \Cref{thm: functional equation 1} holds, we can immediately conclude the more general functional equation described in the following theorem.
\begin{theorem}\label{thm: functional equation 2}
	Suppose that the cuspidal supports of $\pi$ and of $\Contragradient{\tau}$ do not intersect. 
	Then for any $W \in \Whittaker\left(\SpehRepresentation{\tau}{c}, \fieldCharacterkc{k}{c} \right)$ and any $0 \le j \le k - 2$, we have $$\dualZetaOperator_j\left(W, \pi \times \tau \right) = \GKPreGammaFactor{\pi}{\tau}{\fieldCharacter} \zetaOperator_j\left(W, \pi \times \tau \right).$$
\end{theorem}
\begin{proof}
The proof is analogous to the $r=1$ case of~\cite[Theorem 2.10]{Nien14}. For the reader's convenience, we repeat the argument here with the necessary modifications.

The proof is done by descending induction on $j$, where the base case $j=k-2$ is given by \Cref{thm: functional equation 1}.
Let $j<k-2$.
For $W\in\Whittaker\left(\SpehRepresentation{\tau}{c}, \fieldCharacterkc{k}{c} \right)$, denote
\begin{equation*}
W'
=
\sum_{
	X_0\in\Mat{c}{c}(\finiteField)
}
\SpehRepresentation{\tau}{c}
\begin{pmatrix}
	I_c	&				&	&\\
	&I_{(k-3-j)c}	&	&\\
	X_0	&				&I_c&\\
	&				&	&I_{(j+1)c}
\end{pmatrix}
W.
\end{equation*}
We note that
\begin{equation}\label{eq: General functional equation eq 1}
\zetaOperator_j(W,\pi\times\tau)
=
q^{-\frac{c^2}{2}}
\zetaOperator_{j+1}(W',\pi\times\tau),
\end{equation}
	where the factor of $q^{-\frac{c^2}{2}}$ is due to the scaling constant $q^{\frac{-(k-j-2)c^2}{2}}$ in the definition of the operator $\zetaOperator_j(W,\pi\times\tau)$.
	
Assuming the statement of the theorem holds for $j+1$, we multiply both sides of \eqref{eq: General functional equation eq 1} by $\GKPreGammaFactor{\pi}{\tau}{\fieldCharacter}$ and substitute  $\dualZetaOperator_{j+1}(W',\pi\times\tau)$ for $\GKPreGammaFactor{\pi}{\tau}{\fieldCharacter} \zetaOperator_{j+1}(W',\pi\times\tau)$ to get
\begin{equation}\label{eq: General functional equation eq 2}
\GKPreGammaFactor{\pi}{\tau}{\fieldCharacter}
\zetaOperator_j(W,\pi\times\tau)
=
q^{-\frac{c^2}{2}}
\dualZetaOperator_{j+1}(W',\pi\times\tau).
\end{equation}

For every $X\in\Mat{c}{(j+1)c}(\finiteField)$, $X_0\in\Mat{c}{c}(\finiteField)$ and $g\in\GL_{c}(\finiteField)$, we have that
\begin{equation*}
\begin{pmatrix}
	&I_{(k-3-j)c}	&	&\\
	&				&I_c&\\
	&				&	&I_{(j+1)c}\\
g	&				&	&X	
\end{pmatrix}
\begin{pmatrix}
	I_c	&				&	&\\
	&I_{(k-3-j)c}	&	&\\
	X_0	&				&I_c&\\
	&				&	&I_{(j+1)c}
\end{pmatrix}
\end{equation*}
equals
\begin{equation*}
\begin{pmatrix}
	I_{(k-3-j)c}&	&				&\\
	&I_c&-X_0 g^{-1} X	&X_0 g^{-1}\\
	&	&I_{(j+1)c}		&\\
	&	&				&I_c				
\end{pmatrix}
\begin{pmatrix}
	&I_{(k-3-j)c}	&	&\\
	&				&I_c&\\
	&				&	&I_{(j+1)c}\\
	g	&				&	&X	
\end{pmatrix}.
\end{equation*}
Writing $X_1$ for the left-most $c\times c$ block of $X$, it follows that
\begin{equation*}
\SpehRepresentation{\tau}{c}
\begin{pmatrix}
	I_c	&				&	&\\
	&I_{(k-3-j)c}	&	&\\
	X_0	&				&I_c&\\
	&				&	&I_{(j+1)c}
\end{pmatrix}
W
\begin{pmatrix}
	&I_{(k-2-j)c}	&\\
	&				&I_{(j+1)c)}\\
	g	&				&X	
\end{pmatrix}
\end{equation*}
equals
\begin{equation*}
\psi\left(
\tr\left(
-X_0 g^{-1} X_1
\right)
\right)
W
\begin{pmatrix}
	&I_{(k-2-j)c}	&\\
	&				&I_{(j+1)c)}\\
	g	&				&X	
\end{pmatrix},
\end{equation*}
which implies that $\dualZetaOperator_{j+1}\left(W', \pi \times \tau\right)$ is given by
$$q^{-\frac{\left(j+1\right)c^2}{2}} \sum_{\substack{X \in \Mat{c}{\left(j+1\right)c}\left(\finiteField\right)\\
g \in \GL_c\left(\finiteField\right)}} W
\begin{pmatrix}
&I_{(k-2-j)c}	&\\
&				&I_{(j+1)c)}\\
g	&				&X	
\end{pmatrix}
\sum_{X_0 \in \squareMatrix_c\left(\finiteField\right)} \psi\left(
\tr\left(
-X_0 g^{-1} X_1
\right)
\right).$$
The inner sum over $X_0$ vanishes unless $X_1$ is zero, in which case it evaluates to $\sizeof{\squareMatrix_c\left(\finiteField\right)} = q^{c^2}$. Hence
	\begin{equation}\label{eq: General functional equation eq 3}
		\dualZetaOperator_{j+1}(W',\pi\times\tau)=q^{\frac{c^2}{2}}\dualZetaOperator_{j}(W,\pi\times\tau)
	\end{equation} 
	where the additional factor of $q^{-\frac{c^2}{2}}$ is accounted for by the scaling constant $q^{\frac{-jc^2}{2}}$ in the definition of $\dualZetaOperator_j(W,\pi\times\tau)$.

Combining \eqref{eq: General functional equation eq 2} and \eqref{eq: General functional equation eq 3}, we get the desired equality for $j$, and by induction the theorem holds for all $0\le j\le k-2$.
\end{proof}

We move to prove \Cref{thm: functional equation 1}. The proof is similar to~\cite[Theorem 25]{Kaplan2023}; we will show that functionals describing matrix coefficients of the operators $\zetaOperator_j (W,\pi\times\tau)$ and $\dualZetaOperator_j (W,\pi\times\tau)$ afford a certain multiplicity one character and hence are proportional.
Since the proportionality constant is independent of the choice of matrix coefficients, this will imply that the proportionality relation holds between the original operators.

\begin{remark}\label{rem:gk-zeta-integral-with-function}
	For any $k \ge 1$ we may consider the following zeta operator, defined for $W \in \Whittaker\left(\SpehRepresentation{\tau}{c}, \fieldCharacterkc{k}{c}\right)$ and $f \in \mathcal{S}\left(\squareMatrix_c\left(\finiteField\right)\right)$ by
	$$\zetaOperator_{k-2}\left(W, f, \pi \times \tau\right) = \sum_{g \in \GL_c\left(\finiteField\right)}
	W\begin{pmatrix}
		g \\
		& \IdentityMatrix{\left(k-1\right)c}
	\end{pmatrix} f\left(g\right) \pi\left(g\right).$$
	If $k = 1$, then $\tau = \chi \colon \multiplicativegroup{\finiteField} \to \multiplicativegroup{\cComplex}$ is a character and $\SpehRepresentation{\tau}{c} = \chi_{\GL_c}$. In this case, $\zetaOperator_{k-2}\left(W, f, \pi \times \tau\right)$ is the twisted Godement--Jacquet zeta operator discussed in \Cref{subsubsection-twisting-by-a-character}.
	
	If $k \ge 2$, by writing $f = \fourierTransform{\fieldCharacter}{F}$ where $F = \fourierTransform{\fieldCharacter^{-1}}{f} \in \mathcal{S}\left(\squareMatrix_c\left(\finiteField\right)\right)$, a simple computation shows that $$\zetaOperator_{k-2}\left(W, f, \pi \times \tau\right) = \sum_{g \in \GL_c\left(\finiteField\right)}
	W_F\begin{pmatrix}
		g \\
		& \IdentityMatrix{\left(k-1\right)c}
	\end{pmatrix} \pi\left(g\right),$$
	where for $g \in \GL_{kc}\left(\finiteField\right)$, $$W_F\left(g\right) = q^{-\frac{c^2}{2}} \sum_{Y \in \squareMatrix_c\left(\finiteField\right)} W\left(g \begin{pmatrix}
		\IdentityMatrix{c} & Y\\
		& \IdentityMatrix{c}\\
		& & \IdentityMatrix{\left(k-2\right)c}
	\end{pmatrix}\right) F\left(Y\right).$$
	Therefore, when $k \ge 2$, we have that $$\zetaOperator_{k-2}\left(W, f, \pi \times \tau\right) = \zetaOperator_{k-2}\left(W_{\fourierTransform{\fieldCharacter^{-1}}{f}}, \pi \times \tau\right),$$
	and hence $\zetaOperator_{k-2}\left(W, \pi \times \tau \right)$ can be seen as a generalization of the Godement--Jacquet zeta operator.

	The functional equation takes the form
	$$\dualZetaOperator_{k-2}\left(W, f, \pi \times \tau\right) = \GKPreGammaFactor{\pi}{\tau}{\fieldCharacter} \zetaOperator_{k-2}\left(W, f, \pi \times \tau\right),$$
	where $$\dualZetaOperator_{k-2}\left(W, f, \pi \times \tau\right) = 	q^{-\frac{\left(k-3\right)c^2}{2}}
	\sum_{\substack{{X \in \Mat{c}{\left(k-2\right)c}\left(\finiteField\right)}\\
			Y \in \squareMatrix_c\left(\finiteField\right)\\
			{g \in \GL_c\left(\finiteField\right)}}}
	W \begin{pmatrix}
		& \IdentityMatrix{c}& \\
		& & \IdentityMatrix{\left(k-2\right)c} \\
		g & Y & X
	\end{pmatrix} \fourierTransform{\fieldCharacter^{-1}}{f}\left(g^{-1} Y\right)  \pi\left(g\right).$$
\end{remark}

\subsubsection{Equivariance properties}
Let $v\in \pi$ and $\xi \in \Contragradient{\pi}$.
Define the trilinear forms
\[
\zetaOperator(v,\xi,W)	=	
\sum_{g \in \GL_c\left(\finiteField\right)}
W\begin{pmatrix}
	g \\
	 & \IdentityMatrix{{\left(k-1\right)c}}
\end{pmatrix}
\xi\qty(\pi(g)v),
\]
which is equal to $\xi\qty(\zetaOperator_{k-2}\qty(W,\pi\times \tau)v)$,
and
\[
\dualZetaOperator(v,\xi,W)
=
\sum_{\substack{{X \in \Mat{c}{(k-2)c}\left(\finiteField\right)}\\
		{g \in \GL_c\left(\finiteField\right)}}}
W \begin{pmatrix}
	& \IdentityMatrix{c}& \\
	& & \IdentityMatrix{(k-2)c} \\
	g & & X
\end{pmatrix} \xi(\pi(g)v),
\]
which is equal to $q^{\frac{(k-2)c^2}{2}}\xi\qty(\dualZetaOperator_{k-2}(W,\pi\times\tau)v)$.
We will identify $\zetaOperator$ and $\dualZetaOperator$ with their corresponding linear functionals on $\pi \otimes \Contragradient{\pi} \otimes \Whittaker\left(\SpehRepresentation{\tau}{c}, \fieldCharacterkc{k}{c}\right)$.

Let $\UnipotentRadicalDeleted{k}{c}$ be the subgroup of $\UnipotentRadicalForWss{k}{c}$
\[
\qty{
\begin{pmatrix}
	\IdentityMatrix{c}	&0&*\\
	&\IdentityMatrix{c}	&*\\
	&					&z
\end{pmatrix}
\mid
z\in \UnipotentRadicalForWss{k-2}{c}
}.
\]

Let $e_1:\GL_{c}(\finiteField)\to\GL_{kc}(\finiteField)$ and $e_2:\GL_{c}(\finiteField)\to\GL_{kc}(\finiteField)$ be the embeddings given by
\begin{equation*}\label{eq:definition-of-action-of-g-times-g}
	e_1(g)
	=
	\begin{pmatrix}
		g&\\
		&I_{(k-1)c}
	\end{pmatrix} \,\,\,\,\,\,\,\, \text{ and } \,\,\,\,\,\,\,\, e_2(g)
	=
	\begin{pmatrix}
	I_c&\\
	&\diag^{k-1}(g)
	\end{pmatrix}. 
\end{equation*}
Consider the subgroup of $\GL_{kc}\qty(\finiteField)$ 
\begin{align*}
e_1(\GL_{c}(\finiteField)) \cdot
e_2(\GL_{c}(\finiteField)) \cdot
\UnipotentRadicalDeleted{k}{c}
&=
\qty{
	\begin{pmatrix}
		g_1 &\\
		&\diag^{k-1}\qty(g_2)
	\end{pmatrix}u
	\mid
	g_1,g_2\in \GL_{c}(\finiteField), u\in \UnipotentRadicalDeleted{k}{c}
}
\\
&\cong
\qty(\GL_{c}\qty(\finiteField)\times \GL_{c}\qty(\finiteField))\ltimes \UnipotentRadicalDeleted{k}{c},
\end{align*}
and let $\sigma$ be the representation of this subgroup on the space $\pi \otimes \Contragradient{\pi} \otimes \Whittaker\left(\SpehRepresentation{\tau}{c}, \fieldCharacterkc{k}{c}\right)$ determined by
\[
\sigma(e_1(g_1)e_2(g_2)u)(v\otimes \xi\otimes W)
=
\pi(g_1)v\otimes \Contragradient{\pi}\qty(g_2)\xi\otimes \SpehRepresentation{\tau}{c}\qty(
\begin{pmatrix}
	g_1 &\\
	&\diag^{k-1}\qty(g_2)
\end{pmatrix}
u
)
W.
\]

\begin{proposition}\label{prop: Z and Z* equivariance}

The forms $\zetaOperator$ and $\dualZetaOperator$ belong to
\[
\Hom_{
	e_1(\GL_{c}(\finiteField)) \cdot
	e_2(\GL_{c}(\finiteField)) \cdot
	\UnipotentRadicalDeleted{k}{c}
	}
	\qty(\sigma,1\otimes \left(\centralCharacter{\tau}\right)_{\GL_c} \otimes\fieldCharacter_{\qty(c^k)}).
\]
\end{proposition}
\begin{remark}
We note that $1\otimes\left(\centralCharacter{\tau}\right)_{\GL_c}\otimes\psi_{\qty(c^k)}$ is well defined as a character of $e_1(\GL_{c}(\finiteField)) \cdot
e_2(\GL_{c}(\finiteField)) \cdot
\UnipotentRadicalDeleted{k}{c}$ since conjugation of an element $u\in \UnipotentRadicalDeleted{k}{c}$ by $e_1(g_1)e_2(g_2)$ for $g_1,g_2\in\GL_{c}(\finiteField)$ preserves the character $\psi_{\qty(c^k)}$.
\end{remark}
\begin{proof}
It is sufficient to consider the action of the subgroups  $e_1(\GL_{c}(\finiteField))$, 
$e_2(\GL_{c}(\finiteField))$ and
$\UnipotentRadicalDeleted{k}{c}$ separately.
Throughout the proof, let $g_1,g_2 \in \GL_{c}\qty(\finiteField)$ and $u\in \UnipotentRadicalDeleted{k}{c}$.

We begin with the form $\zetaOperator$.
The action of $e_1(\GL_{c}(\finiteField))$ is given by
\begin{equation}\label{eq: Z first factor equivariance}
\zetaOperator\qty(
\sigma(
e_1(g_1)
)
(v\otimes \xi\otimes W)
)
=\sum_{g\in\GL_{c}\qty(\finiteField)}
W
\begin{pmatrix}
	gg_1&\\
	&\IdentityMatrix{\qty(k-1)c}
\end{pmatrix}
\xi
\qty(
\pi\qty(gg_1)v
).
\end{equation}
Substituting $gg_1^{-1}$ for $g$ in \eqref{eq: Z first factor equivariance}, we get
\[
\zetaOperator\qty(
\sigma(
e_1(g_1)
)
(v\otimes \xi\otimes W)
)
=
\zetaOperator\qty(v\otimes \xi\otimes W).
\]
For the action of $e_2(\GL_{c}(\finiteField))$, we have
\begin{equation}\label{eq: Z second factor equivariance}
\zetaOperator\qty(
\sigma(
e_2(g_2)
)
(v\otimes \xi\otimes W)
)
=
\sum_{g\in\GL_{c}\qty(\finiteField)}
W
\begin{pmatrix}
	g &\\
	&\diag^{k-1}(g_2)
\end{pmatrix}
\xi
\qty(
\pi\qty(g_2^{-1}g)v
).
\end{equation}
Substituting $g_2g$ for $g$ in \eqref{eq: Z second factor equivariance} and using the equivariance properties of $W$ with respect to a left translation by a diagonal embedding of $\GL_{c}(\finiteField)$ in $\GL_{kc}(\finiteField)$, we get
\[
\zetaOperator\qty(
\sigma(
e_2(g_2)
)
(v\otimes \xi\otimes W)
)
=
\centralCharacter{\tau}(\det g_2)\cdot \zetaOperator(v\otimes \xi\otimes W).
\]

Finally, for the action of $\UnipotentRadicalDeleted{k}{c}$, we have
\begin{align}\label{eq: Z third factor equivariance}
\zetaOperator\qty(\sigma(u)(v\otimes\xi\otimes W))
=
\sum_{g\in\GL_{c}\qty(\finiteField)}
W
\qty(
\begin{pmatrix}
	g&\\
	&\IdentityMatrix{\qty(k-1)c}
\end{pmatrix}
u
)
\xi
\qty(
\pi\qty(g)v
).
\end{align}
Conjugating $u \in \UnipotentRadicalDeleted{k}{c}$ by $\begin{pmatrix}
	g&\\
	&\IdentityMatrix{\qty(k-1)c}
\end{pmatrix}$
in \eqref{eq: Z third factor equivariance} does not change the value of $\psi_{\qty(c^k)}$, and so
we have
\[
\zetaOperator\qty(\sigma(u)(v\otimes \xi \otimes W))
=
\psi_{\qty(c^k)}(u)\cdot \zetaOperator(v\otimes \xi \otimes W),
\]
since $W\in \Whittaker\left(\SpehRepresentation{\tau}{c}, \fieldCharacterkc{k}{c}\right)$.

We now consider the form $\dualZetaOperator$.
The action of $e_1(\GL_{c}(\finiteField))$ is given by

\begin{equation}\label{eq: Z* first factor equivariance}
	\dualZetaOperator\qty(
	\sigma(e_1(g_1))(v\otimes\xi\otimes W)
	)
	=
	\sum_{
	\substack{
	X\in\Mat{c}{(k-2)c}\qty(\finiteField)\\
	g\in\GL_{c}\qty(\finiteField)
	}
	}
	W
	\begin{pmatrix}
		& \IdentityMatrix{c}& \\
		& & \IdentityMatrix{(k-2)c} \\
		gg_1 & & X
	\end{pmatrix} \xi(\pi(gg_1)v).
\end{equation}
Substituting $gg_1^{-1}$ for $g$ in \eqref{eq: Z* first factor equivariance}, we get
\[
	\dualZetaOperator\qty(
	\sigma(e_1(g_1))(v\otimes\xi\otimes W)
	)
=
\dualZetaOperator\qty(v\otimes\xi\otimes W).
\]

For the action of $e_2(\GL_{c}(\finiteField))$, we have
\begin{equation}\label{eq: Z* second factor equivariance}
	\dualZetaOperator\qty(\sigma(e_2(g_2))(v\otimes\xi\otimes W)
	)
	=
	\sum_{
	\substack{
	X\in\Mat{c}{(k-2)c}\qty(\finiteField)\\
	g\in\GL_{c}\qty(\finiteField)
	}
	}
		W
		\begin{pmatrix}
			& g_2& \\
			& & \diag^{k-2}(g_2) \\
			g & & X\cdot\diag^{k-2}(g_2)
		\end{pmatrix} \xi(\pi(g^{-1}_2g)v).
\end{equation}
Substituting $g_2g$ for $g$ and $g_2\cdot X\cdot \diag^{k-2}(g_2^{-1})$ for $X$ in \eqref{eq: Z* second factor equivariance}, and using the equivariance properties of $W$ with respect to a left translation by a diagonal embedding of $\GL_{c}(\finiteField)$ in $\GL_{kc}(\finiteField)$, we get
\[
\dualZetaOperator\qty(\sigma(e_2(g_2))(v\otimes\xi\otimes W)
)
=
\centralCharacter{\tau}(\det g_2)\cdot \dualZetaOperator\qty(v\otimes\xi\otimes W).
\]

For the action of $\UnipotentRadicalDeleted{k}{c}$, we first denote
\[
u
=
\begin{pmatrix}
I_c	&0	&A\\
	&I_c&B\\
	&	&U
\end{pmatrix}
\]
where $A,B\in\Mat{c}{(k-2)c}(\finiteField)$ and $U\in \UnipotentRadicalForWss{k-2}{c}$.
In this notation, we have
\begin{equation}\label{eq: Z* third equivariance property}
\dualZetaOperator\qty(\sigma(u)(v\otimes\xi\otimes W))
=
\sum_{
\substack{
X\in\Mat{c}{(k-2)c}\qty(\finiteField)\\
g\in\GL_{c}\qty(\finiteField)
}
}
W
\begin{pmatrix}
	& \IdentityMatrix{c}& B \\
	& & U \\
	g & & gA+XU
\end{pmatrix}
 \xi(\pi(g)v).
\end{equation}
Substituting $(X-gA)U^{-1}$ for $X$ in \eqref{eq: Z* third equivariance property} and noting that
\[
\psi_{\qty(c^k)}
\begin{pmatrix}
	I_c	&B&\\
		&U&\\
		&	&I_c
\end{pmatrix}
=
\psi_{\qty(c^k)}(u),
\]
we have
\[
\dualZetaOperator\qty(\sigma(u)(v\otimes\xi\otimes W))
=
\psi_{\qty(c^k)}(u)\cdot \dualZetaOperator\qty(v\otimes\xi\otimes W)
\]
which finishes the proof.
\end{proof}

\subsubsection{Uniqueness of trilinear form}
In this subsection we show, in the notation of Proposition \ref{prop: Z and Z* equivariance}, the following
\begin{theorem}\label{thm:trilinear functional equation}
Suppose that the cuspidal supports of $\pi$ and of $\Contragradient{\tau}$ do not intersect. 
Then the dimension of 
\begin{equation}\label{eq: Hom space for GK functional equation}
\Hom_{
	e_1(\GL_{c}(\finiteField)) \cdot
	e_2(\GL_{c}(\finiteField)) \cdot
	\UnipotentRadicalDeleted{k}{c}
}
\qty(\sigma,1\otimes\left(\centralCharacter{\tau}\right)_{\GL_c}\otimes\fieldCharacter_{\qty(c^k)})
\end{equation}
is at most $1$.
\end{theorem}

Before proving the theorem, let us recall some notions that will be used throughout the proof.

Recall that any character of $\Mat{m}{n}\left(\finiteField\right)$ is of the form $X \mapsto \fieldCharacter\left(\trace\left(X \cdot \transpose{A}\right)\right)$, where $A \in \Mat{m}{n}\left(\finiteField\right)$. We call $A$ the \emph{coefficient matrix} of the character $X \mapsto \fieldCharacter\left(\trace\left(X \cdot \transpose{A}\right)\right)$.

Let $\Sigma$ be an irreducible representation of $\GL_n\left(\finiteField\right)$. Given a unipotent radical $\UnipotentRadical_{\left(n_1,\dots,n_r\right)}$ of $\GL_n\left(\finiteField\right)$ and a character $\Psi \colon \UnipotentRadical_{\left(n_1,\dots,n_r\right)} \to \multiplicativegroup{\cComplex}$, we define the \emph{(twisted) Jacquet module of $\Sigma$ with respect to $\left(\UnipotentRadical_{\left(n_1,\dots,n_r\right)}, \Psi\right)$} as the subspace $$J_{\UnipotentRadical_{\left(n_1,\dots,n_r\right)}, \Psi}\left(\Sigma\right) = \left\{v \in \Sigma \mid \Sigma\left(u\right)v = \Psi\left(u\right)v \text{ for all } u \in \UnipotentRadical_{\left(n_1,\dots,n_r\right)}\right\}.$$ We also denote $$J_{\UnipotentRadical_{\left(n_1,\dots,n_r\right)}} = J_{\UnipotentRadical_{\left(n_1,\dots,n_r\right)}, 1},$$
where $1$ is the trivial character of $\UnipotentRadical_{\left(n_1,\dots,n_r\right)}$. Notice that the space $J_{\UnipotentRadical_{\left(n_1,\dots,n_r\right)}}$ is invariant under the action of the parabolic subgroup $P_{\left(n_1,\dots,n_r\right)} \subset \GL_n\left(\finiteField\right)$, and in particular affords a representation for the Levi part $D_{\left(n_1,\dots,n_r\right)} \cong \GL_{n_1}\left(\finiteField\right) \times \dots \times \GL_{n_r}\left(\finiteField\right)$.

We move to the proof of \Cref{thm:trilinear functional equation}.
\begin{proof}
If the space in \eqref{eq: Hom space for GK functional equation} is equal to $0$, we are done.
Otherwise, let 

\begin{equation*}
0\ne B\in \Hom_{
	e_1(\GL_{c}(\finiteField)) \cdot
	e_2(\GL_{c}(\finiteField)) \cdot
	\UnipotentRadicalDeleted{k}{c}
}
\qty(\sigma,1\otimes\left(\centralCharacter{\tau}\right)_{\GL_c}\otimes\fieldCharacter_{\qty(c^k)}).
\end{equation*}

Throughout the proof, let $v$ denote an element of $\pi$, let $\xi$ denote an element of $\Contragradient{\pi}$ and let $W$ denote an element of $\Whittaker\left(\SpehRepresentation{\tau}{c}, \fieldCharacterkc{k}{c} \right)$.
We will say that $B$ is supported on $W$ if there exist  $v$ and $\xi$ such that $B(v\otimes \xi \otimes W)\ne 0$.
Otherwise, we will say that $B$ vanishes on $W$.

We will show that $B$ vanishes on the elements of a certain subspace of  $\Whittaker\left(\SpehRepresentation{\tau}{c}, \fieldCharacterkc{k}{c} \right)$, and so the value of $B(v\otimes \xi \otimes W)$ is determined by the projection of $W$ onto the complementary subspace of $\Whittaker\left(\SpehRepresentation{\tau}{c}, \fieldCharacterkc{k}{c} \right)$.
We will also show that the values of $B$ for $W$ in this complementary subspace are determined up to normalization of the value of $B$ on a certain element, and so $B$ is determined up to a scalar. 

Let $W_{\text{Sh}} = \besselSpehFunction{\tau}{c} \in  \Whittaker\left(\SpehRepresentation{\tau}{c}, \fieldCharacterkc{k}{c} \right)$.
We have that the map $(v,\xi)\mapsto B(v\otimes \xi \otimes W_{\text{Sh}})$ is an element of 
\[
\Hom_{\GL_c\left({\finiteField}\right)}\qty(\pi\otimes\Contragradient{\pi},1).
\]
Thus, since $\pi$ is irreducible, this map is proportional to the evaluation map $(v,\xi)\mapsto \xi(v)$.
We will see later in the proof that if this proportionality constant is $0$, then $B=0$.
Thus, we will assume that it is not $0$ and normalize it to be $1$.
Under this assumption, we will show that the values of $B$ are uniquely determined.

We first treat the case where $W$ lies in the Jacquet module $J_{\UnipotentRadicalDeleted{k}{c},\psi_{\qty(c^k)}}\qty(\SpehRepresentation{\tau}{c})$.
This Jacquet module is preserved under the action of the restriction of $\SpehRepresentation{\tau}{c}$ to the subgroups of $\GL_{kc}\qty(\finiteField)$
\begin{align*}
	e_1(\GL_c(\finiteField))\cdot e_2(\GL_c(\finiteField))
	&=
	\qty{
		\begin{pmatrix}
			g_1 &\\
			&\diag^{k-1}\qty(g_2)
		\end{pmatrix}
		\mid
		g_1,g_2\in \GL_{c}(\finiteField)
	}
	\\
	&\cong \GL_{c}\qty(\finiteField)\times \GL_{c}\qty(\finiteField)
\end{align*}
and
\begin{equation}\label{eq:1,2 block}
	\qty{
		\begin{pmatrix}
			I_c	&X	&\\
			&I_c&\\
			&	&I_{(k-2)c}
		\end{pmatrix}
		\mid
		X\in
		\Mat{c}{c}(\finiteField)
	}.
\end{equation}

We make the further assumption that $W$ transforms as a character under the restriction of $\SpehRepresentation{\tau}{c}$ to the abelian subgroup in \eqref{eq:1,2 block}.
Let $\alpha\in\Mat{c}{c}(\finiteField)$ be the coefficient matrix corresponding to this character, i.e.\
\[
\SpehRepresentation{\tau}{c}
	\begin{pmatrix}
		I_c	&X	&\\
		&I_c&\\
		&	&I_{(k-2)c}
	\end{pmatrix}
W
=
\psi(\tr(X \cdot \transpose{\alpha}))W.
\]
If $\alpha=I_c$, then $W=C\cdot W_{\text{Sh}}$ for some constant $C$.
Thus, in this case, for every $v,\xi$ we have 
\[
B(v\otimes\xi\otimes W) = C\cdot\xi(v).
\]

If $\alpha = 0$, then $W$ belongs to the Jacquet module $J_{N_{\qty(c^k)},\psi^\circ_{\qty(c^k)}}\qty(\SpehRepresentation{\tau}{c})$, where we denote by $\psi^\circ_{\qty(c^k)}$ the character of $N_{\qty(c^k)}$ given by
\[
\psi^\circ_{\qty(c^k)}
\begin{pmatrix}
	\IdentityMatrix{c}	&X_1 &\ast	&\ast		&\ast	\\
	&\IdentityMatrix{c} & X_2	& \ast	&\ast		\\
	&		&\ddots & \ddots & \ast \\
	& & & \IdentityMatrix{c} &X_{k-1}\\
	&		&		& & \IdentityMatrix{c}
\end{pmatrix}
=
\psi\qty(\tr\qty(X_2+\cdots+X_{k-1})).
\]
This Jacquet module is preserved under the action of $\SpehRepresentation{\tau}{c}(e_1(\GL_c(\finiteField)))$.
Fixing $\xi$, the map $(v,W)\mapsto B(v\otimes \xi\otimes W)$ induces an element of 
\[
\Hom_{\GL_{c}(\finiteField)}
\qty(
\pi\otimes
J_{N_{\qty(c,(k-1)c)}}\qty(\SpehRepresentation{{\tau}}{c})
,1)
\]
where $g\in\GL_c(\finiteField)$ acts by $\pi(g)\otimes \SpehRepresentation{\tau}{c}(e_1(g))$.
Since the cuspidal supports of $\pi$ and $\Contragradient{\tau}$ are disjoint, this Hom space is $0$ and so $B$ vanishes on $W$ in this case.

If $\alpha$ is of the form
\begin{equation}\label{eq: alpha diagonal}
\begin{pmatrix}
	0_{c-\ell}&\\
	&I_{\ell}
\end{pmatrix},
\end{equation}
where $0<\ell<c$, we make the further assumption that $W$ transforms as a character under the restriction of $\SpehRepresentation{\tau}{c}$ to the abelian subgroup
\begin{equation}\label{eq: parabolic inside c block}
	\qty{
		\begin{pmatrix}
			I_{c-\ell}	&	X\\
			&	I_{\ell}\\
			&	&I_{(k-1)c}
		\end{pmatrix}
		\mid
		X\in\Mat{\ell}{c-\ell}(\finiteField)
	}
\end{equation}
(which preserves the isotypic component of $J_{\UnipotentRadicalDeleted{k}{c},\psi_{\qty(c^k)}}\qty(\SpehRepresentation{\tau}{c})$ corresponding to this $\alpha$).
Let $\beta\in \Mat{(c-\ell)}{\ell}(\finiteField)$ be the coefficient matrix corresponding to this character.
If $\beta = 0$, then, similarly to the case $\alpha=0$, we have that $B$ vanishes on $W$.
If $\beta$ is of the form

\begin{equation}\label{eq: beta diagonal}
\begin{pmatrix}
	0_{ (c-\ell-\ell') \times (\ell-\ell')}&0_{(c-\ell-\ell')\times \ell'}\\
	0_{\ell' \times (\ell-\ell')}&I_{\ell'}
\end{pmatrix}
\end{equation}
for $1\le\ell'\le\min(\ell,c-\ell)$, then $W$ must be $0$ since otherwise we would get a contradiction to the fact that $\SpehRepresentation{\tau}{c}$ is a representation of $(k,c)$ type (see Example \ref{ex: middle rank case}), and so $B$ vanishes on $W$.

We now reduce all other cases to one of the cases already considered.
Remaining in the case where $\alpha$ is as in \eqref{eq: alpha diagonal}, there exist $g_1\in \GL_{c-\ell}(\finiteField)$ and $g_2\in\GL_\ell(\finiteField)$ such that
\[
\transpose{g_1} \cdot \beta \cdot \transpose{\left(g_2^{-1}\right)}
\]
is as in \eqref{eq: beta diagonal}, with $\ell' = \rank \beta $.
Let $g= \diag\left(g_1, g_2\right)$.
Notice that $\SpehRepresentation{\tau}{c}\left(e_1\left(g^{-1}\right) e_2\left(g^{-1}\right)\right) W$ then transforms by the character whose coefficient matrix is $\transpose{g_1} \cdot \beta \cdot \transpose{\left(g_2^{-1}\right)}$. Thus we have from the previous case
\begin{align*}
B(v\otimes \xi \otimes W)
&=
\centralCharacter{\tau}(\det g) \cdot B(\pi(g^{-1})v\otimes \Contragradient{\pi}(g^{-1})\xi \otimes  \SpehRepresentation{\tau}{c}\qty(e_1(g^{-1})e_2(g^{-1})) W)\\
&=0.
\end{align*}

Returning to the generality where $W$ transforms as a character of the subgroup in \eqref{eq:1,2 block} with coefficient matrix $\alpha$, there exist $g_1,g_2\in\GL_c(\finiteField)$ such that
\[
\transpose{g_1} \cdot \alpha \cdot \transpose{\left(g_2^{-1}\right)}
\]
is as in \eqref{eq: alpha diagonal}, with $\ell = \rank \alpha$. Similarly to before, notice that $\SpehRepresentation{\tau}{c}\left(e_1\left(g_1^{-1}\right) e_2\left(g_2^{-1}\right)\right) W$ transforms by the character whose coefficient matrix is $\transpose{g_1} \cdot \alpha \cdot \transpose{\left(g_2^{-1}\right)}$.
Thus we have
\[
B(v
\otimes
\xi
\otimes
W
)
=
\centralCharacter{\tau}(\det g_2)\cdot
B(\pi(g_1^{-1})v
\otimes
\Contragradient{\pi}(g_2^{-1})\xi
\otimes
\SpehRepresentation{\tau}{c}
(e_1(g_1^{-1})e_2(g_2^{-1}))
W)
\]
and so either vanishes, if $\rank \alpha < c$, or else is uniquely determined and equal to
\[
\centralCharacter{\tau}(\det g_2)\cdot
C \cdot
\xi \left(\pi\left(g_2 g_1^{-1}\right)v\right)
\]
where the constant $C$ is such that $\SpehRepresentation{\tau}{c}
(e_1(g_1^{-1})e_2(g_2^{-1}))
W
=
C\cdot W_{\text{Sh}}$.

Since the Jacquet module $J_{\UnipotentRadicalDeleted{k}{c},\psi_{\qty(c^k)}}\qty(\SpehRepresentation{\tau}{c})$
is a representation space of the abelian group in \eqref{eq:1,2 block}, it remains to reduce from the general case $W\in \Whittaker\left(\SpehRepresentation{\tau}{c}, \fieldCharacterkc{k}{c} \right)$ to the case where $W \in J_{\UnipotentRadicalDeleted{k}{c},\psi_{\qty(c^k)}}\qty(\SpehRepresentation{\tau}{c})$.
Assume that $W$ transforms as a character of the abelian subgroup
\[
\qty{
	\begin{pmatrix}
		I_{(k-1)c}	&Z\\
		&I_c
	\end{pmatrix}
	\mid
	Z\in\Mat{(k-1)c}{c}(\finiteField)
}
\]
with coefficient matrix $\alpha\in\Mat{(k-1)c}{c}(\finiteField)$.
By the equivariance properties of $B$, we have
\begin{align*}
\psi(\tr(Z \cdot \transpose{\alpha}))\cdot B(v\otimes\xi\otimes W)
&=
B(v\otimes \xi\otimes \SpehRepresentation{\tau}{c}\begin{pmatrix}
	I_{(k-1)c}	&Z\\
	&I_c
\end{pmatrix}W)
\\
&=
\psi(\tr(Z\begin{pmatrix}0_{c\times(k-2)c}&I_c\end{pmatrix}))\cdot B(v\otimes\xi\otimes W)
\end{align*}
for all $Z\in\Mat{(k-1)c}{c}(\finiteField)$.
Thus $B$ vanishes on $W$ unless $\alpha=\transpose{\begin{pmatrix}0_{c\times(k-2)c}&I_c\end{pmatrix}}$.
The isotypic component of the character corresponding to this $\alpha$ in $\Whittaker\left(\SpehRepresentation{\tau}{c}, \fieldCharacterkc{k}{c} \right)$ is preserved under the restriction of  $\SpehRepresentation{\tau}{c}$ to the abelian subgroup 
\begin{equation}\label{eq:column by column character}
	\qty{
		\begin{pmatrix}
			I_{(k-2)c}	&Z&\\
			&I_c&\\
			&	&I_c
		\end{pmatrix}
		\mid
		Z\in\Mat{(k-2)c}{c}(\finiteField)
	}.
\end{equation}
Repeating the same argument, block-column by block-column, going from right to left, we get
$B(v\otimes\xi\otimes W)=B(v\otimes\xi\otimes W')$, where $W'$ is the projection of $W$ onto 
the Jacquet module $J_{\UnipotentRadicalDeleted{k}{c},\psi_{\qty(c^k)}}\qty(\SpehRepresentation{\tau}{c})$,
which finishes the proof.
\end{proof}

\begin{example}\label{ex: middle rank case}
In this example we give details for the arguments in the case that $\beta \ne 0$ in the proof of Theorem \ref{thm:trilinear functional equation} above, for specific values of $k$, $c$, $\ell$ and $\ell'$.
For the sake of this example we will use the following notation:
when a character of a subgroup of $\GL_{kc}(\finiteField)$ is given by $X\mapsto \psi(\tr(X \cdot \transpose{Y}))$, where $Y\in\Mat{kc}{kc}(\finiteField)$ and each entry of $Y$ is either $0$ or $1$, we will depict this character by writing a general element of the subgroup of $\GL_{kc}(\finiteField)$ and marking with circles the entries at indices where $Y$ has an entry equal to $1$.

In the notation of the proof of Theorem \ref{thm:trilinear functional equation}, let $k=4$, $c=3$, $\ell=1$ and $\ell'=1$.
As in the proof of Theorem \ref{thm:trilinear functional equation}, let $W$ be an element of $\SpehRepresentation{\tau}{c}$ which transforms according to the character in \Cref{fig: ostar matrices sub1} and assume $W\ne 0$.

We permute the columns and rows by a permutation sending the columns and rows at indices $c,2c,\dots,kc$ to those at the indices $k(c-1)+1,k(c-1)+2,\dots,kc$, respectively, and sending the column and row at index $c-\ell$ to those at index $k(c-1)$ (the action of the permutation on the other indices may be arbitrary).
Thus, we have that $\SpehRepresentation{\tau}{c}$ contains a character as in \Cref{fig: ostar matrices sub2}.
Note that, since $\beta \ne 0 $, we have a ``run'' of length more than $k-1$ of entries on the bottom-right of the superdiagonal which contribute to the character.

Starting from the rightmost column and working column by column from right to left until the $(k(c-1)+1)$-th column, we perform a root exchange process  which exchanges roots at indices of the form $(a,b)$ with those at corresponding indices $(b-1,a)$.
In each column we proceed from top to bottom, i.e.\ from a low value of $a$ to a high value of $a$.
We do this for $a$ and $b$ such that $a<b$ and such that the subgroup appearing in \Cref{fig: ostar matrices sub2} does not contain the root subgroup at index $(a,b)$.
By this procedure, we get that $\SpehRepresentation{\tau}{c}$ contains the character in \Cref{fig: ostar matrices sub3}.

Restricting to a subgroup, we have that $\SpehRepresentation{\tau}{c}$ contains the character in \Cref{fig: ostar matrices sub4}.
In terms of the notation of the figure, we ``remove'' the off-diagonal entries in columns of indices $1,\dots,k(c-1)$.

The character in \Cref{fig: ostar matrices sub4} may be extended to a Zelevinsky character of $\UnipotentSubgroup_{kc}$ appearing in $\SpehRepresentation{\tau}{c}$.
However, this Zelevinsky character will be of type greater or not comparable to $(k^c)$ since its type will have a component greater or equal to $k+1$.
This is in contradiction to $\SpehRepresentation{\tau}{c}$ being a representation of $(k,c)$ type (see \Cref{sec:kc-representations}).

\begin{figure}[H]
	\centering
	\begin{subfigure}[b]{0.45\textwidth}
		\centering
		\[
		\qty(\!\!
		\begin{array}{ccc|ccc|ccc|ccc}
			1 &  & * & * & * & * & * & * & * & * & * & *\\
			& 1 & \ostar & * & * & * & * & * & * & * & * & *\\
			&  & 1 & * & * & \ostar & * & * & * & * & * & *\\
			\hline  &  &  & 1 &  &  & \ostar & * & * & * & * & *\\
			&  &  &  & 1 &  & * & \ostar & * & * & * & *\\
			&  &  &  &  & 1 & * & * & \ostar & * & * & *\\
			\hline  &  &  &  &  &  & 1 &  &  & \ostar & * & *\\
			&  &  &  &  &  &  & 1 &  & * & \ostar & *\\
			&  &  &  &  &  &  &  & 1 & * & * & \ostar\\
			\hline  &  &  &  &  &  &  &  &  & 1\\
			&  &  &  &  &  &  &  &  &  & 1\\
			&  &  &  &  &  &  &  &  &  &  & 1
		\end{array}
		\!\!)
		\]
		\caption{}
		\label{fig: ostar matrices sub1}
	\end{subfigure}
	\hfill
	\begin{subfigure}[b]{0.45\textwidth}
		\centering
		\[
		\qty(\!\!
		\begin{array}{cccc|cccc|cccc}
			1 & * & * & * & * & * & * &  & * & * & * & *\\
			& 1 &  & \ostar & * & * & * &  &  &  & * & *\\
			&  & 1 & * & \ostar & * & * &  &  &  & * & *\\
			&  &  & 1 &  & \ostar & * &  &  &  &  & *\\
			\hline  &  &  &  & 1 & * & \ostar &  &  &  &  & *\\
			&  &  &  &  & 1 &  & \\
			&  &  &  &  &  & 1 & \\
			& * & * & * & * & * & * & 1 & \ostar & * & * & *\\
			\hline  & * & * & * & * & * & * &  & 1 & \ostar & * & *\\
			&  &  & * & * & * & * &  &  & 1 & \ostar & *\\
			&  &  &  &  & * & * &  &  &  & 1 & \ostar\\
			&  &  &  &  &  &  &  &  &  &  & 1
		\end{array}
		\!\!)
		\]
		\caption{}
		\label{fig: ostar matrices sub2}
	\end{subfigure}
	\vskip\baselineskip
	\begin{subfigure}[b]{0.45\textwidth}
		\centering
		\[
		\qty(\!\!
		\begin{array}{cccc|cccc|cccc}
			1 & * & * & * & * & * & * &  & * & * & * & *\\
			& 1 &  & \ostar & * & * & * &  & * & * & * & *\\
			&  & 1 & * & \ostar & * & * &  & * & * & * & *\\
			&  &  & 1 &  & \ostar & * &  & * & * & * & *\\
			\hline
			&  &  &  & 1 & * & \ostar &  & * & * & * & *\\
			&  &  &  &  & 1 &  &  & * & * & * & *\\
			&  &  &  &  &  & 1 &  & * & * & * & *\\
			&  &  &  &  &  &  & 1 & \ostar & * & * & *\\
			\hline  
			&  &  &  &  &  &  &  & 1 & \ostar & * & *\\
			&  &  &  &  &  &  &  &  & 1 & \ostar & *\\
			&  &  &  &  &  &  &  &  &  & 1 & \ostar\\
			&  &  &  &  &  &  &  &  &  &  & 1
		\end{array}
		\!\!)
		\]
		\caption{}
		\label{fig: ostar matrices sub3}
	\end{subfigure}
	\hfill
	\begin{subfigure}[b]{0.45\textwidth}
		\centering
		\[
		\qty(\!\!
		\begin{array}{cccc|cccc|cccc}
			1 &  &  &  &  &  &  &  & * & * & * & *\\
			& 1 &  &  &  &  &  &  & * & * & * & *\\
			&  & 1 &  &  &  &  &  & * & * & * & *\\
			&  &  & 1 &  &  &  &  & * & * & * & *\\
			\hline  &  &  &  & 1 &  &  &  & * & * & * & *\\
			&  &  &  &  & 1 &  &  & * & * & * & *\\
			&  &  &  &  &  & 1 &  & * & * & * & *\\
			&  &  &  &  &  &  & 1 & \ostar & * & * & *\\
			\hline  &  &  &  &  &  &  &  & 1 & \ostar & * & *\\
			&  &  &  &  &  &  &  &  & 1 & \ostar & *\\
			&  &  &  &  &  &  &  &  &  & 1 & \ostar\\
			&  &  &  &  &  &  &  &  &  &  & 1
		\end{array}
		\!\!)
		\]
		\caption{}
		\label{fig: ostar matrices sub4}
	\end{subfigure}
	\caption{}
	\label{fig: ostar matrices}
	\end{figure}
\end{example}

\subsubsection{Computation of the constant}
		
By Proposition \ref{prop: Z and Z* equivariance} and Theorem \ref{thm:trilinear functional equation}, the forms $\zetaOperator$ and $\dualZetaOperator$ are proportional when the cuspidal supports of $\pi$ and $\Contragradient{\tau}$ do not intersect.

Let $C\in\mathbb{C}$ be such that $C\cdot \zetaOperator = \dualZetaOperator$.
We thus have the equality of operators
\begin{equation}\label{eq:functional-equation-as-equality-of-operators}
C\cdot
\sum_{g\in\GL_{c}(\finiteField)}
W
\begin{pmatrix}
	g&\\
	&I_{(k-1)c}
\end{pmatrix}
\pi(g)
=
\sum_{X\in\Mat{c}{(k-2)c}(\finiteField)}
\sum_{g\in\GL_{c}(\finiteField)}
W
\begin{pmatrix}
	&I_c&\\
	&&I_{(k-2)c}\\
	g&&X
\end{pmatrix}
\pi(g)
.	
\end{equation}
We compute the value of $C$ by substituting $\besselSpehFunction{\tau}{c}$ into $W$ in \eqref{eq:functional-equation-as-equality-of-operators}.
By \Cref{prop:bessel-speh-on-diag-block-matrices}, the summand on the left-hand side of \eqref{eq:functional-equation-as-equality-of-operators} vanishes for $g\neq I_c$.
In every term of the sum on the right-hand side, we can perform a right translation of $\besselSpehFunction{\tau}{c}$ by
\[
\begin{pmatrix}
I_c&&-g^{-1}X\\
&I_c&\\
&&I_{(k-2)c}
\end{pmatrix}
\]
(which does not change its value) and replace the summation over $X$ by multiplication by $q^{(k-2)c^2}$.
Hence we have
\begin{align*}
C\cdot \idmap_{\pi}
&=
q^{(k-2)c^2}
\sum_{g\in \GL_{c}(\finiteField)}
\besselSpehFunction{\tau}{c}
\begin{pmatrix}
&I_{(k-1)c}\\
g&
\end{pmatrix}
\pi(g)
\\
&=
q^{\frac{(k-2)c^2}{2}} \cdot
\GKPreGammaFactor{\pi}{\tau}{\fieldCharacter} \cdot \idmap_{\pi},
\end{align*}
by the definition of $\GKPreGammaFactor{\pi}{\tau}{\fieldCharacter}$ (see \Cref{subsec:gk-non-abelian-gauss-sums}). 
The left-hand side of \eqref{eq:functional-equation-as-equality-of-operators} equals $C\cdot \zetaOperator_{k-2}\qty(W,\pi\times \tau)$ and the right-hand side is equal to $q^{\frac{(k-2)c^2}{2}} \dualZetaOperator_{k-2}\qty(W,\pi\times \tau)$.
Thus
\[
\dualZetaOperator_{k-2}\qty(W,\pi\times \tau)=\GKPreGammaFactor{\pi}{\tau}{\fieldCharacter}\zetaOperator_{k-2}\qty(W,\pi\times \tau),
\]
which finishes the proof of Theorem \ref{thm: functional equation 1}. \qed

\begin{corollary}\label{cor:gamma-factor-norm-equal-1}
	Suppose that the cuspidal support of $\pi$ does not intersect the cuspidal support of  $\Contragradient{\tau}$. Then $$\abs{\GKPreGammaFactor{\pi}{\tau}{\fieldCharacter}} = 1.$$
\end{corollary}
\begin{proof}
	We apply the functional equation twice. First we apply it for $\zetaOperator_j$ and $\dualZetaOperator_j$. Then we use \Cref{rem:z-star-expressed-in-terms-of-Z} and apply the functional equation for $\zetaOperator_{k-2-j}$ and $\dualZetaOperator_{k-2-j}$. We get that \begin{equation}\label{eq:gk-gamma-factor-norm-is-equal-to-one}
		\GKPreGammaFactor{\pi}{\tau}{\fieldCharacter} \GKPreGammaFactor{\Contragradient{\pi}}{\Contragradient{\tau}}{\fieldCharacter^{-1}} = 1.
	\end{equation}
	From \Cref{subsec:ginzburg-kaplan-contragradient} we have $ \GKPreGammaFactor{\Contragradient{\pi}}{\Contragradient{\tau}}{\fieldCharacter^{-1}} = \conjugate{\GKPreGammaFactor{\pi}{\tau}{\fieldCharacter}}$, and the result follows.
\end{proof}
\begin{remark}
	The equality \eqref{eq:gk-gamma-factor-norm-is-equal-to-one} can be seen as an analog of~\cite[Equation (3.14)]{Kaplan2023}.
\end{remark}

\begin{remark}
	As a corollary of the results in the next section and of~\cite[Theorem 2.18]{zelingher2022values} (see also~\cite[Theorem A.1]{SoudryZelingher2023}), if $k = c$ and $\pi$ and $\tau$ are cuspidal such that $\pi \cong \Contragradient{\tau}$, we have $$\GKPreGammaFactor{\pi}{\tau}{\fieldCharacter} = -q^{-\frac{c}{2}}.$$
	In particular in this case, the functional equation does not hold (otherwise \Cref{cor:gamma-factor-norm-equal-1} would have been true). We are only able to prove this using our results regarding level zero representations. It would be interesting to give a representation theoretic proof of this result.
\end{remark}

\subsection{Equality with epsilon factors}

In this section, we establish equality of the Ginzburg--Kaplan gamma factor with the tensor product $\varepsilon_0$-factor.

In~\cite{Zelingher2024b} the second author applies the analogous local integrals of Ginzburg--Kaplan to level zero supercuspidal representations of $\GL_c\left(\localField\right)$ and $\GL_k\left(\localField\right)$, where $\localField$ is a non-archimedean local field with residue field $\finiteField$. He obtains the following result.

\begin{theorem}\label{thm:equality-of-epsilon-gk-factors-for-cuspidal-representations}
	Let $k \ge 2$, and let $\pi$ and $\tau$ be irreducible cuspidal representations of $\GL_c\left(\finiteField\right)$ and $\GL_k\left(\finiteField\right)$, respectively. Then $$\GKGammaFactor{\pi}{\tau}{\fieldCharacter} = \varepsilon_0\left(\pi \times \tau, \fieldCharacter\right).$$
\end{theorem}
See~\cite{ye2021epsilon} for the definition of $\varepsilon_0$-factors. \Cref{thm:equality-of-epsilon-gk-factors-for-cuspidal-representations} also holds for $k=1$ due to the work of Kondo~\cite{Kondo1963}, see also~\cite{Macdonald80}.

Let $\pi$ and $\tau$ be irreducible representations of $\GL_c\left(\finiteField\right)$ and $\GL_k\left(\finiteField\right)$, respectively. Assume that $\tau$ is generic. Suppose that the cuspidal support of $\pi$ is $\left\{\pi_1,\dots,\pi_s \right\}$ and that the cuspidal support of $\tau$ is $\left\{\tau_1,\dots,\tau_t\right\}$. Then we have that $$\varepsilon_0\left(\pi \times \tau, \fieldCharacter\right) = \prod_{i = 1}^s \prod_{j = 1}^t \varepsilon_0\left(\pi_i \times \tau_j, \fieldCharacter\right).$$
On the other hand, by Theorems \ref{thm:multiplicativitiy-in-second-variable} and \ref{thm:multiplicativitiy-in-first-variable} we have
$$\GKGammaFactor{\pi}{\tau}{\fieldCharacter} = \prod_{i=1}^s \prod_{j=1}^t \GKGammaFactor{\pi_i}{\tau_j}{\fieldCharacter}.$$
Hence by \Cref{thm:equality-of-epsilon-gk-factors-for-cuspidal-representations} and the work of Kondo, we obtain equality of these factors for all representations for which the Ginzburg--Kaplan gamma factor is defined.
\begin{theorem}\label{thm:equality-of-GK-factors-and-epsilon-factors}
	Let $\pi$ be an irreducible representation of $\GL_c\left(\finiteField\right)$ and let $\tau$ be an irreducible generic representation of $\GL_k\left(\finiteField\right)$. Then $$ \GKGammaFactor{\pi}{\tau}{\fieldCharacter} = \varepsilon_0\left(\pi \times \tau, \fieldCharacter\right). $$
\end{theorem}

As a corollary, we get the following property of the absolute value of the Ginzburg--Kaplan gamma factor. It is analogous to~\cite[Theorem 4.3]{SoudryZelingher2023}, the main difference being that it is valid for any irreducible representation $\pi$ and not only for generic $\pi$.
\begin{corollary}
	Let $\pi$ be an irreducible representation of $\GL_c\left(\finiteField\right)$ and let $\tau$ be an irreducible cuspidal representation of $\GL_k\left(\finiteField\right)$. Then $$\abs{\GKGammaFactor{\pi}{\Contragradient{\tau}}{\fieldCharacter}} = q^{-\frac{d_{\pi}\left(\tau\right) \cdot k}{2}},$$
	where $d_{\pi}\left(\tau\right)$ is the number of times that $\tau$ appears in the cuspidal support of $\pi$.
\end{corollary}

\section{Applications}\label{sec:special-values-of-bessel-speh-function}

\subsection{Multiplicativity of matrix Kloosterman sums}

Let $\tau$ be a generic principal series representation, that is, suppose that $\tau$ is the unique irreducible generic subrepresentation of the parabolic induction $\alpha_1 \circ \dots \circ \alpha_k$, where $\alpha_1, \dots, \alpha_k \colon \multiplicativegroup{\finiteField} \to \multiplicativegroup{\cComplex}$ are characters, then by applying \Cref{lem:convolution-of-bessel-speh-functions} repeatedly we get for any $h \in \GL_c\left(\finiteField\right)$,
$$ \specialBesselSpeh{\tau}\left(h\right) = q^{-\left(k-1\right)c^2} \sum_{\substack{x_1,\dots,x_k \in \GL_c\left(\finiteField\right)\\
x_1 \cdot \dots \cdot x_k = \left(-1\right)^{k-1} h^{-1} }} \left(\prod_{j=1}^k \alpha_j^{-1}\left(\det x_j\right)\right) \fieldCharacter\left(\sum_{j=1}^k \trace x_j\right).$$
If we denote $\alpha = \alpha_1 \times \dots \times \alpha_k \colon \left(\multiplicativegroup{\finiteField}\right)^k \to \multiplicativegroup{\cComplex}$, this can be rewritten as 
\begin{equation}\label{eq:special-bessel-speh-of-principal-series-is-twisted-matrix-kloosterman-sum}
	 \specialBesselSpeh{\tau}\left(h\right) = q^{-\left(k-1\right)c^2} \ExoticKloosterman(\alpha^{- 1}, \fieldCharacter, \left(-1\right)^{k-1} h^{-1}).
\end{equation}
Here, $\ExoticKloosterman(\alpha^{-1}, \fieldCharacter, \left(-1\right)^{k-1} h^{-1})$ is the twisted matrix Kloosterman sum from ~\cite[Section 2.1]{Zelingher2023}. We use our results to prove the following multiplicativity identity of twisted matrix Kloosterman sums. This identity first appeared in \cite[Theorem 1.1]{erdelyi2021matrix} for the case $k=2$ and $\alpha_1 = \alpha_2 = 1$. Our result is more general, and our proof is entirely different and is based on the identities we proved for the Bessel--Speh function.

\begin{theorem}\label{thm:multiplicativity-of-exotic-kloosterman-sums}
	Let $c = c_1 + c_2$, and let $h_1 \in \GL_{c_1}\left(\finiteField\right)$ and $h_2 \in \GL_{c_2}\left(\finiteField\right)$. Then $$\sum_{n \in \UnipotentRadical_{\left(c_1, c_2\right)}}  \ExoticKloosterman\left(\alpha, \fieldCharacter, n\,\diag\left(h_1, h_2\right)\right) = q^{k c_1 c_2} \ExoticKloosterman\left(\alpha, \fieldCharacter, h_1\right) \ExoticKloosterman\left(\alpha, \fieldCharacter, h_2\right).$$ Moreover, if $h_1$ and $h_2$ do not have any common eigenvalues in the algebraic closure $\algebraicClosure{\finiteField}$, then
	$$\ExoticKloosterman\left( \alpha, \fieldCharacter, \diag\left(h_1, h_2\right) \right) = q^{\left(k - 1\right) c_1 c_2} \ExoticKloosterman\left( \alpha, \fieldCharacter, h_1 \right) \ExoticKloosterman\left( \alpha, \fieldCharacter, h_2 \right).$$
\end{theorem}
\begin{proof}
	Let $h = \diag\left(h_1, h_2\right)$.
	By \Cref{thm:unipotent-fourier-transofrm-bessel-speh}, we have that \begin{align*}
		\frac{1}{\sizeof{\UnipotentRadical_{\left(c_1, c_2\right)}}} \sum_{n \in \UnipotentRadical_{\left(c_1, c_2\right)}} \specialBesselSpeh{\tau}\left(nh\right) &= q^{-c_1 c_2 \left(k-1\right)} \specialBesselSpeh{\tau}\left(h_1\right) \specialBesselSpeh{\tau}\left(h_2\right).
	\end{align*}
	By \eqref{eq:special-bessel-speh-of-principal-series-is-twisted-matrix-kloosterman-sum}, this implies that \begin{align*}
		\MoveEqLeft[3]q^{-\left(k-1\right)c^2-c_1 c_2} \sum_{n \in \UnipotentRadical_{\left(c_1, c_2\right)}} \ExoticKloosterman\left(\alpha, \fieldCharacter, \left(-1\right)^{k-1} h^{-1} n^{-1} \right)\\
		&= q^{-\left(k-1\right)\left(c_1^2 + c_2^2\right) - c_1 c_2 \left(k-1\right)} \ExoticKloosterman\left(\alpha, \fieldCharacter, \left(-1\right)^{k-1} h_1^{-1}\right) \ExoticKloosterman\left(\alpha, \fieldCharacter, \left(-1\right)^{k-1} h_2^{-1}\right).
	\end{align*}
	Replacing $h_j$ with $\left(-1\right)^{k-1} h_j^{-1}$ for $j =1,2$, and replacing $n$ with $n^{-1}$, and using the fact that $h n$ and $n h$ lie in the same conjugacy class, we get the first part.
	
	Suppose that $h_1$ and $h_2$ have no common eigenvalues. Then $nh$ is conjugate to $h$ for any $n \in \UnipotentRadical_{\left(c_1,c_2\right)}$, and therefore, using the fact that $h' \mapsto \ExoticKloosterman\left(\alpha, \fieldCharacter, h'\right)$ is a class function, we get $$\ExoticKloosterman\left(\alpha, \fieldCharacter, h\right) = q^{-c_1 c_2} \sum_{n \in \UnipotentRadical_{\left(c_1, c_2\right)}} \ExoticKloosterman\left(\alpha, \fieldCharacter,n h \right).$$
	Therefore, we get the second part, as required.
\end{proof}

\subsection{A converse theorem}

As another application of our results, we may also state and prove a new converse theorem for generic representations, based on the special values $\specialBesselSpeh{\tau}$.

\begin{theorem}\label{thm:converse-theorem-based-on-bessel-speh-representation}
	Suppose that $\tau_1$ and $\tau_2$ are irreducible generic representations of $\GL_k\left(\finiteField\right)$ with the same central character. Suppose that for every $1 \le c \le \frac{k}{2}$ and for any $h \in \GL_{c}\left(\finiteField\right)$, $$\specialBesselSpeh{\tau_1}\left(h\right) = \specialBesselSpeh{\tau_2}\left(h\right).$$ Then $\tau_1$ and $\tau_2$ are isomorphic.
\end{theorem}
\begin{proof}
	The assumptions of the theorem imply that $\tau_1$ and $\tau_2$ have the same central character and that for every $1 \le c \le \frac{k}{2}$ and every irreducible representation $\pi$ of $\GL_c\left(\finiteField\right)$, $$\GKGammaFactor{\pi}{\tau_1}{\fieldCharacter} = \GKGammaFactor{\pi}{\tau_2}{\fieldCharacter}.$$
	By \Cref{thm:equality-of-GK-factors-and-epsilon-factors} and by~\cite[Theorem 2.19]{zelingher2022values}, this implies that
	$$\ShGammaFactor{\tau_1}{\pi}{\fieldCharacter} = \ShGammaFactor{\tau_2}{\pi}{\fieldCharacter},$$
	for every $0 \le c \le \frac{k}{2}$ and every irreducible cuspidal representation $\pi$ of $\GL_c\left(\finiteField\right)$, where for $i=1,2$, $\ShGammaFactor{\tau_i}{\pi}{\fieldCharacter}$ is the Shahidi gamma factor studied in~\cite{SoudryZelingher2023}. The result now follows from~\cite[Theorem 4.8]{SoudryZelingher2023}. 
\end{proof}
\begin{remark}
	Since characters of cuspidal representations are supported on generalized Jordan matrices, it suffices to check the equality for every $1 \le ab \le \frac{k}{2}$ and every $h = J_{\mu}\left(h\right)$, where $\mu \vdash b$ and where $h \in \GL_{a}\left(\finiteField\right)$ is a regular elliptic element (see \cite[Section 2.2]{Zelingher2023} for the notation).
\end{remark}

\appendix

\section{\texorpdfstring{Expression of the special value of $\besselSpehFunction{\tau}{2}$ in terms of Bessel functions}{Expression of the special value of BS in the case c=2 in terms of Bessel functions}}\label{appendix:appendix-special-value-in-for-c-2}

The goal of this appendix is to prove the following result.

\begin{theorem}\label{thm:GL2-formula-for-voronoi-element}
		Suppose $k \ge 2$ and let $\tau$ be an irreducible cuspidal representation of $\GL_k\left(\finiteField\right)$. 
		Then for any $h = \begin{pmatrix}
			a & b\\
			c & d
		\end{pmatrix} \in \GL_2\left(\finiteField\right)$ we have
		\begin{align*}
			\besselSpehFunction{\tau}{2}\begin{pmatrix}
				& \IdentityMatrix{\left(k-1\right)2}\\
				h
			\end{pmatrix} ={}& q^{-k}\sum_{t \in \multiplicativegroup{\finiteField}} \besselFunction_{\tau, \fieldCharacter}\begin{pmatrix}
				& \IdentityMatrix{k-2}\\
				\left(\begin{smallmatrix}
					t &\\
					& 1
				\end{smallmatrix}\right)h \left(\begin{smallmatrix}
					t &\\
					& 1
				\end{smallmatrix}\right)^{-1}
			\end{pmatrix} \fieldCharacter^{\delta_{k, 2}} \left(t^{-1} c'\right)\\
			& + \delta_{c, 0} q^{-\left(k-1\right)} \besselFunction_{\tau, \fieldCharacter} \begin{pmatrix}
				& \IdentityMatrix{k-1}\\
				a
			\end{pmatrix} \besselFunction_{\tau, \fieldCharacter}\begin{pmatrix}
				& \IdentityMatrix{k-1}\\
				d
			\end{pmatrix},
		\end{align*}
		where $h^{-1} = \begin{pmatrix}
			a' & b'\\
			c' & d'
		\end{pmatrix}$.
\end{theorem}

\subsection{Projection operator for \texorpdfstring{$\SpehRepresentation{\tau}{c}$}{Speh representation}}
Suppose that $\tau$ is an irreducible cuspidal representation of $\GL_k\left(\finiteField\right)$. 
By~\cite[Section 4]{SilbergerZink00}, we have the following projection operator $\ProjectionOperator_{\SpehRepresentation{\tau}{c}} \colon \tau^{\circ c} \to \SpehRepresentation{\tau}{c}$:
$$ \ProjectionOperator_{\SpehRepresentation{\tau}{c}} = \frac{1}{\PoincarePolynomial{k}{c}\left(q^k\right)} \sum_{w \in \SymmetricGroup_c} h^0_w,$$
where $\SymmetricGroup_c$ is the symmetric group on $c$ elements, $$\PoincarePolynomial{k}{c}\left(X\right) = \left(1+X\right)\left(1 + X + X^2\right) \dots \left(1+ X + \dots + X^{c-1}\right)$$ is the Poincare polynomial and $h^0_w$ is given by the Howe isomorphism. 
Let us explain this in more detail. 
For any permutation $w \in \SymmetricGroup_c$, we may realize $w$ as a $kc \times kc$ block matrix, with blocks of size $k$, in the following way. 
First represent $w$ as $c \times c$ matrix, and then replace each zero entry in the matrix with the matrix $0_k$, and each 1 entry in the matrix with $\IdentityMatrix{k}$. 
Now let $h_w \colon \tau^{\circ c} \to \tau^{\circ c}$ be the operator corresponding to the characteristic function of the double coset $\ParabolicForSpeh{k}{c} w \ParabolicForSpeh{k}{c}$ by Mackey theory, that is, for $f \in \tau^{\circ c}$ and $x \in \GL_{kc}\left(\finiteField\right)$,
$$ \left(h_w f\right)\left(x\right) = \sum_{g \in (\ParabolicForSpeh{k}{c} w \ParabolicForSpeh{k}{c}) \slash \ParabolicForSpeh{k}{c}} f^w \left(g^{-1} x\right) = \sum_{u \in \UnipotentForSpeh{k}{c} \slash  (\UnipotentForSpeh{k}{c} \cap w \UnipotentForSpeh{k}{c} w^{-1}) } f^w \left(w^{-1} u^{-1} x\right),$$
where for $w \in \SymmetricGroup_c$ and $g \in \GL_{kc}\left(\finiteField\right)$, we denote $f^{w}\left(g\right) = w f\left(g\right)$, where $\SymmetricGroup_c$ acts on $\tau^{\otimes c}$ by permuting the components of pure tensors. 
We then normalize $h_w$ as Howe~\cite[Page 12]{Howe1985}, by setting $$h_w^0 = q^{-\frac{k\left(k-1\right)}{2} \ell\left(w\right)} \centralCharacter{\tau}\left(-1\right)^{\ell \left(w\right)} h_w,$$ where $\ell \left(w\right)$ is the length of the permutation $w$.

\subsubsection{Simple functions}

Given an element $v_{\tau^{\otimes c}} \in \tau^{\otimes c}$, we let $f^{0}_{v_{\tau^{\otimes c}}} \in \tau^{\circ c}$ be the function supported on $\ParabolicForSpeh{k}{c}$ such that $f^{0}_{v_{\tau^{\otimes c}}}\left( \IdentityMatrix{kc} \right) = v_{\tau^{\otimes c}}$. 
Applying the projection operator $\ProjectionOperator_{\SpehRepresentation{\tau}{c}}$ to $f^0_{v_{\tau^{\otimes c}}}$ yields an element in the Speh representation $\SpehRepresentation{\tau}{c}$. 
For convenience, we multiply by the scalar $\PoincarePolynomial{k}{c}\left(q^k\right)$ and set
$$f_{v_{\tau^{\otimes c}}} = \PoincarePolynomial{k}{c}\left(q^k\right) \ProjectionOperator_{\SpehRepresentation{\tau}{c}} f^0_{v_{\tau^{\otimes c}}}.$$
Then $f_{v_{\tau^{\otimes c}}}$ is the function supported on double cosets $\ParabolicForSpeh{k}{c} w \ParabolicForSpeh{k}{c}$ for $w \in \SymmetricGroup_c$, such that for $w \in \SymmetricGroup_c$ and $u \in \UnipotentForSpeh{k}{c}$, $$f_{v_{\tau^{\otimes c}}}\left( w u \right) = q^{-\ell\left(w\right) \frac{k\left(k-1\right)}{2}} \centralCharacter{\tau}\left(-1\right)^{\ell\left(w\right)} \left(w v_{\tau^{\otimes c}}\right),$$
where we remind that $\SymmetricGroup_c$ acts on $\tau^{\otimes c}$ by permuting components of pure tensors.

\subsection{Another realization of the \texorpdfstring{$\kcNotation{k}{c}{\fieldCharacter}$ }{(k,c) }model of \texorpdfstring{$\SpehRepresentation{\tau}{c}$}{the Speh representation}}

Recall the recursive expression for a $\kcNotation{k}{c}{\fieldCharacter}$ functional from \Cref{sec:wss-models}. Using this expression repeatedly, we may obtain a formula for the $\kcNotation{k}{c}{\fieldCharacter}$ functional, where we realize $\SpehRepresentation{\tau}{c}$ as a subrepresentation of $\tau^{\circ c}$. In order to describe this formula, let $\sigma_{k,c}$ be the following permutation of $\left\{1,2,\dots,kc\right\}$: $$\sigma_{k,c}\left(1 + a + bc \right) = 1 + b + ak,$$ for any $0 \le a \le c - 1$ and $0 \le b \le k-1$. 
Let $\sigma = \sigma_{k, c}$ be the column permutation matrix corresponding to $\sigma_{k,c}$, that is, $$\sigma = \begin{pmatrix}
	\transpose{e_{\sigma_{k,c}\left(1\right)}}, \dots,  \transpose{e_{\sigma_{k,c}\left(kc\right)}}
\end{pmatrix},$$ where $\transpose{e_j}$ is the standard $j$-th column vector.

Let $\UnipotentRadicalForWssRecursion{k}{c}$ be the following subgroup of $\GL_{kc}\left(\finiteField\right)$: $$\left\{ \left. \begin{pmatrix}
	\IdentityMatrix{k} \\
	X_{21} & \IdentityMatrix{k}\\
	X_{31 }& X_{32} & \IdentityMatrix{k} \\
	\vdots & \vdots & \ddots & \ddots &\\
	X_{c1} & X_{c2} &\cdots & X_{c, c-1} & \IdentityMatrix{k} 
\end{pmatrix} \right| X_{i j} \text{ is a strictly upper triangular matrix} \right\}.$$

Then we have the following formula.
\begin{theorem} \label{thm:formula-for-k-c-functional-big-model}
	For $f \in \SpehRepresentation{\tau}{c} \subset \tau^{\circ c}$ we have $$\innerproduct{f}{\gShortSpehWhittakerFunctional{\tau}{k}{c}} = \frac{1}{\sizeof{\UnipotentRadicalForWssRecursion{k}{c}}}\sum_{n \in \UnipotentRadicalForWssRecursion{k}{c}} \standardForm{f\left(n \sigma\right)}{\WhittakerFunctional{\tau} \otimes \dots \otimes \WhittakerFunctional{\tau}}.$$ 
\end{theorem}

\subsubsection{Expression of the Bessel--Speh function of $\SpehRepresentation{\tau}{c}$ using projection operators}

Let $\tau$ be an irreducible cuspidal representation of $\GL_k\left(\finiteField\right)$. 
Choose a $\GL_k\left(\finiteField\right)$ invariant inner product $\innerproduct{\cdot}{\cdot}_{\tau}$ on $\tau$. 
It induces an invariant inner product $\innerproduct{\cdot}{\cdot}_{\tau^{\otimes c}}$ on $\tau^{\otimes c}$, which induces an invariant inner product on $\tau^{\circ c}$ given by
$$ \innerproduct{f_1}{f_2} = \sum_{g \in \ParabolicSubgroup_{\left(k\right)^c} \backslash \GL_{kc}\left(\finiteField\right)} \innerproduct{f_1\left(g\right)}{f_2\left(g\right)}_{\tau^{\otimes c}},$$
where $f_1, f_2 \in \tau^{\circ c}$. 

We have a projection operator $\WhittakerProjection \colon \tau^{\circ c} \to \tau^{\circ c}$ from $\tau^{\circ c}$ to its space of $\kcNotation{k}{c}{\fieldCharacter}$ vectors, given by $$\WhittakerProjection = \frac{1}{\sizeof{\UnipotentRadicalForWss{k}{c}}} \sum_{u \in \UnipotentRadicalForWss{k}{c}} \fieldCharacterkc{k}{c}^{-1}\left(u\right) \tau^{\circ c}\left(u\right).$$

Let $0 \ne \whittakerVector{\tau} \in \tau$ be a $\fieldCharacter$-Whittaker vector, normalized such that $\innerproduct{\whittakerVector{\tau}}{\whittakerVector{\tau}}_\tau = 1$. 
Consider the function $f_{\tau^{\circ c}}^{\mathrm{Wh}} \in \tau^{\circ c}$, given by $ f_{\tau^{\circ c}}^{\mathrm{Wh}} = \WhittakerProjection \tau^{\circ c}\left(\sigma^{-1}\right) f^0_{\whittakerVector{\tau} \otimes \dots \otimes \whittakerVector{\tau}}$. 
We will use $f_{\tau^{\circ c}}^{\mathrm{Wh}}$ and $\ProjectionOperator_{\SpehRepresentation{\tau}{c}}$ in order to give an expression for the Bessel--Speh function $\besselSpehFunction{\tau}{c}$. 
We have that $\ProjectionOperator_{\SpehRepresentation{\tau}{c}} f_{\tau^{\circ c}}^{\mathrm{Wh}}$ is a $\kcNotation{k}{c}{\fieldCharacter}$ vector of $\SpehRepresentation{\tau}{c}$. 
We will show that this element is not zero and express the Bessel--Speh function it defines in terms of the simple function $f_{\whittakerVector{\tau} \otimes \dots \otimes \whittakerVector{\tau}}$.

\begin{proposition}\label{prop:projection-of-simple-function-to-whittaker-space}
	\begin{enumerate}
		\item For any $g \in \GL_{kc}\left(\finiteField\right)$, we have that \begin{align*}
			\MoveEqLeft[3] \innerproduct{\tau^{\circ c}\left(g\right) \ProjectionOperator_{\SpehRepresentation{\tau}{c}} f_{\tau^{\circ c}}^{\mathrm{Wh}}}{\ProjectionOperator_{\SpehRepresentation{\tau}{c}} f_{\tau^{\circ c}}^{\mathrm{Wh}}} \\
			&= \frac{1}{\PoincarePolynomial{k}{c}\left(q^k\right) \sizeof{\UnipotentRadicalForWssRecursion{k}{c}}^2 } \sum_{u_1, u_2 \in \UnipotentRadicalForWssRecursion{k}{c} } \innerproduct{f_{\whittakerVector{\tau} \otimes \dots \otimes \whittakerVector{\tau} }\left(u_1 \sigma g \sigma^{-1} u_2 \right)}{\whittakerVector{\tau} \otimes \dots \otimes \whittakerVector{\tau}}_{\tau \otimes \dots \otimes \tau},
		\end{align*}
		where $\UnipotentRadicalForWssRecursion{k}{c} = \UnipotentRadicalForWssRecursion{k}{c}\left(\finiteField\right)$.
		\item $$\innerproduct{\ProjectionOperator_{\SpehRepresentation{\tau}{c}} f_{\tau^{\circ c}}^{\mathrm{Wh}}}{\ProjectionOperator_{\SpehRepresentation{\tau}{c}} f_{\tau^{\circ c}}^{\mathrm{Wh}}} = \frac{1}{\sizeof{\UnipotentRadicalForWssRecursion{k}{c}} \cdot \PoincarePolynomial{k}{c}\left(q^k\right)} \ne 0.$$
	\end{enumerate}
\end{proposition}
\begin{proof}
	Since the projection operators are self adjoint, the inner product in the proposition is given by $$\innerproduct{\tau^{\circ c}\left(g\right) \ProjectionOperator_{\SpehRepresentation{\tau}{c}} f_{\tau^{\circ c}}^{\mathrm{Wh}}}{ f_{\tau^{\circ c}}^{\mathrm{Wh}}} = \innerproduct{\WhittakerProjection \tau^{\circ c}\left(g\right) \ProjectionOperator_{\SpehRepresentation{\tau}{c}}f_{\tau^{\circ c}}^{\mathrm{Wh}} }{\tau^{\circ c}\left(\sigma^{-1}\right) f^0_{\whittakerVector{\tau} \otimes \dots \whittakerVector{\tau}}}.$$
	Since $\ProjectionOperator_{\SpehRepresentation{\tau}{c}}$, commutes with the action of $\tau^{\circ c}$ (which implies that it also commutes with $\WhittakerProjection$), we get that this inner product equals $$\frac{1}{\PoincarePolynomial{k}{c}\left(q^k\right)} \innerproduct{\WhittakerProjection \tau^{\circ c}\left(g\right) \WhittakerProjection \tau^{\circ c}\left(\sigma^{-1}\right) f_{\whittakerVector{\tau} \otimes \dots \otimes \whittakerVector{\tau} }}{\tau^{\circ c}\left(\sigma^{-1}\right) f^0_{\whittakerVector{\tau} \otimes \dots \whittakerVector{\tau}}}.$$
	Since $\tau^{\circ c}\left(\sigma^{-1}\right) f^0_{\whittakerVector{\tau} \otimes \dots \whittakerVector{\tau}}$ is supported on $\ParabolicForSpeh{k}{c} \sigma$, this evaluates to $$\frac{1}{\PoincarePolynomial{k}{c}\left(q^k\right)} \innerproduct{\left(\WhittakerProjection \tau^{\circ c}\left(g\right) \WhittakerProjection \tau^{\circ c}\left(\sigma^{-1}\right) f_{\whittakerVector{\tau} \otimes \dots \otimes \whittakerVector{\tau} }\right)\left(\sigma\right) }{\whittakerVector{\tau} \otimes \dots \whittakerVector{\tau}}_{\tau \otimes \dots \otimes \tau}.$$
	We have that \begin{align*}
		\MoveEqLeft[3]\left(\WhittakerProjection \tau^{\circ c}\left(g\right) \WhittakerProjection \tau^{\circ c}\left(\sigma^{-1}\right) f_{\whittakerVector{\tau} \otimes \dots \otimes \whittakerVector{\tau} }\right)\left(\sigma\right)\\
		&= \frac{1}{\sizeof{\UnipotentRadicalForWss{k}{c}}^2} \sum_{u_1, u_2 \in \UnipotentRadicalForWss{k}{c}} f_{\whittakerVector{\tau} \otimes \dots \otimes \whittakerVector{\tau} }\left(\sigma u_1 g u_2 \sigma^{-1} \right) \fieldCharacterkc{k}{c}^{-1}\left(u_1\right) \fieldCharacterkc{k}{c}^{-1}\left(u_2\right).
	\end{align*} Using the $\fieldCharacter$-equivariance property of $\whittakerVector{\tau}$ and the $\UnipotentForSpeh{k}{c}$ invariance property of $f_{\whittakerVector{\tau} \otimes \dots \otimes \whittakerVector{\tau} }$, the sum reduces to $$\frac{1}{\sizeof{\UnipotentRadicalForWssRecursion{k}{c}}^2} \sum_{u_1, u_2 \in \UnipotentRadicalForWssRecursion{k}{c} } f_{\whittakerVector{\tau} \otimes \dots \otimes \whittakerVector{\tau} }\left(u_1 \sigma g \sigma^{-1} u_2 \right).$$
	This proves the first part.
	
	For the second part, we substitute $g= \IdentityMatrix{kc}$ and by changing variable we get $$\frac{1}{\sizeof{\UnipotentRadicalForWssRecursion{k}{c}}^2} \sum_{u_1, u_2 \in \UnipotentRadicalForWssRecursion{k}{c} } f_{\whittakerVector{\tau} \otimes \dots \otimes \whittakerVector{\tau} }\left(u_1  u_2 \right) = \frac{1}{\sizeof{\UnipotentRadicalForWssRecursion{k}{c}}} \sum_{u \in \UnipotentRadicalForWssRecursion{k}{c} } f_{\whittakerVector{\tau} \otimes \dots \otimes \whittakerVector{\tau} }\left( u \right).$$
	
	Since the matrices $X_{ij}$ in $u \in \UnipotentRadicalForWssRecursion{k}{c}$ are strictly upper triangular, the Bruhat decomposition of $u$ cannot have a permutation part that lies in $\SymmetricGroup_c$, unless $X_{ij} = 0$ for all $i,j$. 
	Therefore, we get that 
	\begin{equation*}
	\innerproduct{\ProjectionOperator_{\SpehRepresentation{\tau}{c}} f_{\tau^{\circ c}}^{\mathrm{Wh}}}{\ProjectionOperator_{\SpehRepresentation{\tau}{c}} f_{\tau^{\circ c}}^{\mathrm{Wh}}} =  \frac{1}{\sizeof{\UnipotentRadicalForWssRecursion{k}{c}} \cdot \PoincarePolynomial{k}{c}\left(q^k\right)} \ne 0.\qedhere
	\end{equation*}
\end{proof}

Using \Cref{prop:projection-of-simple-function-to-whittaker-space} we deduce the following expression for the Bessel--Speh function of $\SpehRepresentation{\tau}{c}$.

\begin{corollary}\label{cor:expression-of-bessel-speh-in-terms-of-simple-function}
	For any $g \in \GL_{kc}\left(\finiteField\right)$,
	$$ \besselSpehFunction{\tau}{c}\left(g\right) = \frac{1}{ \sizeof{\UnipotentRadicalForWssRecursion{k}{c}} } \sum_{u_1, u_2 \in \UnipotentRadicalForWssRecursion{k}{c} } \innerproduct{f_{\whittakerVector{\tau} \otimes \dots \otimes \whittakerVector{\tau} }\left(u_1 \sigma g \sigma^{-1} u_2 \right)}{\whittakerVector{\tau} \otimes \dots \otimes \whittakerVector{\tau}}.$$
\end{corollary}

\subsubsection{Proof of the main theorem}
We are now ready to prove \Cref{thm:GL2-formula-for-voronoi-element}.

\begin{proof}
	Let $g = \begin{pmatrix}
		& \IdentityMatrix{\left(k-1\right)2}\\
		h
	\end{pmatrix}$. 
	By \Cref{cor:expression-of-bessel-speh-in-terms-of-simple-function}, we need to consider the sum
	$$\sum_{u_1, u_2 \in \UnipotentRadicalForWssRecursion{k}{2} } f_{\whittakerVector{\tau} \otimes \whittakerVector{\tau} }\left(u_1 \sigma g \sigma^{-1} u_2 \right).$$
	We have that $$\sigma g \sigma^{-1} = \begin{pmatrix}
		& \IdentityMatrix{k-1} & & 0_{k-1}\\
		a & & b & \\
		& 0_{k-1} & & \IdentityMatrix{k-1}\\
		c & & d & 
	\end{pmatrix}.$$ 
	Write $u_1 = \begin{pmatrix}
		\IdentityMatrix{k} & \\
		X & \IdentityMatrix{k}
	\end{pmatrix}$ and $u_2 = \begin{pmatrix}
		\IdentityMatrix{k} & \\
		Y & \IdentityMatrix{k}
	\end{pmatrix}$, where $$X = \begin{pmatrix}
		0 & x_{12} & x_{13} & \cdots & x_{1k}\\
		& 0 & x_{23} & \cdots & x_{2k}\\
		& & \ddots & \ddots & \vdots\\
		& & & 0 & x_{k-1, k}\\
		& & & & 0
	\end{pmatrix} \text{ and } Y = \begin{pmatrix}
		0 & y_{12} & y_{13} & \cdots & y_{1k}\\
		& 0 & y_{23} & \cdots & y_{2k}\\
		& & \ddots & \ddots & \vdots\\
		& & & 0 & y_{k-1, k}\\
		& & & & 0
	\end{pmatrix}.$$
	We set $$X' = \begin{pmatrix}
		x_{12} & x_{13}  & \cdots & x_{1,k-1}\\
		& x_{23} & \cdots & x_{2,k-1}  \\
		& & \ddots & \vdots  \\
		&  & & x_{k-2, k-1}\\
	\end{pmatrix},\ X'' = \begin{pmatrix}
		x_{1k}\\
		\vdots\\
		x_{k-2,k}
	\end{pmatrix},$$
	and 
	$$Y' = \begin{pmatrix}
		y_{23} & y_{24}  & \cdots & y_{2k}\\
		& y_{34} & \cdots & y_{3k}  \\
		& & \ddots & \vdots  \\
		&  & & y_{k-1, k}\\
	\end{pmatrix},\ Y'' = \begin{pmatrix}
		y_{13} &
		\cdots &
		y_{1k}
	\end{pmatrix}.$$ 
	Then
	$$X = \begin{pmatrix}
		0_{\left(k-2\right) \times 1} & X' & X''\\
		& 0_{1 \times \left(k-2\right)} & x_{k-1,k}\\
		& & 0
	\end{pmatrix} \text{ and } Y = \begin{pmatrix}
		0 & y_{12} & Y''\\
		& 0_{\left(k-2\right) \times 1} & Y'\\
		& & 0_{1 \times \left(k-2\right)}
	\end{pmatrix}.$$
	
	We have \begin{equation} \label{eq:c-equals-2-summand}
		u_1 \sigma g \sigma^{-1} u_2 = \begin{pmatrix}
			& 1 & & & 0_{k-2}\\
			& & \IdentityMatrix{k-2} & & & 0\\
			a & b y_{12} & bY'' & b\\
			X''a & X'' b y_{12} & X' + Y' + X'' b Y'' & X'' b & \IdentityMatrix{k-2} & \\
			x_{k-1,k}a & x_{k-1,k}b y_{12} & x_{k-1,k} b Y'' & x_{k-1,k}b &  & 1\\
			c & dy_{12} & dY'' & d & 
		\end{pmatrix}.
	\end{equation}
	
	We are interested in checking for which $X$ and $Y$ the matrix in \eqref{eq:c-equals-2-summand} is in the support of $f_{\whittakerVector{\tau} \otimes \whittakerVector{\tau}}$. 
	We have two options to consider.
	
	The first option is $u_1 \sigma g \sigma^{-1} u_2 \in \ParabolicForSpeh{k}{c}$, which can only happen if $c = 0$. 
	In this case, we must have $a \ne 0$ and $d \ne 0$ and therefore $X''$, $Y''$, $y_{12}$ and $x_{k-1,k}$ are all zero, and $X' = -Y'$. 
	There are $q^{\binom{k-1}{2}}$ options for $X$ and $Y$ whose entries satisfy these conditions. 
	We get in this case $$u_1 \sigma g \sigma^{-1} u_2 = \begin{pmatrix}
		& \IdentityMatrix{k-1} & & 0_{k-1}\\
		a & & b\\
		& 0_{k-1} & & \IdentityMatrix{k-1}\\
		0 & & d
	\end{pmatrix},$$ and therefore \begin{equation} \label{eq:c-equals-2-case-1-parabolic-element}
		f_{\whittakerVector{\tau} \otimes \whittakerVector{\tau}}\left( u_1 \sigma g \sigma^{-1} u_2 \right) = \tau\begin{pmatrix}
			& \IdentityMatrix{k-1}\\
			a
		\end{pmatrix} \otimes \tau\begin{pmatrix}
			& \IdentityMatrix{k-1}\\
			d
		\end{pmatrix} \whittakerVector{\tau} \otimes \whittakerVector{\tau}.
	\end{equation}
	
	The second option is $u_1 \sigma g \sigma^{-1} u_2 \in \ParabolicForSpeh{k}{c} \begin{pmatrix}
		& \IdentityMatrix{k}\\
		\IdentityMatrix{k}
	\end{pmatrix} \ParabolicForSpeh{k}{c}$, which implies that $$u_1 \sigma g \sigma^{-1} u_2 = \begin{pmatrix}
		h_1 & A\\
		& h_2
	\end{pmatrix} \begin{pmatrix}
		& \IdentityMatrix{k}\\
		\IdentityMatrix{k}
	\end{pmatrix} \begin{pmatrix}
		\IdentityMatrix{k} & B\\
		& \IdentityMatrix{k}
	\end{pmatrix} = \begin{pmatrix}
		A & AB + h_1\\
		h_2 & h_2 B
	\end{pmatrix}$$ for some $A, B \in \squareMatrix_{k}\left(\finiteField\right)$ and $h_1, h_2 \in \GL_k\left(\finiteField\right)$. 
	In particular, this means that the left bottom $k \times k$ block of \eqref{eq:c-equals-2-summand} is invertible and equals $h_2$. 
	We decompose this block into
	\begin{equation}\label{eq:c-equals-2-decomposition-of-left-bottom-block}
		\begin{split}
			h_2 = {} &\begin{pmatrix}
				X''a & X'' b y_{12} & X' + Y' + X'' b Y''\\
				x_{k-1,k}a & x_{k-1,k}b y_{12} & x_{k-1,k} b Y''\\
				c & dy_{12} & dY''
			\end{pmatrix} \\
			={}& \begin{pmatrix}
				\IdentityMatrix{k-2} & X''\\
				& x_{k-1,k}\\
				& & 1
			\end{pmatrix} \begin{pmatrix}
				& & \IdentityMatrix{k-2}\\
				a & b &\\
				c & d &
			\end{pmatrix} \begin{pmatrix}
				1 &\\
				& y_{12} & Y''\\
				& & X' + Y'
			\end{pmatrix}.
		\end{split}
	\end{equation}
	In order for \eqref{eq:c-equals-2-decomposition-of-left-bottom-block} to be invertible, we must have that $x_{k-1,k}$ and $y_{12}$ are non-zero, and that $X' + Y'$ is an invertible matrix, which is equivalent to requiring $x_{j, j+1} + y_{j+1, j+2} \ne 0$ for every $1 \le j \le k-2$. 
	Suppose that we are in this case. 
	Then $$h_2 = \begin{pmatrix}
		X''a & X'' b y_{12} & X' + Y' + X'' b Y''\\
		x_{k-1,k}a & x_{k-1,k}b y_{12} & x_{k-1,k} b Y''\\
		c & dy_{12} & dY''
	\end{pmatrix} \text{ and } A = \begin{pmatrix}
		& 1 \\
		& & \IdentityMatrix{k-2}\\
		a & by_{12} & bY''
	\end{pmatrix},$$
	and also $$h_2 B = \begin{pmatrix}
		X'' b & \IdentityMatrix{k-2} &\\
		x_{k-1,k}b & & 1\\
		d & &
	\end{pmatrix} \text{ and } AB + h_1 = \begin{pmatrix}
		& 0_{k-1}\\
		b
	\end{pmatrix}.$$
	Then $$h_2^{-1} = \begin{pmatrix}
		1\\
		& y_{12}^{-1} & -y_{12}^{-1} Y'' \left(X' + Y'\right)^{-1}\\
		& & \left(X' + Y'\right)^{-1}
	\end{pmatrix} \begin{pmatrix}
		& a' & b'\\
		& c' & d'\\
		\IdentityMatrix{k-2}
	\end{pmatrix} \begin{pmatrix}
		\IdentityMatrix{k-2} & -X'' x_{k-1, k}^{-1} & \\
		& x_{k-1, k}^{-1}\\
		& & 1
	\end{pmatrix}.$$ 
	This implies that \begin{align*}
		B &= h_2^{-1} \begin{pmatrix}
			X'' b & \IdentityMatrix{k-2} &\\
			x_{k-1,k}b & & 1\\
			d & &
		\end{pmatrix} \\
		&= \begin{pmatrix}
			& & a' x_{k-1, k}^{-1}\\
			y_{12}^{-1} & -y_{12}^{-1} Y'' \left(X' + Y'\right)^{-1} & y_{12}^{-1}\left(c' + Y'' \left(X' + Y'\right)^{-1} X'' \right) x_{k-1,k}^{-1}\\
			& \left(X' + Y'\right)^{-1} & -\left(X' + Y'\right)^{-1} X'' x_{k-1,k}^{-1}.
		\end{pmatrix}.
	\end{align*}
	
	Hence, \begin{equation}\label{eq:c-equals-2-formula-for-h1}
		\begin{split}
			h_1 &= \begin{pmatrix}
				& 0_{k-1}\\
				b
			\end{pmatrix} - AB \\
			&= -\begin{pmatrix}
				y_{12}^{-1} & -y_{12}^{-1} Y'' \left(X' + Y'\right)^{-1} & y_{12}^{-1} \left(c' + Y'' \left(X' + Y'\right)^{-1} X''\right) x_{k-1,k}^{-1}\\
				& \left(X' + Y'\right)^{-1} & -\left(X' + Y'\right)^{-1} X'' x_{k-1,k}^{-1}\\
				& & x_{k-1,k}^{-1}
			\end{pmatrix}\\
			& = -\begin{pmatrix}
				y_{12} & Y'' & -c'\\
				& X' + Y' & X''\\
				& & x_{k-1,k}
			\end{pmatrix}^{-1}.
		\end{split}
	\end{equation}
	Therefore, for $X$ and $Y$ satisfying $y_{12}$, $x_{k-1,k} \ne 0$ and $x_{j, j+1} + y_{j+1, j+2} \ne 0$ for every $1 \le j \le k-2$, we have $$f_{\whittakerVector{\tau} \otimes \whittakerVector{\tau} }\left(u_1 \sigma g \sigma^{-1} u_2 \right) = q^{-\binom{k}{2}} \centralCharacter{\tau}\left(-1\right) \tau\left( h_2 \right) \otimes \tau\left(h_1\right) \whittakerVector{\tau} \otimes \whittakerVector{\tau},$$ where $h_1$ is given by \eqref{eq:c-equals-2-formula-for-h1} and $h_2$ is given by \eqref{eq:c-equals-2-decomposition-of-left-bottom-block}. 
	We have that
	\begin{equation} \label{eq:c-equals-2-case-2-in-terms-of-bessel-functions}
		\innerproduct{f_{\whittakerVector{\tau} \otimes \whittakerVector{\tau} }\left(u_1 \sigma g \sigma^{-1} u_2 \right)}{\whittakerVector{\tau} \otimes \whittakerVector{\tau}} = q^{-\binom{k}{2}} \centralCharacter{\tau}\left(-1\right) \besselFunction_{\tau, \fieldCharacter}\left(h_2\right) \besselFunction_{\tau, \fieldCharacter}\left(h_1\right).
	\end{equation}
	Since $h_1$ is an upper triangular matrix, $\besselFunction_{\tau, \fieldCharacter}\left(h_1\right)$ is zero unless the diagonal entries of $h_1$ are all the same, which is equivalent to $y_{12} = x_{k-1,k} = x_{1,2} + y_{2,3} = \dots = x_{k-2,k-1} + y_{k-1,k} = t$ for some $t \in \multiplicativegroup{\finiteField}$. 
	Given such $t \in \multiplicativegroup{\finiteField}$ there are exactly $q^{2 \binom{k}{2} - k}$ options for $X$ and $Y$ satisfying the latter condition. 
	The $\fieldCharacter$-equivariance properties of the Bessel function imply that $$\fieldCharacter^{-1}\left(t^{-1}\left(X' + Y'\right)\right) \besselFunction_{\tau, \fieldCharacter}\left(h_1\right) = \begin{cases}
		\centralCharacter{\tau}\left(-t^{-1}\right) \fieldCharacter\left(-t^{-1} y_{23} - t^{-1} x_{k-2, k} \right) & k \ge 3,\\
		\centralCharacter{\tau}\left(-t^{-1}\right) \fieldCharacter\left(t^{-1} c'\right) & k = 2.
	\end{cases},$$ and that
	$$\fieldCharacter\left(t^{-1} \left(X' + Y'\right) \right) \besselFunction_{\tau, \fieldCharacter}\left(h_2\right) = \begin{cases}
		\fieldCharacter\left(t^{-1} y_{23} + t^{-1} x_{k-2, k} \right) \besselFunction_{\tau, \fieldCharacter} \begin{pmatrix}
			& & t\IdentityMatrix{k-2}\\
			t a& t^2b\\
			c& td
		\end{pmatrix} & k \ge 3,\\
		\besselFunction_{\tau, \fieldCharacter} \begin{pmatrix}
			& & t\IdentityMatrix{k-2}\\
			t a& t^2 b\\
			c& td
		\end{pmatrix} & k = 2.
	\end{cases}$$
	Therefore, we have that  \begin{equation}\label{eq:c-equals-2-case-2-bessel-function-expression}
		\besselFunction_{\tau, \fieldCharacter}\left(h_1\right) \besselFunction_{\tau, \fieldCharacter}\left(h_2\right) = \centralCharacter{\tau}\left(-1\right) \besselFunction_{\tau, \fieldCharacter} \begin{pmatrix}
			& & \IdentityMatrix{k-2}\\
			a& t b\\
			t^{-1} c& d
		\end{pmatrix} \fieldCharacter\left(t^{-1} c'\right)^{\delta_{k,2}}.
	\end{equation}
	Combining \eqref{eq:c-equals-2-case-1-parabolic-element}, \eqref{eq:c-equals-2-case-2-in-terms-of-bessel-functions}, \eqref{eq:c-equals-2-case-2-bessel-function-expression} and the counting of the elements, we get by \Cref{cor:expression-of-bessel-speh-in-terms-of-simple-function},
	\begin{align*}
		q^{\binom{k}{2}} \besselSpehFunction{\tau}{2} \begin{pmatrix}
			& & \IdentityMatrix{\left(k-1\right)2}\\
			a & b\\
			c & d
		\end{pmatrix} ={}& \delta_{c, 0} q^{\binom{k-1}{2}} \cdot \besselFunction_{\tau, \fieldCharacter}\begin{pmatrix}
			& \IdentityMatrix{k-1}\\
			a
		\end{pmatrix} \begin{pmatrix}
			& \IdentityMatrix{k-1}\\
			d
		\end{pmatrix} \\
		& +  q^{2\binom{k}{2} - k -\binom{k}{2}} \sum_{t \in \multiplicativegroup{\finiteField}} \besselFunction_{\tau, \fieldCharacter} \begin{pmatrix}
			& & \IdentityMatrix{k-2}\\
			a& t b\\
			t^{-1} c& d
		\end{pmatrix} \fieldCharacter\left(t^{-1} c'\right)^{\delta_{k,2}}.
	\end{align*}
	The theorem now follows because $\binom{k}{2} - \binom{k - 1}{2} = k - 1$ and $$\begin{pmatrix}
		t & \\
		& 1
	\end{pmatrix} \begin{pmatrix}
		a & b\\
		c & d
	\end{pmatrix} \begin{pmatrix}
		t^{-1} & \\
		& 1
	\end{pmatrix} = \begin{pmatrix}
		a & bt\\
		ct^{-1} & d
	\end{pmatrix}.$$
\end{proof}

\bibliographystyle{abbrv}
\bibliography{references}

\begin{thebibliography}{10}

\bibitem{AlperinJames1995}
J.~L. Alperin and G.~D. James.
\newblock Bessel functions on finite groups.
\newblock {\em J. Algebra}, 171(2):524--530, 1995.

\bibitem{CaiFriedbergGinzburgKaplan2019}
Y.~Cai, S.~Friedberg, D.~Ginzburg, and E.~Kaplan.
\newblock Doubling constructions and tensor product {$L$}-functions: the linear
  case.
\newblock {\em Invent. Math.}, 217(3):985--1068, 2019.

\bibitem{CaiFriedbergGourevitchKaplan2023}
Y.~Cai, S.~Friedberg, D.~Gourevitch, and E.~Kaplan.
\newblock The generalized doubling method: {$(k,c)$} models.
\newblock {\em Proc. Amer. Math. Soc.}, 151(7):2831--2845, 2023.

\bibitem{CaiFriedbergKaplan2022}
Y.~Cai, S.~Friedberg, and E.~Kaplan.
\newblock The generalized doubling method: local theory.
\newblock {\em Geom. Funct. Anal.}, 32(6):1233--1333, 2022.

\bibitem{kaplan2018}
Y.~Cai, S.~Friedberg, and E.~Kaplan.
\newblock Doubling constructions: global functoriality for non-generic cuspidal
  representations.
\newblock {\em Ann. of Math. (2)}, 200(3):893--966, 2024.

\bibitem{Carmon2023}
O.~Carmon.
\newblock {S}halika {V}ectors for {I}rreducible {R}epresentations of
  $\mathrm{GL}_n$ over a {F}inite {F}ield.
\newblock Master's thesis, Tel Aviv University, 2023.
\newblock
  \href{https://tau.primo.exlibrisgroup.com/discovery/delivery/972TAU_INST:TAU/12435684040004146}{DaTA:9933623517704146}.

\bibitem{curtis2004zeta}
C.~W. Curtis and K.-i. Shinoda.
\newblock Zeta functions and functional equations associated with the
  components of the {G}elfand-{G}raev representations of a finite reductive
  group.
\newblock In {\em Representation theory of algebraic groups and quantum
  groups}, volume~40 of {\em Adv. Stud. Pure Math.}, pages 121--139. Math. Soc.
  Japan, Tokyo, 2004.

\bibitem{erdelyi2021matrix}
M.~Erd{\'e}lyi and {\'A}.~T{\'o}th.
\newblock {M}atrix {K}loosterman sums.
\newblock {\em Algebra Number Theory}, 18(12):2247--2308, 2024.

\bibitem{Gelfand70}
S.~I. Gel'fand.
\newblock Representations of the full linear group over a finite field.
\newblock {\em Math. USSR Sb.}, 12(13):13--39, 1970.

\bibitem{ginzburg2019tensor}
D.~Ginzburg.
\newblock {T}ensor product {$L$}-functions on metaplectic covering groups of
  {${GL}_r$}.
\newblock {\em Proceedings of the American Mathematical Society}.
\newblock In press.

\bibitem{GinzburgSoudry2020}
D.~Ginzburg and D.~Soudry.
\newblock Integrals derived from the doubling method.
\newblock {\em Int. Math. Res. Not. IMRN}, (24):10553--10596, 2020.

\bibitem{GodementJacquet1972}
R.~Godement and H.~Jacquet.
\newblock {\em Zeta functions of simple algebras}.
\newblock Lecture Notes in Mathematics, Vol. 260. Springer-Verlag, Berlin-New
  York, 1972.

\bibitem{GourevitchKaplan2023}
D.~Gourevitch and E.~Kaplan.
\newblock Multiplicity one theorems for the generalized doubling method (with
  an appendix by {A}vraham {A}izenbud and {D}mitry {G}ourevitch).
\newblock {\em J. Eur. Math. Soc. (JEMS)}, 25(8):3007--3092, 2023.

\bibitem{Green55}
J.~A. Green.
\newblock The characters of the finite general linear groups.
\newblock {\em Trans. Amer. Math. Soc.}, 80:402--447, 1955.

\bibitem{GurevichHowe2021}
S.~Gurevich and R.~Howe.
\newblock Harmonic analysis on {$GL_n$} over finite fields.
\newblock {\em Pure Appl. Math. Q.}, 17(4):1387--1463, 2021.

\bibitem{Howe1985}
R.~Howe.
\newblock {\em Harish-{C}handra homomorphisms for {$p$}-adic groups}, volume~59
  of {\em CBMS Regional Conference Series in Mathematics}.
\newblock Conference Board of the Mathematical Sciences, Washington, DC; by the
  American Mathematical Society, Providence, RI, 1985.
\newblock With the collaboration of Allen Moy.

\bibitem{Jacquet1984}
H.~Jacquet.
\newblock On the residual spectrum of {${\rm GL}(n)$}.
\newblock In {\em Lie group representations, {II} ({C}ollege {P}ark, {M}d.,
  1982/1983)}, volume 1041 of {\em Lecture Notes in Math.}, pages 185--208.
  Springer, Berlin, 1984.

\bibitem{Jacquet1983rankin}
H.~Jacquet, I.~I. Piatetskii-Shapiro, and J.~A. Shalika.
\newblock Rankin-{S}elberg convolutions.
\newblock {\em Amer. J. Math.}, 105(2):367--464, 1983.

\bibitem{Kaplan2023}
E.~Kaplan.
\newblock Rankin-{S}elberg integrals and {$L$}-functions for covering groups of
  general linear groups.
\newblock {\em Int. Math. Res. Not. IMRN}, (15):13332--13386, 2023.

\bibitem{katz2016gauss}
N.~M. Katz.
\newblock {\em Gauss sums, {K}loosterman sums, and monodromy groups}, volume
  116 of {\em Annals of Mathematics Studies}.
\newblock Princeton University Press, Princeton, NJ, 1988.

\bibitem{katz1993estimates}
N.~M. Katz.
\newblock Estimates for {S}oto-{A}ndrade sums.
\newblock {\em J. Reine Angew. Math.}, 438:143--161, 1993.

\bibitem{Kondo1963}
T.~Kondo.
\newblock On {G}aussian sums attached to the general linear groups over finite
  fields.
\newblock {\em J. Math. Soc. Japan}, 15:244--255, 1963.

\bibitem{larsen2013largest}
M.~Larsen, G.~Malle, and P.~H. Tiep.
\newblock The largest irreducible representations of simple groups.
\newblock {\em Proc. Lond. Math. Soc. (3)}, 106(1):65--96, 2013.

\bibitem{Macdonald80}
I.~G. Macdonald.
\newblock Zeta functions attached to finite general linear groups.
\newblock {\em Math. Ann.}, 249(1):1--15, 1980.

\bibitem{Nien14}
C.~Nien.
\newblock A proof of the finite field analogue of {J}acquet's conjecture.
\newblock {\em Amer. J. Math.}, 136(3):653--674, 2014.

\bibitem{Nien17}
C.~Nien.
\newblock {$n\times 1$} local gamma factors and {G}auss sums.
\newblock {\em Finite Fields Appl.}, 46:255--270, 2017.

\bibitem{Roditty10}
E.-A. Roditty-{G}ershon.
\newblock On {G}amma factors and {B}essel functions for representations of
  general linear groups over finite fields.
\newblock Master's thesis, Tel Aviv University, 2010.
\newblock
  \href{https://tau.primo.exlibrisgroup.com/discovery/delivery/972TAU_INST:TAU/12426764880004146}{DaTA:9932948793904146}.

\bibitem{SayagVerma2020}
E.~Sayag and M.~K. Verma.
\newblock On disjointness of linear models and degenerate {W}hittaker models.
\newblock {\em J. Number Theory}, 207:56--82, 2020.

\bibitem{shahidi1984fourier}
F.~Shahidi.
\newblock Fourier transforms of intertwining operators and {P}lancherel
  measures for {${\rm GL}(n)$}.
\newblock {\em Amer. J. Math.}, 106(1):67--111, 1984.

\bibitem{shahidi1990proof}
F.~Shahidi.
\newblock A proof of {L}anglands' conjecture on {P}lancherel measures;
  complementary series for {$p$}-adic groups.
\newblock {\em Ann. of Math. (2)}, 132(2):273--330, 1990.

\bibitem{SilbergerZink00}
A.~J. Silberger and E.-W. Zink.
\newblock The characters of the generalized {S}teinberg representations of
  finite general linear groups on the regular elliptic set.
\newblock {\em Trans. Amer. Math. Soc.}, 352(7):3339--3356, 2000.

\bibitem{soudry1979}
D.~Soudry.
\newblock On gamma functions of pairs over finite fields.
\newblock 1979.

\bibitem{SoudryZelingher2023}
D.~Soudry and E.~Zelingher.
\newblock On gamma factors for representations of finite general linear groups.
\newblock {\em Essent. Number Theory}, 2(1):45--82, 2023.

\bibitem{Ye18}
R.~Ye.
\newblock Rankin-{S}elberg gamma factors of level zero representations of
  {${\rm GL}_n$}.
\newblock {\em Forum Math.}, 31(2):503--516, 2019.

\bibitem{YeZeligher18}
R.~Ye and E.~Zelingher.
\newblock Exterior square gamma factors for cuspidal representations of {${\rm
  GL}_n$}: finite field analogs and level-zero representations.
\newblock {\em Israel J. Math.}, 240(2):889--934, 2020.

\bibitem{ye2021epsilon}
R.~Ye and E.~Zelingher.
\newblock Epsilon factors of representations of finite general linear groups.
\newblock {\em J. Number Theory}, 221:122--142, 2021.

\bibitem{zelingher2022values}
E.~Zelingher.
\newblock On values of the {B}essel function for generic representations of
  finite general linear groups.
\newblock {\em Adv. Math.}, 434:Paper No. 109314, 47, 2023.

\bibitem{Zelingher2024b}
E.~Zelingher.
\newblock On {S}peh representations for level zero supercuspidal
  representations and {G}inzburg--{K}aplan gamma factors.
\newblock 2024.
\newblock \href{https://arxiv.org/abs/2408.01856}{arXiv:2408.01856}.

\bibitem{Zelingher2025}
E.~Zelingher.
\newblock On exotic matrix exponential sums and {B}essel--{S}peh functions.
\newblock 2025.
\newblock \href{https://arxiv.org/abs/2507.06394}{arXiv:2507.06394}.

\bibitem{Zelingher2023}
E.~Zelingher.
\newblock On matrix {K}loosterman sums and {H}all-{L}ittlewood polynomials.
\newblock {\em Trans. Amer. Math. Soc.}, 378(5):3597--3623, 2025.

\end{thebibliography}
\end{document}